\documentclass[11pt,a4paper,english]{article}

\usepackage{amsfonts,amssymb,stmaryrd,amsmath,amsthm,amssymb}
\usepackage{graphicx,color}
\usepackage{multimedia}
\usepackage{float}
\usepackage{placeins}

\usepackage[a4paper]{geometry}
\newtheorem{theorem}{Theorem}
\newtheorem{lemma}{Lemma}
\newtheorem{remark}{Remark}
\newtheorem{proposition}{Proposition}
\newtheorem{corollary}{Corollary}

\newcommand \p {\partial}
\newcommand \dist {\mathrm{dist}}
\newcommand \com {\mathrm{cof}}

\newcommand \w {\mathrm{w}}
\newcommand \R {\mathbb{R}}

\renewcommand \L {\mathrm{L}}

\newcommand \LL {\mathbf{L}}
\newcommand \MM {\mathbf{M}}
\newcommand \NN {\mathbf{N}}
\newcommand \GG {\mathbf{G}}
\renewcommand \H {\mathrm{H}}

\newcommand \I {\mathrm{I}}
\newcommand \Id {\mathrm{Id}}

\renewcommand \d {\mathrm{d}}

\renewcommand \div {\mathrm{div}}
\renewcommand \det {\mathrm{det}}

\newcommand \trace {\mathrm{trace}}

\title{Existence of 3D strong solutions for a system modeling a deformable solid inside a viscous incompressible fluid\thanks{This work was partially supported by the ANR-project CISIFS, 09-BLAN-0213-03.}}

\author{S\'ebastien Court\thanks{Laboratoire de Math\'ematiques UMR CNRS 6620, Universit\'e Blaise Pascal, Campus des C\'ezeaux, F-63177 Aubi\`ere Cedex, France, email: {\tt sebastien.court@math.univ-bpclermont.fr}.}}

\begin{document}

\maketitle

\begin{abstract}
In this paper we study a coupled system modeling the movement of a deformable solid inside a viscous incompressible fluid. For the solid we consider a given deformation that has to obey several physical constraints. The motion of the fluid is modeled by the incompressible Navier-Stokes equations in a time-dependent bounded domain of $\R^3$, and the solid satisfies the Newton's laws. Our contribution consists in adapting and completing some results of \cite{SMSTT} in dimension 3, in a framework where the regularity of the deformation of the solid is limited. We rewrite the main system in domains which do not depend on time, by using a new means of defining a change of variables, and a suitable change of unknowns. We study the corresponding linearized system before setting a local-in-time existence result. Global existence is obtained for small data, and in particular for deformations of the solid which are close to the identity.
\end{abstract}

\hfill 

\tableofcontents




\section{Introduction} 
In this paper we are interested in a deformable solid immersed in a viscous incompressible fluid in dimension 3. The domain occupied by the solid at time $t$ is denoted by $\mathcal{S}(t)$. We assume that \textcolor{black}{$\mathcal{S}(t) \subset \subset \mathcal{O}$}, where $\mathcal{O}$ is a bounded regular domain. The fluid surrounding the solid occupies the domain $\mathcal{O} \setminus \overline{\mathcal{S}(t)} = \mathcal{F}(t)$.

\subsection{Presentation of the model} 
The movement of the solid in the inertial frame of reference is described through the time by a Lagrangian mapping $X_{\mathcal{S}}$, so we define
\begin{eqnarray*}
\mathcal{S}(t) & = & X_{\mathcal{S}}(\mathcal{S}(0),t), \quad t \geq 0.
\end{eqnarray*}
The mapping $X_{\mathcal{S}}(\cdot,t)$ can be decomposed as follows
\begin{eqnarray*}
X_{\mathcal{S}}(y,t) & = & h(t) + \mathbf{R}(t)X^{\ast}(y,t), \quad y \in \mathcal{S}(0),
\end{eqnarray*}
where the vector $h(t)$ describes the position of the center of mass, and $\mathbf{R}(t)$ is the rotation associated with \textcolor{black}{a vector $\omega$(t) denoting the angular velocity of the solid}. More precisely, $\omega$ and $\mathbf{R}$ are related to each other through the following Cauchy problem
\begin{eqnarray*}
\left\{\begin{array} {ccccc}
\displaystyle \frac{\d \mathbf{R}}{\d t} & = & \mathbb{S}\left( \omega\right) \mathbf{R} \\
\mathbf{R}(0) & = & \I_{\R^3}
\end{array} \right.,
\quad \text{with }
\mathbb{S}(\omega) = \left(
\begin{matrix}
0 & -\omega_3 & \omega_2 \\
\omega_3 & 0 & -\omega_1 \\
-\omega_2 & \omega_1 & 0
\end{matrix} \right). \label{rotation}
\end{eqnarray*}
\footnote{\textcolor{black}{The identity matrix of $\R^{3\times 3}$ is denoted by $\I_{\R^3}$.}}The couple $(h(t),\mathbf{R}(t))$ describes the position of the solid and is unknown, whereas the mapping $X^{\ast}(\cdot,t)$ can be imposed. This latter represents the deformation of the solid in its own frame of reference and will constitute the main datum of the problem. When this Lagrangian mapping $X^{\ast}(\cdot,t)$ is invertible, we can link to it an Eulerian velocity $w^{\ast}$ through the following Cauchy problem
\begin{eqnarray*}
\frac{\p X^{\ast}}{\p t}(y,t) = w^{\ast}(X^{\ast}(y,t),t), \quad X^{\ast}(y,0) = y-h(0), \quad y\in \mathcal{S}(0). \label{cauchystar}
\end{eqnarray*}
If $Y^{\ast}(\cdot,t)$ denotes the inverse of $X^{\ast}(\cdot,t)$, we have
\begin{eqnarray*}
 w^{\ast}(x^{\ast},t) & = & \frac{\p X^{\ast}}{\p
t}(Y^{\ast}(x^{\ast},t),t), \quad x^{\ast} \in \mathcal{S}^{\ast}(t)=X^{\ast}(\mathcal{S}(0),t). \label{wsuperstar}
\end{eqnarray*}
This Eulerian velocity $w^{\ast}$ can also be considered as a datum defining the way the solid is deforming itself. Considering $X^{\ast}$ - or $w^{\ast}$ - as a datum is equivalent to assuming that the solid is strong enough to impose its own shape.\\
The fluid flow is described by its velocity $u$ and its pressure $p$. For $w^{\ast}$ satisfying a set of hypotheses given further, we aim at proving the existence of strong solutions for the following coupled system
\begin{eqnarray}
\frac{\p u}{\p t} - \nu \Delta u + (u\cdot \nabla )u + \nabla p  =  0 , & \quad &
x \in \mathcal{F} (t),\quad t\in (0,T), \label{prems} \\
\div \ u   =  0 ,  & \quad &  x\in \mathcal{F} (t), \quad t\in (0,T), \label{deus}
\end{eqnarray}
\begin{eqnarray}
u  =  0 , & \quad & x \in \p \mathcal{O} ,\quad t\in (0,T),  \label{trois} \\
u  =  h'(t) +  \omega (t)\wedge(x-h(t)) + w(x,t) , & \quad &  x\in \ \p \mathcal{S}(t),\quad t\in
(0,T), \label{quatre}
\end{eqnarray}
\begin{eqnarray}
M h''(t)  =  - \int_{\p \mathcal{S}(t)} \sigma(u,p) n \d \Gamma  , & \quad & t\in (0,T), \label{cinq} \\
\left(I\omega\right)' (t)  = - \int_{\p \mathcal{S}(t)} (x-h(t))\wedge  \sigma(u,p) n \d \Gamma, & \quad & t\in(0,T),
\label{six}
\end{eqnarray}
\begin{eqnarray}
u(y,0)  =  u_0 (y), \  y\in \mathcal{F}(0) , \quad h(0) = h_0 \in \R^3 ,\quad h'(0)=h_1 \in \R^3 ,\quad \omega(0) = \omega_0 \in \R^3 ,
\end{eqnarray}
where
\begin{eqnarray}
\mathcal{S}(t) = h(t) + \mathbf{R}(t)X^{\ast}(\mathcal{S}(0),t), & \quad & \mathcal{F}(t) = \mathcal{O} \setminus \overline{\mathcal{S}(t)},
\end{eqnarray}
and where the velocity $w$ is defined by the following change of frame
\begin{eqnarray}
w(x,t) & = & \mathbf{R}(t)\ w^{\ast}\left(\mathbf{R}(t)^{T}(x-h(t)), t\right), \quad x\in \mathcal{S}(t). \label{wwwstar} \label{ders}
\end{eqnarray}
Without loss of generality, we can assume that $h_0 = 0$, for a sake of simplicity. The symbol $\wedge$ denotes the cross product in $\R^3$. The linear map $\omega \wedge \cdot$ can be represented by the matrix $\mathbb{S}(\omega)$. In equations \eqref{cinq} and \eqref{six}, the mass of the solid $M$ is constant, whereas the moment of inertia tensor depends on time, as
\begin{eqnarray*}
I(t) & = & \int_{\mathcal{S}(t)} \rho_{\mathcal{S}}(x,t) \left( |x-h(t)|^2 \I_{\R^3} -
(x-h(t))\otimes (x-h(t)) \right) \d x.
\end{eqnarray*}
The quantity $\rho_{\mathcal{S}}$ denotes the density of the solid, and obeys the principle of mass conservation
\begin{eqnarray*}
\rho_{\mathcal{S}}(X_{\mathcal{S}}(y,t),t) & = &
\frac{\rho_{\mathcal{S}}(y,0)}{\det \left(\nabla X_{\mathcal{S}}(y,t)\right)}, \quad y \in \mathcal{S}(0),
\end{eqnarray*}
where $\nabla X_{\mathcal{S}}(\cdot,t)$ is the Jacobian matrix of mapping $X_{\mathcal{S}}(\cdot,t)$. We can define
\begin{eqnarray*}
\rho^{\ast}(x^{\ast},t) & = & \frac{\rho_{\mathcal{S}}(Y^{\ast}(x^{\ast},t),0)}{\det \left(\nabla X^{\ast}(Y^{\ast}(x^{\ast},t),t)\right)}, \quad x^{\ast} \in \mathcal{S}^{\ast}(t).
\end{eqnarray*}
For a sake of simplicity we assume that the solid is homogeneous at time $t=0$:
\begin{eqnarray*}
\rho_{\mathcal{S}}(y,0) & = & \rho_{\mathcal{S}} > 0.
\end{eqnarray*}
In system \eqref{prems}--\eqref{ders}, $\nu$ is the kinematic viscosity of the fluid and the normalized vector $n$ is the normal at $\p \mathcal{S}(t)$ exterior to $\mathcal{F}(t)$. It is a coupled system between the incompressible Navier-Stokes equations and the Newton's laws. The coupling is in particular made in the fluid-structure interface, through the equality of velocities \eqref{quatre} and through the Cauchy stress tensor given by
\begin{eqnarray*}
\sigma(u,p)  =  2\nu D(u) - p \ \Id
  =  \nu\left( \nabla u + \left(\nabla u\right)^{T} \right) - p \ \Id.
\end{eqnarray*}
We assume that the deformation $X^{\ast}$ satisfies a set of hypotheses:
\begin{description}
\item[H1] For all $t\in [0,T]$, $X^{\ast}(\cdot,t)$ is a \textcolor{black}{$C^{1}$-diffeomorphism} from $\overline{\mathcal{S}(0)}$ onto $\overline{\mathcal{S}^{\ast}(t)}$.
\item[H2] In order to respect the incompressibility condition given by \eqref{deus}, the volume of the whole solid is preserved through the time. That is equivalent to say that
\begin{eqnarray}
\int_{\p \mathcal{S}^{\ast}(t)} w^{\ast}\cdot n \d \Gamma = \int_{\p \mathcal{S}(0)} \frac{\p X^{\ast}}{\p t}\cdot \left(\com \nabla X^{\ast}\right)n\d \Gamma =  0, \label{const1}
\end{eqnarray}
where $\com \nabla X^{\ast}$ denotes the cofactor matrix of $\nabla X^{\ast}$.
\item[H3] \textcolor{black}{The deformation of the solid does not modify its linear momentum}, that means
\begin{eqnarray}
\int_{\mathcal{S}^{\ast}(t)} \rho^{\ast}(x^{\ast},t) w^{\ast}(x^{\ast},t)\d x^{\ast} =
\rho_{\mathcal{S}} \int_{\mathcal{S}(0)}  \frac{\p X^{\ast}}{\p t}(y,t) \d y  =  0. \label{const2}
\end{eqnarray}
\item[H4] \textcolor{black}{The deformation of the solid does not modify its angular momentum}, that means
\begin{eqnarray}
\int_{\mathcal{S}^{\ast}(t)} \rho^{\ast}(x^{\ast},t) x^{\ast} \wedge w^{\ast}(x^{\ast},t)\d x^{\ast} = \rho_{\mathcal{S}} \int_{\mathcal{S}(0)} X^{\ast}(y,t)\wedge \frac{\p X^{\ast}}{\p t}(y,t) \d y =  0. \label{const3}
\end{eqnarray}
\end{description}
Imposing constraints \eqref{const2} and \eqref{const3} enables us to ensure that the two following constraints on the velocity $w$ are satisfied:
\begin{eqnarray}
\int_{\mathcal{S}(t)} \rho_{\mathcal{S}}(x,t)w(x,t)  \d y & = & 0, \label{const12} \\
\int_{\mathcal{S}(t)} \rho_{\mathcal{S}}(x,t)(x-h(t))\wedge w(x,t)  \d y & = & 0. \label{const13}
\end{eqnarray}
As equations \eqref{cinq} and \eqref{six} are written, the two equalities above are already assumed in system \eqref{prems}--\eqref{ders}. \textcolor{black}{Without the hypotheses {\bf H3} and {\bf H4}, the balance of momenta would lead to expressions of \eqref{cinq} and \eqref{six} with additional terms involving the quantities of \eqref{const12} and \eqref{const13}}. These hypotheses are made to guarantee the {\it self-propelled} nature of the motion of the solid, that means no other help than its own deformation enables it to move inside the fluid. By the undulatory motion induced by its own internal deformation, the solid imposes partially, through $w$, the nonhomogeneous Dirichlet condition \eqref{quatre}. The latter induces the behavior of the environing fluid through \eqref{prems}--\eqref{trois}, and thus the response of the fluid - given by $\sigma(u,p)n$ on the interface - enables the whole solid to be carried, regarding to the ordinary differential equations \eqref{cinq} and \eqref{six}. The other part of the interaction consists in the fact that domains occupied by the fluid and the solid change through the time, as follows
\begin{eqnarray*}
\mathcal{S}(t) = h(t) + \mathbf{R}(t)\mathcal{S}^{\ast}(t), & \quad & \mathcal{F}(t) = \mathcal{O} \setminus \overline{\mathcal{S}(t)}.
\end{eqnarray*}

\subsection{Main result and contributions}
The main result we state in this paper is Theorem \ref{ch3_existglobale}, that we give as follows:

\begin{theorem} \label{mainTH}
\textcolor{black}{Assume that $0<  \dist(\mathcal{S},\p\mathcal{O})$, that $X^{\ast}$ satisfies the hypotheses $\mathbf{H1}-\mathbf{H4}$ with the initial conditions\footnote{\textcolor{black}{The first condition is natural, but the second one can be replaced by a non-null velocity. However, in order to avoid additional complexities, we choose it homogeneous, without loss of generality.}}
\begin{eqnarray*}
X^{\ast}(\cdot,0) = \Id_{\mathcal{S}}, & \quad & \frac{\p X^{\ast}}{\p t}(\cdot, 0) = 0,
\end{eqnarray*}
and that the deformation velocity $\frac{\p X^{\ast}}{\p t}$ is small enough in 
\begin{eqnarray*}
\L^2(0,\infty;\mathbf{H}^3(\mathcal{S}(0))) \cap \H^1(0,\infty;\mathbf{H}^1(\mathcal{S}(0))).
\end{eqnarray*}
Assume that $u_0 \in \mathbf{H}^1(\mathcal{F}(0))$ satisfies
\begin{eqnarray*}
\div \  u_0  =  0   \textrm{ in $\mathcal{F}(0)$}, \quad
u_0  =  0  \textrm{ on $\p \mathcal{O}$}, \quad
u_0(y)  =  h_1 + \omega_0 \wedge y  \textrm{ on $\p \mathcal{S}(0)$},
\end{eqnarray*}
and also that $\|u_0\|_{\mathbf{H}^1(\mathcal{F}(0))}$, $|h_1|_{\R^3}$ and $|\omega_0|_{\R^3}$ are small enough. Then problem \eqref{prems}--\eqref{ders} admits a unique global strong solution $(u,p,h',\omega)$ such that
\begin{eqnarray*}
& & \int_0^{\infty} \|u(\cdot,t)\|_{\mathbf{H}^2(\mathcal{F}(t))}^2 \d t
+ \int_0^{\infty} \left\|\frac{\p u}{\p t}(\cdot,t)\right\|_{\mathbf{L}^2(\mathcal{F}(t))}^2 \d t
+ \sup_{t\geq 0} \left\| u(\cdot,t) \right\|_{\mathbf{H}^1(\mathcal{F}(t))}^2 < \infty, \\
& & \int_0^{\infty} \|p(\cdot,t)\|_{\mathbf{H}^1(\mathcal{F}(t))}^2 \d t < \infty,
\quad \|h'\|_{\H^1(0,\infty;\R^3)} < \infty, \quad  \|\omega\|_{\H^1(0,\infty;\R^3)} < \infty
\end{eqnarray*}
if the condition $\dist(\mathcal{S}(t), \p \mathcal{O}) > 0$ is satisfied for all $t \geq 0$.}
\end{theorem}

This type of problem has been studied in \cite{SMSTT} in 2 dimensions, in the case where no limitation was supposed on the regularity of the mapping $X^{\ast}$. In particular global existence is obtained without smallness assumption on the data. We extend this result to dimension 3 in a framework where the regularity of the mapping $X^{\ast}$ is limited, which has not been done yet for this system, as far as we know. The strategy for proving the existence of strong solutions is globally the same as the one used in \cite{TT, Cumsille} (for rigid solids), \cite{SMSTT}, or even in  \cite{BoulST} for instance: We first define a change of variables which enables us to set a change of unknowns whose the space domain of definition does not depend on time anymore. Then we write the nonlinear system that have to satisfy the new unknowns, and we study the linearized system associated with. Then a local-in-time existence result is proven by a fixed point method, and the global existence is obtained by writing appropriate energy estimates.\\
\hfill \\
In addition to the technical difficulties induced by the framework of dimension 3, the originality of our approach lies in the fact that we have to develop new means of handling a deformation of the solid which is limited in regularity. First, a new method is developed in order to define a change of variables in the fluid part. Indeed, the method introduced in \cite{TT} cannot be applied anymore, or at least not so straightforwardly anymore; The way we proceed is more direct and more adapted for obtaining the change of variables with the desired properties, moreover when the datum $X^{\ast}$ is limited in regularity. The price to pay is a technical lemma proven in Appendix A. \\

Thus we extend the Lagrangian flow $X_{\mathcal{S}}(\cdot ,t)$ associated with the solid as a mapping $X(\cdot,t)$ defined in the fluid part. We denote by $Y(\cdot,t)$ the inverse of $X(\cdot,t)$. Then we rewrite system \eqref{prems}--\eqref{ders} in a cylindrical domain. For that, another novelty is the use of a well-chosen change of unknowns; We introduce the following unknowns
\begin{eqnarray*}
\tilde{u}(y,t)  =  \mathbf{R}(t)^T u(X(y,t),t), &\quad & \tilde{p}(y,t)  =  p(X(y,t),t) \label{chv1},
\end{eqnarray*}
rather than using the whole Jacobian matrix
\begin{eqnarray}
\overline{u}(y,t)  =  \nabla Y(X(y,t),t)u(X(y,t),t),& \quad & \overline{p}(y,t)  =  p(X(y,t),t) , \label{chv2}
\end{eqnarray}
which is done in \cite{IW} for instance, or in several papers which only consider a rigid solid (see \cite{TT}, \cite{Tucsnak}, \cite{Cumsille} for instance), or simply suggested in \cite{SMSTT}. Let us notice that in our case the Jacobian matrix $\nabla Y(X(\cdot,t),t)$ actually depends on the space variable \textcolor{black}{near of $\p \mathcal{S}(0)$}, and thus the use of this classical change of unknowns \eqref{chv2} would lead to unappropriate complicated calculations and especially it would require more regularity than we actually need for the deformation of the solid.\\
\hfill \\
The corresponding nonlinear system - satisfied by the new unknowns, written in a cylindrical domain - is stated in \eqref{premsfix}--\eqref{dersfix}. The change of variables we have chosen enables us to write this system in the simplest form we have found. In particular, the equation of velocities \eqref{quatre} on $\p \mathcal{S}(t)$ becomes \eqref{quatrefix}
\begin{eqnarray*}
\tilde{u}  =  \tilde{h}' + \tilde{\omega} \wedge X^{\ast} + \frac{\p X^{\ast}}{\p t} \quad  \text{on } \p \mathcal{S}(0),
\end{eqnarray*}
where the datum $X^{\ast}$ and its time derivative appear in a simple way. The price to pay is that we have to study a system in which the divergence of $\tilde{u}$ is not equal to $0$.\\ 
The proof of the existence of local-in-time strong solutions is similar to the one provided in \cite{TT}. For proving that the solution so obtained is actually global in time, we show that our framework enables us to apply the techniques developed in \cite{Cumsille}; In particular, we get regularity on the Eulerian velocity $w^{\ast}$ associated with the Lagrangian mapping $X^{\ast}$, and we consider an extension of $w^{\ast}$ to the fluid part. We also quantify the regularity needed on this Eulerian velocity, and we observe that the regularity assumed on $X^{\ast}$ is sufficient.

\paragraph{The choice of the functional framework for the deformation of the solid}\hfill \\
The mapping $X^{\ast}$ is chosen such that its time derivative (representing a velocity of deformation) lies in $\L^2(0,\infty;\mathbf{H}^3(\mathcal{S}(0))) \cap \H^1(0,\infty;\mathbf{H}^1(\mathcal{S}(0)))$. The regularity $\mathbf{H}^3$ in space is considered in order to make $X^{\ast}(\cdot,t)$ and its extensions of class $C^1$, \textcolor{black}{and thus likely to be used for transforming integrals on $\p \mathcal{S}(t)$. Note that under this regularity the hypothesis {\bf H1} can be relaxed; Indeed, since in the statement of Theorem \ref{mainTH} we assume smallness on the time derivative of $X^{\ast}$, we can consider that $X^{\ast} - \Id_{\mathcal{S}}$ stays close to $0$, and thus that $X^{\ast}(\cdot,t)$ defines a $C^1$-diffeomorphism}. Besides, the way we treat the nonhomogeneous divergence condition (in the proof of local strong solutions) requires such a regularity, in space as well as in time. Moreover, the estimates we obtain in the proof of global existence (see Proposition \ref{ch2_zprop1}) require an Eulerian velocity $w^{\ast}$ whose the regularity - roughly speaking - corresponds with the one chosen for the deformation velocity $\frac{\p X^{\ast}}{\p t}$ (see Lemma \ref{Xwstar}).\\
\hfill \\
Let us quote other works which treat of systems coupling the Navier-Stokes equations with some deformable structure, like the mathematical analysis of the interactions between a Navier-Stokes fluid and an elastic or viscoelastic structure: \cite{Boulakia}, \cite{Desjardins1}, \cite{Chambolle}, \cite{Coutand1}, \cite{Coutand2}. For the fluid-solid system we consider in the present paper, the case of weak solutions (in 3 dimensions) has been recently investigated in \cite{Necasova}. Our approach looks like a recent work of \cite{BoulST} in which the authors consider an elastic structure whose the regularity of its deformation is limited. The interest of considering deformations of the solid restricted in regularity lies especially in the perspective of a work where the deformation of the solid would be considered as a control function.

\subsection{Plan}
In section \ref{secdef} we bring precisions to the functional framework, for the unknowns written in time-depending domains and for the datum $X^{\ast}$ representing the deformation of the solid. In section \ref{secchange} we extend the flow of the solid to the fluid part; It enables us to set a change of unknowns and to write the nonlinear system that has to satisfy the new unknowns. The linearized system associated with is studied in section \ref{seclinear}. In particular Proposition \ref{ch2_thsolfaible} is used in the next section \ref{secfixe} in order to define a mapping whose a fixed point is a strong solution of the nonlinear system. We then prove that for small time this mapping is a contraction in a ball chosen large enough. Section \ref{secglobale} is devoted to prove the main result, that is to say that the local strong solution can be actually global if we assume that the data are small enough. Finally, technical lemmas used before that are proven in Appendixes A and B.

\section{Definitions, notation \textcolor{black}{and basic properties}} \label{secdef} 
We denote by
\begin{eqnarray*}
\mathcal{F} = \mathcal{F}(0) & \quad \text{ and } \quad & \mathcal{S}= \mathcal{S}(0)
\end{eqnarray*}
the domains occupied at time $t=0$ by the fluid and the solid respectively. We assume that $\mathcal{S}$ is simply connected and regular enough. We also assume that $\mathcal{O}$ is regular enough. Note that the boundary of $\mathcal{F}$ is equal to $\p \mathcal{O} \cup \p \mathcal{S}$. We set for all $t \geq 0$
\begin{eqnarray*}
\mathcal{S}^{\ast}(t) = X^{\ast}(\mathcal{S},t), & \quad & \tilde{\mathcal{F}}(t) = \mathcal{O} \setminus \overline{\mathcal{S}^{\ast}(t)}.
\end{eqnarray*}
Let be $T \in [0, + \infty]$. We set
\begin{eqnarray*}
S^0_T  =   \mathcal{S} \times (0,T), & \quad &  Q_T^0  =  \mathcal{F} \times (0,T),
\end{eqnarray*}
and
\begin{eqnarray*}
Q_T & = & \displaystyle \bigcup_{t\in (0,T)} \mathcal{F}(t) \times \{t\}.
\end{eqnarray*}
In order to deal with some Sobolev functional spaces, we use the notation
\begin{eqnarray*}
\mathbf{L}^2(\Omega) = [\L^2(\Omega)]^3, & \quad & \mathbf{H}^s(\Omega) = [\H^s(\Omega)]^3,
\end{eqnarray*}
for all domain $\Omega \subset \mathcal{O}$. Nevertheless this type of notation will be also used for other multidimensional spaces (for tensors) like $[\L^2(\Omega)]^{3\times 3}$, $[\L^2(\Omega)]^{3\times 3\times 3}$, $[\L^2(\Omega)]^{3\times 3\times 3\times 3}$, $[\H^s(\Omega)]^{3\times 3}$, $[\H^s(\Omega)]^{3\times 3\times 3}$ or $[\H^s(\Omega)]^{3\times 3\times 3\times 3}$, without ambiguity. Let us now make precise the functional spaces that we will consider in order to look for strong solutions to Problem \eqref{prems}--\eqref{ders}.\\

\subsection{Functional setting for the unknowns}
\textcolor{black}{Let be $T\in [0,+\infty]$. Let us consider a family of time-depending domains $\left(\mathcal{F}(t)\right)_{t\geq 0}$, for instance the one generated by $h(t)$ and $\mathbf{R}(t)$, as described below
\begin{eqnarray*}
\mathcal{S}(t) = h(t) + \mathbf{R}(t) \mathcal{S}(0), & \quad & \mathcal{F}(t) = \mathcal{O} \setminus \overline{\mathcal{S}(t)}.
\end{eqnarray*}
Let us first define the space
\begin{eqnarray*}
\mathcal{U}(0,T;\mathcal{F}) := \L^2(0,T;\mathbf{H}^2( \mathcal{F})) \cap 
\H^1(0,T;\mathbf{L}^2( \mathcal{F})) \cap C(0,T;\mathbf{H}^1( \mathcal{F}))
\end{eqnarray*}
}\footnote{\textcolor{black}{Note that we have the embedding $\L^2(0,T;\mathbf{H}^2( \mathcal{F})) \cap \H^1(0,T;\mathbf{L}^2( \mathcal{F})) \hookrightarrow  C([0,T];\mathbf{H}^1(\mathcal{F}))$, but we will need to control the quantity $\L^{\infty}(0,T;\mathbf{H}^1(\mathcal{F}))$ independently of $T$, when $T$ will be considered close to $0$.}}\textcolor{black}{that we endow with the norm given by
\begin{eqnarray*}
 \| \tilde{u}\|^2_{\mathcal{U}(0,T; \mathcal{F})}  :=   \int_0^T \|\tilde{u}(\cdot,t)\|_{\mathbf{H}^2(\mathcal{F})}^2 \d t
+ \int_0^T \left\|\frac{\p \tilde{u}}{\p t}(\cdot,t)\right\|_{\mathbf{L}^2(\mathcal{F})}^2 \d t
+ \sup_{t\in [0,T]} \left\| \tilde{u}(\cdot,t) \right\|_{\mathbf{H}^1(\mathcal{F})}^2.
\end{eqnarray*}
} 
The velocity $u$ will be searched \textcolor{black}{in the space $\mathcal{U}(0,T; \mathcal{F}(t))$ that we endow and define with the norm given by}
\begin{eqnarray*}
\|u\|^2_{\mathcal{U}(0,T; \mathcal{F}(t))} & := &
\int_0^T \|u(\cdot,t)\|_{\mathbf{H}^2(\mathcal{F}(t))}^2 \d t
+ \int_0^T \left\|\frac{\p u}{\p t}(\cdot,t)\right\|_{\mathbf{L}^2(\mathcal{F}(t))}^2 \d t
+ \sup_{t\in [0,T]} \left\| u(\cdot,t) \right\|_{\mathbf{H}^1(\mathcal{F}(t))}^2.
\end{eqnarray*}
Analogously we can define spaces of type $\H^{s_1}(0,T;\H^{s_2}(\Omega(t)))$ and $\H^{s_1}(0,T;\mathbf{H}^{s_2}(\Omega(t)))$ for all family of time-depending domains $(\Omega(t))_{t\geq 0}$, where $s_1$ and $s_2$ are non-negative integers, \textcolor{black}{by using the norms of the following type
\begin{eqnarray*}
\|u\|^2_{\H^{s_1}(0,T;\mathbf{H}^{s_2}(\Omega(t)))}
& := & 
\sum_{k=0}^{s_1}\int_0^T \left\|\frac{\p^k u}{\p t^k}(\cdot,t)\right\|_{\mathbf{H}^{s_{2}}(\Omega(t))}^2 \d t.
\end{eqnarray*}
}

\begin{remark}
The definitions we give here for Sobolev spaces dealing with time-depending domains are not the same as the ones given in \cite{TT} or \cite{SMSTT} for instance. If the mapping $X_{\mathcal{S}}$ is smooth enough, as in \cite{SMSTT}, we claim that these definitions are equivalent to ours, in the sense that the spaces they define are the same. But this is not so obvious when $X_{\mathcal{S}}$ is limited in regularity, like in our case.
\end{remark}

The pressure $p$ will be searched in $\L^2(0,T;\H^1(\mathcal{F}(t)))$; At each time $t$ it is determined up to a constant that we fix such that $\int_{\mathcal{F}(t)} p = 0$. Thus in particular from the Poincar\'e-Wirtinger inequality the pressures $P$ defined in $\mathcal{F}$ can be estimated in $\H^1(\mathcal{F})$ as follows\footnote{In the following the symbols $C$ and $\tilde{C}$ will denote some generic positive constants independent of time, data, or the unknowns.}
\begin{eqnarray*}
\| P \|_{\H^1(\mathcal{F})} & \leq & C \| \nabla P \|_{\L^2(\mathcal{F})}.
\end{eqnarray*}
The same estimate will be considered for other functions which play the role of a pressure in $\mathcal{F}(0)$. 

\subsection{Functional setting for changes of variables}
We consider deformations of the solid $X^{\ast}$ whose the displacement $X^{\ast} - \Id_{\mathcal{S}}$ associated with lies in the space $\mathcal{W}_0(0,T;\mathcal{S})$ that we define as follows
\begin{eqnarray*}
X^{\ast} - \Id_{\mathcal{S}} \in \mathcal{W}_0(0,T;\mathcal{S}) \Leftrightarrow \left\{ \begin{array} {lll}
\displaystyle  \frac{\p X^{\ast}}{\p t} \in
\L^2(0,T;\mathbf{H}^3( \mathcal{S})) \cap \H^{1}(0,T;\mathbf{H}^1(\mathcal{S}))
\hfill \\
X^{\ast}(y,0) = y, \ \displaystyle \frac{\p X^{\ast}}{\p t}(y , 0) = 0 \quad \forall y\in \mathcal{S}.
\end{array} \right.
\end{eqnarray*}
We endow it with the norm
\begin{eqnarray*}
\| X^{\ast} - \Id_{\mathcal{S}} \|_{\mathcal{W}_0(0,T;\mathcal{S})} 
& : = &
\left\| \frac{\p X^{\ast}}{\p t} \right\|_{\L^2(0,T;\mathbf{H}^3( \mathcal{S})) \cap \H^{1}(0,T;\mathbf{H}^1(\mathcal{S}))}.
\end{eqnarray*}
Notice that for $T < \infty$ the following embedding holds
\begin{eqnarray*}
\mathcal{W}_0(0,T;\mathcal{S}) & \hookrightarrow & \H^1(0,T;\mathbf{H}^3( \mathcal{S})) \cap \H^{2}(0,T;\mathbf{H}^1( \mathcal{S})).
\end{eqnarray*}
Thus, for more clarity, we set
\begin{eqnarray*}
\mathcal{W}(S_T^0) & = & \H^1(0,T;\mathbf{H}^3( \mathcal{S})) \cap \H^{2}(0,T;\mathbf{H}^1( \mathcal{S})), \\
\mathcal{W}(Q_T^0) & = & \H^1(0,T;\mathbf{H}^3( \mathcal{F})) \cap \H^{2}(0,T;\mathbf{H}^1( \mathcal{F})), \\
\mathcal{H}(S_T^0) & = & \L^2(0,T;\mathbf{H}^3( \mathcal{S})) \cap \H^{1}(0,T;\mathbf{H}^1( \mathcal{S})), \\
\mathcal{H}(Q_T^0) & = & \L^2(0,T;\mathbf{H}^3( \mathcal{F})) \cap \H^{1}(0,T;\mathbf{H}^1( \mathcal{F})).
\end{eqnarray*}

\textcolor{black}{
\subsection{Basic estimates}
Let us remind basic embedding estimates for functions which vanish at time $t=0$. Specifying the way the constants of these estimates depend on the existence time $T$ is important, in view of the methods used for proving the main result. In particular, the fact that they are non-decreasing with respect to $T$ is essential. First, given some Banach space $\mathrm{B}$, we have for all $f\in \H^1_0(0,T;\mathrm{B})$:
\begin{eqnarray*}
\| f \|_{\L^{\infty}(0,T;\mathrm{B})} & \leq & \sqrt{T} \left\| f' \right\|_{\L^{2}(0,T;\mathrm{B})}.
\end{eqnarray*}
On the other hand we have always:
\begin{eqnarray*}
\| f \|_{\L^{2}(0,T;\mathrm{B})} & \leq & \sqrt{T} \| f \|_{\L^{\infty}(0,T;\mathrm{B})}.
\end{eqnarray*}
The combination of these two estimates leads - for $f\in \H^1_0(0,T;B)$ -  to:
\begin{eqnarray*}
\| f \|_{\H^{1}(0,T;\mathrm{B})} & \leq & \sqrt{1+T^2} \left\| f' \right\|_{\L^{2}(0,T;\mathrm{B})}.
\end{eqnarray*}
Thus in particular the embedding mentioned above holds with the following estimate
\begin{eqnarray*}
\| X^{\ast} - \Id_{\mathcal{S}} \|_{\mathcal{W}(S_T^0)} & \leq & 
\sqrt{1+T^2} \| X^{\ast} - \Id_{\mathcal{S}} \|_{\mathcal{W}_0(0,T;\mathcal{S})},
\end{eqnarray*}
for $X^{\ast} - \Id_{\mathcal{S}} \in \mathcal{W}_0(0,T;\mathcal{S})$.}

\section{The change of variables and the change of unknowns} \label{secchange}
In order to transform the main system in domains which do not depend on time, we first extend to the whole domain $\overline{\mathcal{O}}$ the mappings $X_{\mathcal{S}}(\cdot,t)$ and $Y_{\mathcal{S}}(\cdot , t)$, initially defined respectively on $\mathcal{S}$ and $\mathcal{S}(t)$. The respective extensions $X(\cdot,t)$ and $Y(\cdot,t)$ then obtained define a change of variables which will be used to set a change of unknowns for the main system. The aim is to consider new unknowns $(\tilde{u},\tilde{p},\tilde{h}',\tilde{\omega})$ which are defined in cylindrical domains.

\subsection{The change of variables}
Let be $T_0>0$. Let $h \in \H^2(0,T_0;\R^3)$ be a vector and $\mathbf{R} \in \H^2(0,T_0;\R^9)$ a rotation which provides an angular velocity $\omega$ whose components can be read on
\begin{eqnarray*}
\mathbb{S}(\omega) & = & \frac{\d \mathbf{R}}{\d t} \mathbf{R}^T, \quad \text{with }
\mathbb{S}(\omega) = \left(
\begin{matrix}
0 & -\omega_3 & \omega_2 \\
\omega_3 & 0 & -\omega_1 \\
-\omega_2 & \omega_1 & 0
\end{matrix} \right).
\end{eqnarray*}
Since $\H^1(0,T_0;\R^3)$ is an algebra, we have $\omega \in \H^1(0,T_0;\R^3)$. \textcolor{black}{In what follows we will use the notation
\begin{eqnarray*}
\tilde{h}' = \mathbf{R}^T h', & \quad & \tilde{\omega} = \mathbf{R}^T \omega.
\end{eqnarray*}
}
For a given mapping $X^{\ast} \in \mathcal{W}_0(0,\infty;\mathcal{S})$ which satisfies the constraint
\begin{eqnarray*}
\int_{\p \mathcal{\mathcal{S}}} \frac{\p X^{\ast}}{\p t}\cdot (\com \nabla X^{\ast})n \d \Gamma & = & 0,
\end{eqnarray*}
the purpose of this subsection is to construct a mapping $X$ which satisfies
\begin{eqnarray*}
\left\{ \begin{array} {lll}
\det \nabla X = 1, & \quad & \text{in } \mathcal{F} \times (0,T), \\
X = h+\mathbf{R} X^{\ast} , & \quad & \text{on } \p \mathcal{S} \times (0,T), \\
X = \Id_{\p \mathcal{O}}, & \quad & \text{on } \p \mathcal{O} \times (0,T),
\end{array} \right.
\end{eqnarray*}
for some $T>0$, and which is such that for all $t\in [0,T)$ the function $X(\cdot,t)$ maps $\mathcal{F}$ onto $\mathcal{F}(t)$, $\p \mathcal{S}$ onto $\p \mathcal{S}(t)$, and leaves invariant the boundary $\p \mathcal{O}$. For that, let us first construct an intermediate mapping.

\begin{lemma} \label{ch3_lemmaxtension}
Let $X^{\ast}$ be a mapping such that $X^{\ast} - \Id_{\mathcal{S}} \in \mathcal{W}_0(0,\infty;\mathcal{S})$ and which satisfies for all $t\geq 0$ the equality
\begin{eqnarray}
\int_{\p \mathcal{S}} \frac{\p X^{\ast}}{\p t} \cdot \left(\com \nabla X^{\ast} \right)n \d \Gamma &  = & 0. \label{ch3_eqreccomp}
\end{eqnarray}
Then for $T >0$ small enough, there exists a mapping $\tilde{X} \in \mathcal{W}(Q_T^0) $ satisfying
\begin{eqnarray}
\left\{ \begin{array} {lcl}
\det \nabla \tilde{X}  = 1 & \quad & \text{in } \mathcal{F} \times (0,T), \label{ch3_ext1}\\
\tilde{X} = X^{\ast} & \quad & \text{on } \p \mathcal{S} \times (0,T), \\
\tilde{X} = \mathbf{R}^T(\Id-h) & \quad & \text{on } \p \mathcal{O} \times (0,T), \label{ch3_ext3} \\
\tilde{X}(\cdot,0) = \Id_{\mathcal{F}} & &
\end{array} \right.
\end{eqnarray}
and the estimate
\begin{eqnarray*}
\| \tilde{X} - \Id_{\mathcal{F}} \|_{\mathcal{W}(Q_T^0)} & \leq & C \left(
\left\| X^{\ast} - \Id_{\mathcal{S}} \right\|_{\mathcal{W}(S_{T_0}^0)}
+ \|\tilde{h}'\|_{\H^1(0,T_0;\R^3)} + \|\tilde{\omega}\|_{\H^1(0,T_0;\R^3)} \right), \label{ch3_contSF}
\end{eqnarray*}
for some independent positive constant $C$ - which in particular does not depend on $T$. Besides, if $\tilde{X}_1$ and $\tilde{X}_2$ are the solutions of problem \eqref{ch3_ext1} corresponding to the data $(X^{\ast},h_1,\mathbf{R}_1)$ and $(X^{\ast},h_2,\mathbf{R}_2)$ respectively, with
\begin{eqnarray*}
h_1(0) = h_2(0) = 0, \quad \mathbf{R}_1(0) = \mathbf{R}_2(0) = \I_{\R^3}, \quad h'_1(0) = h'_2(0), \quad \omega_1(0) = \omega_2(0),
\end{eqnarray*}
then the difference $\tilde{X}_2 - \tilde{X}_1$ satisfies
\begin{eqnarray*}
\| \tilde{X}_2 - \tilde{X}_1 \|_{\mathcal{W}_m(Q_T^0)} & \leq & \tilde{C} \left(
\|\tilde{h}'_2 -\tilde{h}'_1 \|_{\H^1(0,T_0;\R^3)} + \|\tilde{\omega}_2 - \tilde{\omega}_1 \|_{\H^1(0,T_0;\R^3)} \right),  \label{ch3_contSF21}
\end{eqnarray*}
where the constant $\tilde{C}$ does not depend on $T$.
\end{lemma}

The proof of this lemma is given in Appendix A. For all $t \in [0, T)$ we denote by $\tilde{Y}(\cdot,t)$ the inverse of $\tilde{X}(\cdot,t)$. We now directly set
\begin{eqnarray*}
X(y,t) & = & h(t) + \mathbf{R}(t) \tilde{X}(y,t),  \quad  (y,t) \in \overline{\mathcal{F}} \times (0,T), \\
Y(x,t) & = & \tilde{Y}(\mathbf{R}(t)^T(x-h(t)),t), \quad x\in \overline{\mathcal{F}(t)}, \ t\in (0,T).
\end{eqnarray*}

\subsection{Rewriting the main system in cylindrical domains} \label{sectionequiv}
Let us transform system \eqref{prems}--\eqref{ders} into a system which deals with non-depending time domains. For that we make the change of unknowns
\begin{eqnarray}
\begin{array} {lll}
\tilde{u} (y,t) =  \mathbf{R}(t)^T u(X(y,t),t), & \quad & u(x,t)  =  \mathbf{R}(t) \tilde{u} (Y(x,t),t), \label{tildeu} \\
\tilde{p} (y,t)  =  p(X(y,t),t), & \quad & p(x,t)  =  \tilde{p}(Y(x,t),t), \label{tildep}
\end{array}
\end{eqnarray}
for $x\in \overline{\mathcal{F}(t)}$ and $y\in \overline{\mathcal{F}}$. The change of variables $X$ used to define this change of unknowns has been constructed in the previous subsection. We also set
\begin{eqnarray}
\tilde{h}'(t)  =  \mathbf{R}(t)^T h'(t), & \quad & \tilde{\omega}(t)  =  \mathbf{R}(t)^T \omega(t). \label{tildeh}
\end{eqnarray}

\begin{remark} \label{remarkc}
Let us notice that if $\tilde{h}'$ and $\tilde{\omega}$ are given, then by using the second equality of \eqref{tildeh} we see that $\mathbf{R}$ satisfies the Cauchy problem
\begin{eqnarray*}
\begin{array} {ccccc}
\displaystyle \frac{\d}{\d t}(\mathbf{R}) & = & \mathbb{S}\left(
\mathbf{R}\tilde{\omega}\right) \mathbf{R} & = & \mathbf{R} \mathbb{S}\left(\tilde{\omega}\right) \\
\mathbf{R}(t=0) & = & \I_{\R^3}, & &
\end{array}
\quad \text{with }
\mathbb{S}(\tilde{\omega}) = \left(
\begin{matrix}
0 & -\tilde{\omega}_3 & \tilde{\omega}_2 \\
\tilde{\omega}_3 & 0 & -\tilde{\omega}_1 \\
-\tilde{\omega}_2 & \tilde{\omega}_1 & 0
\end{matrix} \right). \label{rotationpb}
\end{eqnarray*}
So $\mathbf{R}$ is determined in a unique way. Thus it is obvious to see that in \eqref{tildeh} the vectors $h'$ and $\omega$ are also
determined in a unique way. Moreover, since we have
\begin{eqnarray*}
u(x,t)  =  \mathbf{R}(t) \tilde{u}(Y(x,t) ,t), & \quad & p(x,t)  =  \tilde{p}(Y(x,t) ,t),
\end{eqnarray*}
and since the mapping $Y$ depends only on $h$, $\omega$ and the datum $X^{\ast}$, we finally see that if $(\tilde{u}, \tilde{p}, \tilde{h}', \tilde{\omega})$ is given, then $(u,p,h',\omega)$ is determined in a unique way.
\end{remark}

Using the change of unknowns given above by \eqref{tildeu} and \eqref{tildeh}, system \eqref{prems}--\eqref{ders} is rewritten in the cylindrical domain $\mathcal{F} \times (0,T)$ as follows
\begin{eqnarray}
\frac{\p \tilde{u}}{\p t} - \nu \LL \tilde{u} + \MM
(\tilde{u}, \tilde{h}', \tilde{\omega})+ \NN \tilde{u} + \tilde{\omega}(t)\wedge\tilde{u}+ \GG \tilde{p}
 =  0, & \quad & y \in \mathcal{F} ,\quad t\in (0,T), \label{premsfix} \\
\div \ \tilde{u}  =  g, & \quad & y \in \mathcal{F} ,\quad t\in (0,T),
\label{deusfix}
\end{eqnarray}
\begin{eqnarray}
\tilde{u} = 0 , & \quad & y\in \p \mathcal{O} ,\quad t\in (0,T), \label{troisfix} \\
\tilde{u}  =  \tilde{h}'(t) + \tilde{\omega} (t) \wedge X^{\ast}(y,t)+ \frac{\p X^{\ast}}{\p t}(y,t) ,
& \quad & y\in \p \mathcal{S},\quad t\in (0,T), \label{quatrefix}
\end{eqnarray}
\begin{eqnarray}
M \tilde{h}''(t) & = & - \int_{\p \mathcal{S}} \tilde{\sigma}(\tilde{u},
\tilde{p}) \nabla \tilde{Y}(\tilde{X})^T n \d \Gamma - M\tilde{\omega}(t)\wedge \tilde{h}'(t) , \quad t\in (0,T),   \label{cinqfix} \\
I^{\ast}(t)\tilde{\omega}' (t) & = &  -  \int_{\p \mathcal{S}} X^{\ast} \wedge \left(
\tilde{\sigma}(\tilde{u},\tilde{p})\nabla \tilde{Y}(\tilde{X})^T n\right) \d \Gamma \nonumber \\
& & - {I^{\ast}}'(t)\tilde{\omega}(t)+ I^{\ast}(t)\tilde{\omega}(t)\wedge\tilde{\omega}(t), \quad t\in (0,T), \label{sixfix}
\end{eqnarray}
\begin{eqnarray}
\tilde{u}(y,0)  =  u_0 (y), \  y\in \mathcal{F} , \quad \tilde{h}'(0)=h_1 \in \R^3 ,\quad \tilde{\omega}(0) = \omega_0 \in \R^3 , \label{dersfix}
\end{eqnarray}
\noindent where $[ \cdot ]_i$ specifies the i-th component of a vector
\begin{eqnarray}
& & [ \LL  \tilde{u} ]_i(y,t)  =  [ \nabla \tilde{u}(y,t) \Delta \tilde{Y}(\tilde{X}(y,t),t)]_i  +  \nabla^2 \tilde{u}_i(y,t) : \left(\nabla \tilde{Y} \nabla \tilde{Y}^T \right)(\tilde{X}(y,t),t),
 \label{LL} \\
& & \MM (\tilde{u}, \tilde{h}', \tilde{\omega})(y,t) = -\nabla \tilde{u} (y,t) \nabla \tilde{Y}(\tilde{X}(y,t),t)\left(\tilde{h}'(t) + \tilde{\omega} \wedge \tilde{X}(y,t) + \frac{\p \tilde{X}}{\p t}(y,t)\right) ,\nonumber \\  \label{MM} \\
& & \NN  \tilde{u}(y,t) = \nabla \tilde{u}(y,t) \nabla \tilde{Y}(\tilde{X}(y,t),t) \tilde{u}(y,t),  \label{NN} \\
& & \GG  \tilde{p} (y,t)  =  \nabla \tilde{Y}(\tilde{X}(y,t),t)^T \nabla \tilde{p}(y,t), \label{GG} \\
& & \tilde{\sigma}(\tilde{u},\tilde{p})(y,t) =
\nu\left(\nabla \tilde{u}(y,t) \nabla \tilde{Y}(\tilde{X}(y,t),t) + \nabla \tilde{Y}(\tilde{X}(y,t),t)^T \nabla \tilde{u}(y,t)^T \right) - \tilde{p}(y,t) \I_{\R^3}, \nonumber \\
& & \text{\textcolor{black}{$I^{\ast}(t) \displaystyle  = \rho_{\mathcal{S}} \displaystyle \int_{\mathcal{S}}\left(|X^{\ast}(y,t)|^2\I_{\R^3}-X^{\ast}(y,t)\otimes X^{\ast}(y,t) \right)\d y$}} \nonumber
\end{eqnarray}
and
\begin{eqnarray*}
g(y,t)  & = &  \trace\left( \nabla \tilde{u}(y,t) \left(\I_{\R^3} -
\nabla \tilde{Y}\left(\tilde{X}(y,t),t\right) \right) \right) \nonumber
\\
& = & \nabla \tilde{u}(y,t) : \left(\I_{\R^3} - \nabla \tilde{Y}\left(\tilde{X}(y,t),t\right)^T\right)
. \label{expg}
\end{eqnarray*}
\textcolor{black}{This additional divergence term can be actually expressed in the form $g = \div \ G$, where
\begin{eqnarray*}
G(y,t) & = &  \left(\I_{\R^3} - \nabla \tilde{Y}(\tilde{X}(y,t),t)\right)\tilde{u}(y,t) .
\end{eqnarray*}
Indeed, if we calculate
\begin{eqnarray*}
\div \ G & = &  \left(\I_{\R^3} - \nabla \tilde{Y}(\tilde{X})\right)^T : \nabla \tilde{u} - \tilde{u} \cdot \div \left( \nabla \tilde{Y}(\tilde{X})^T\right),
\end{eqnarray*}
the second term of this expression vanishes, because we have by construction
\begin{eqnarray*}
 \nabla \tilde{Y}(\tilde{X})^T  = \det(\nabla \tilde{X}) \nabla \tilde{Y}(\tilde{X})^T = \com(\nabla \tilde{X}) ,
\end{eqnarray*}
and the Piola identity (see \cite{Ciarlet} for instance, the first part of the proof of Theorem 1.7-1 page 39) can be written as
\begin{eqnarray*}
\div \left(\com(\nabla \tilde{X})\right) & = & 0.
\end{eqnarray*}
}

\textcolor{black}{
Searching for solutions $(u,p,h',\omega)$ to system \eqref{prems}--\eqref{ders} in the space
\begin{eqnarray*}
\mathcal{U}(0,T;\mathcal{F}(t)) \times \L^2(0,T;\mathbf{H}^1(\mathcal{F}(t))) \times \H^1(0,T;\R^3) \times \H^1(0,T;\R^3).
\end{eqnarray*}
is equivalent to searching for solutions $(\tilde{u},\tilde{p},\tilde{h}',\tilde{\omega})$ to system \eqref{premsfix}--\eqref{dersfix} in
\begin{eqnarray*}
\mathcal{U}(0,T;\mathcal{F}) \times \L^2(0,T;\mathbf{H}^1(\mathcal{F})) \times \H^1(0,T;\R^3) \times \H^1(0,T;\R^3).
\end{eqnarray*}
The main tools for transforming system \eqref{prems}--\eqref{ders} into \eqref{premsfix}--\eqref{dersfix} lie in the {\it chain rule} and in change of variables formulas in integrals, given in \cite{Gurtin} (page 51) for instance. For example, equation \eqref{cinq} is transformed into \eqref{cinqfix} by writing
\begin{eqnarray*}
\int_{\p \mathcal{S}(t)} \sigma(u,p)(x,t)n(x,t)\d \Gamma & = & 
\int_{\mathcal{S}}  \sigma(u,p)(X(y,t),t) \nabla Y(X(y,t),t)^Tn(y,0)\d \Gamma
\end{eqnarray*}
with
\begin{eqnarray*}
\nabla Y(X(y,t),t) & = & \nabla \tilde{Y}(\tilde{X}(y,t),t)\mathbf{R}^T, \\
\nabla u(X(y,t),t) & = & \mathbf{R}(t) \nabla \tilde{u}(y,t) \nabla Y(X(y,t),t) \\
& = & \mathbf{R}(t) \nabla \tilde{u}(y,t) \nabla \tilde{Y}(\tilde{X}(y,t),t)\mathbf{R}(t)^T, \\
 \sigma(u,p)(X(y,t),t) & = & \nu \left(\nabla u(X(y,t),t) + \nabla u(X(y,t),t)^T \right)
 - p(X(y,t),t)\I_{\R^3} \\
 & = &  \nu \mathbf{R}(t) \tilde{\sigma}(\tilde{u}, \tilde{p})\mathbf{R}(t)^T,
\end{eqnarray*}
and so we have
\begin{eqnarray*}
\int_{\p \mathcal{S}(t)} \sigma(u,p)n\d \Gamma & = & 
\mathbf{R}(t)\int_{\mathcal{S}} \tilde{\sigma}(\tilde{u}, \tilde{p})\nabla \tilde{Y}(\tilde{X})^Tn \d \Gamma.
\end{eqnarray*}
The same type of calculations holds for transforming \eqref{six} into \eqref{sixfix}.\\
}
\hfill \\
In order to consider a linearized system, we rewrite the nonlinear system \eqref{premsfix}--\eqref{dersfix} as follows
\begin{eqnarray}
\frac{\p \tilde{u}}{\p t} -\nu \Delta \tilde{u} + \nabla \tilde{p}
 =  F(\tilde{u},\tilde{p},\tilde{h}',\tilde{\omega}), & \quad &  \text{in } \mathcal{F}\times (0,T), \label{hpremsfix} \\
\div \ \tilde{u}  =  \div \ G(\tilde{u},\tilde{h}',\tilde{\omega}), & \quad & \text{in } \mathcal{F}\times (0,T), \label{hdeusfix}
\end{eqnarray}
\begin{eqnarray}
\tilde{u} = 0 , & \quad & \text{in } \p \mathcal{O}\times (0,T),  \label{htroisfix} \\
\tilde{u}  =  \tilde{h}'(t) + \tilde{\omega} (t) \wedge y + W(\tilde{\omega}) , & \quad & (y,t)\in \p \mathcal{S}\times (0,T), \label{hquatrefix}
\end{eqnarray}
\begin{eqnarray}
M \tilde{h}''  =  - \int_{\p \mathcal{S}} \sigma(\tilde{u},\tilde{p}) n  \d \Gamma + F_M(\tilde{u},\tilde{p},\tilde{h}',\tilde{\omega}),
& \quad & \text{in } (0,T),  \label{hcinqfix} \\
I_0\tilde{\omega}' (t)  =   -  \int_{\p \mathcal{S}} y \wedge \sigma(\tilde{u},\tilde{p}) n  \d \Gamma + F_I(\tilde{u},\tilde{p},\tilde{h}',\tilde{\omega}), & \quad  & \text{in } (0,T),  \label{hsixfix}
\end{eqnarray}
\begin{eqnarray}
\tilde{u}(y,0)  =  u_0 (y), \  y\in \mathcal{F} , \quad \tilde{h}'(0)=h_1 \in \R^3 ,\quad \tilde{\omega}(0) =\omega_0 \in \R^3 , \label{hdersfix}
\end{eqnarray}
with
\begin{eqnarray*}
F(\tilde{u},\tilde{p},\tilde{h}',\tilde{\omega}) & = & \nu (\LL - \Delta) \tilde{u} - \MM (\tilde{u}, \tilde{h}', \tilde{\omega}) - \NN \tilde{u} - (\GG-\nabla) \tilde{p} - \tilde{\omega} \wedge \tilde{u}, \label{rhsF} \\
G(\tilde{u},\tilde{h}',\tilde{\omega}) & = & \left(\I_{\R^3}- \nabla \tilde{Y}(\tilde{X}(y,t),t)\right)\tilde{u}, \label{0rhsG} \\
W(\tilde{\omega}) & = & \tilde{\omega} \wedge \left(X^{\ast} - \Id\right) + \frac{\p X^{\ast}}{\p t}, \label{0rhsW} \\
F_M(\tilde{u},\tilde{p},\tilde{h}',\tilde{\omega}) & = & -M \tilde{\omega} \wedge \tilde{h}'(t) \nonumber \\
 & & - \nu\int_{\p \mathcal{S}}\left(\nabla \tilde{u} \left(\nabla \tilde{Y}(\tilde{X}) - \I_{\R^3}\right) + \left({\nabla \tilde{Y}(\tilde{X})} - \I_{\R^3}\right)^T\nabla \tilde{u}^T \right)\nabla \tilde{Y}(\tilde{X})^T n \d \Gamma \nonumber \\
 & & - \int_{\p \mathcal{S}}\sigma(\tilde{u},\tilde{p})\left(\nabla \tilde{Y}(\tilde{X})-\I_{\R^3}\right)^Tn\d \Gamma, \label{0rhsFM} \\
F_I(\tilde{u},\tilde{p},\tilde{h}', \tilde{\omega}) & = & -\left(I^{\ast} - I_0\right) \tilde{\omega}' - {I^{\ast}}'\tilde{\omega} +I^{\ast}\tilde{\omega} \wedge \tilde{\omega} \nonumber \\
& &  - \nu\int_{\p \mathcal{S}}y\wedge \left(\nabla \tilde{u} \left(\nabla \tilde{Y}(\tilde{X}) - \I_{\R^3}\right)
+ ({\nabla \tilde{Y}(\tilde{X})} - \I_{\R^3})^T\nabla \tilde{u}^T\right)\nabla \tilde{Y}(\tilde{X})^T n \d \Gamma \nonumber \\
& & - \int_{\p \mathcal{S}}y\wedge \left(\sigma(\tilde{u},\tilde{p})(\nabla \tilde{Y}(\tilde{X})-\I_{\R^3})^Tn\right)\d \Gamma \nonumber \\
& & + \int_{\p \mathcal{S}}\left(X^{\ast}-\Id\right)\wedge \left(\tilde{\sigma}(\tilde{u},\tilde{p})
\nabla \tilde{Y}(\tilde{X})^T n\right)\d \Gamma, \label{0rhsFI} \\
\text{\textcolor{black}{$\displaystyle I^{\ast}(t)$}} 
& \text{\textcolor{black}{=}} & 
\text{\textcolor{black}{$\rho_{\mathcal{S}} \int_{\mathcal{S}}\left(|X^{\ast}(y,t)|^2\I_{\R^3}-X^{\ast}(y,t)\otimes X^{\ast}(y,t) \right)\d y.$}}
\end{eqnarray*}

\begin{remark} \label{remarkcc}
\textcolor{black}{First, we can verify that $G \in \mathcal{U}(0,T;\mathcal{F})$, and from the homogeneous condition \eqref{troisfix} on $\tilde{u}$ we have $G = 0$ on $\p \mathcal{O}$.} An important remark is the following: Since systems \eqref{prems}--\eqref{ders} and \eqref{hpremsfix}--\eqref{hdersfix} are equivalent, and since under Hypothesis {\bf H2} the compatibility condition is satisfied for system \eqref{prems}--\eqref{ders}, in system \eqref{hpremsfix}--\eqref{hdersfix} the underlying compatibility condition enables us to have automatically the following equality
\begin{eqnarray*}
\int_{\p \mathcal{S}} G(\tilde{u},\tilde{h}',\tilde{\omega})\cdot n \d \Gamma & = & \int_{\p \mathcal{S}}W(\tilde{\omega}) \cdot n \d \Gamma
\end{eqnarray*}
as soon as $\tilde{u} = 0$ on $\p \mathcal{O}$.\\
\textcolor{black}{Moreover, given the expression \eqref{hquatrefix} of $\tilde{u}$ on $\p \mathcal{S}$, we can prove that if $\tilde{h}', \tilde{\omega} \in \H^1(0,T;\R^3)$ and if $X^{\ast} - \Id_{\mathcal{S}} \in \mathcal{W}_0(0,T;\mathcal{S})$, then we can consider
\begin{eqnarray*}
G(\tilde{u},\tilde{h}',\tilde{\omega})_{| \p \mathcal{S}} & \in & \L^2(0,T;\mathbf{H}^{3/2}(\p \mathcal{S}))\cap \H^1(0,T;\mathbf{H}^{\varepsilon}(\p \mathcal{S}))
\end{eqnarray*}
as soon as $\tilde{u}_{| \p \mathcal{S}} \in \L^2(0,T;\mathbf{H}^{3/2}(\p \mathcal{S}))\cap \H^1(0,T;\mathbf{H}^{\varepsilon}(\p \mathcal{S}))$, with $0 < \varepsilon  < 1/2$. See Lemma \ref{ch3_lemmaH3} for the proof of these regularities.}
\end{remark}

\section{The nonhomogeneous linear system} \label{seclinear}
Let $\mathbb{F}$, $\mathbb{G}$, $\mathbb{W}$, $\mathbb{F}_M$ and $\mathbb{F}_I$ be some data. We assume that $\mathbb{G}$ satisfies the homogeneous condition
\begin{eqnarray*}
\mathbb{G} = 0 & & \text{ on } \p \mathcal{O}
\end{eqnarray*}
and also the compatibility condition
\begin{eqnarray*}
\int_{\p \mathcal{S}} \mathbb{G}\cdot n \d \Gamma & = & \int_{\p \mathcal{S}} \mathbb{W} \cdot n \d \Gamma.  \label{ch2_conddiv}
\end{eqnarray*}
\textcolor{black}{For $0 < \varepsilon < 1/2$,} we assume that
\begin{eqnarray*}
\begin{array} {l}
\mathbb{F} \in \L^2(0,T;\mathbf{L}^2(\mathcal{F})),
 \quad  \mathbb{G} \in \mathcal{U}(0,T;\mathcal{F}), \\
\text{\textcolor{black}{$\mathbb{G}_{|\p \mathcal{S}} \in \L^2(0,T;\mathbf{H}^{3/2}(\p \mathcal{S}))\cap \H^1(0,T;\mathbf{H}^{\varepsilon}(\p \mathcal{S}))$}}, \\
\mathbb{W} \in \L^2(0,T;\mathbf{H}^{3/2}(\p \mathcal{S}))\cap \H^1(0,T;\mathbf{H}^{1/2}(\p \mathcal{S})), \\
\mathbb{F}_M \in \L^2(0,T;\R^3),  \mathbb{F}_I \in \L^2(0,T;\R^3).
\end{array}
\end{eqnarray*}
In this section we consider for $0<  \dist(\mathcal{S},\p\mathcal{O})$ the following linear system
\begin{eqnarray}
\frac{\p \tilde{U}}{\p t} - \nu \Delta \tilde{U} + \nabla \tilde{P}  = \mathbb{F}, & \quad & \textrm{in $\mathcal{F} \times (0,T)$}, \label{premslin}\\
\div \  \tilde{U}  =  \div \ \mathbb{G}, & \quad & \textrm{in $\mathcal{F} \times (0,T)$},
\end{eqnarray}
\begin{eqnarray}
\tilde{U}  =  0 , &\quad & \textrm{on $\p \mathcal{O} \times (0,T)$}, \\
\tilde{U}  =  \tilde{h}'(t) + \tilde{\omega} (t) \wedge y + \mathbb{W} , & \quad & y \in \ \p \mathcal{S} ,\quad t\in (0,T),
\end{eqnarray}
\begin{eqnarray}
M \tilde{h}''(t) = - \int_{\p \mathcal{S}} \sigma(\tilde{U},\tilde{P}) n \d \Gamma + \mathbb{F}_M , \quad t\in (0,T),\\
I_0\tilde{\omega}' (t) = -  \int_{\p \mathcal{S}} y\wedge \sigma(\tilde{U},\tilde{P}) n \d \Gamma  + \mathbb{F}_I  , \quad  t\in (0,T),
\end{eqnarray}
\begin{eqnarray}
\tilde{U}(y,0)  =  u_0 (y), \ y \in \mathcal{F}, \quad \tilde{h}'(0)=h_1 \in \R^3 ,  \quad  \tilde{\omega}(0) = \omega_0 \in \R^3. \label{derslin}
\end{eqnarray}

\subsection{A lifting method}
Let us first eliminate the nonhomogeneous divergence condition: By setting
\begin{eqnarray*}
U=\tilde{U}-G, \quad P=\tilde{P}, \quad H'=\tilde{h}', \quad  \Omega = \tilde{\omega}
\end{eqnarray*}
we rewrite system \eqref{premslin}--\eqref{derslin} as
\begin{eqnarray*}
\frac{\p U}{\p t} - \nu \Delta U + \nabla P  =  \hat{F}, & \quad & \textrm{in $\mathcal{F} \times (0,T)$}, \label{ch2_premsss}\\
\div \  U  =  0, & \quad & \textrm{in $\mathcal{F} \times (0,T)$},
\end{eqnarray*}
\begin{eqnarray*}
U  =  0, & \quad & \textrm{on } \p \mathcal{O} \times (0,T), \\
U  =  H'(t) + \Omega (t) \wedge y + \hat{W}, & \quad &  y \in \ \p \mathcal{S}
,\ t\in (0,T),
\end{eqnarray*}
\begin{eqnarray*}
M H''(t) = - \int_{\p \mathcal{S}} \sigma(U,P) n \d \Gamma + \hat{F}_M , \quad  t\in
(0,T),\\
I_0\Omega' (t) = -  \int_{\p \mathcal{S}} y\wedge  \sigma(U,P)
n \d \Gamma + \hat{F}_I  , \quad  t\in (0,T),
\end{eqnarray*}
\begin{eqnarray*}
U(y,0)  = u_0 (y), \ y \in \mathcal{F}, \quad H'(0)=h_1 \in \R^3 ,  \quad  \Omega(0) = \omega_0 \in \R^3, \label{ch2_dersss}
\end{eqnarray*}
with
\begin{eqnarray*}
\begin{array} {ll}
\hat{F}  =  \mathbb{F} - \displaystyle \frac{\p \mathbb{G}}{\p t}+ \text{\textcolor{black}{$\nu \Delta \mathbb{G}$}} ,  & \hat{W}  =  \mathbb{W}-\mathbb{G} , \\
\hat{F}_M  =  \mathbb{F}_M - \displaystyle 2\nu \int_{\p \mathcal{S}} D(\mathbb{G})n\d \Gamma ,  & \hat{F}_I  =  \mathbb{F}_I - \displaystyle 2\nu \int_{\p \mathcal{S}} y\wedge D(\mathbb{G})n\d \Gamma.
\end{array}
\end{eqnarray*}
We now use a {\it lifting method} in order to tackle the non-homogeneous Dirichlet condition $\hat{W}$ on $\p \mathcal{S}$ and then establish an existence result for the linear system \eqref{premslin}--\eqref{derslin}. We split this problem into two more simple problems, by setting
\begin{eqnarray*}
U= V + \w , \label{ch2_decompou} & \quad &
P= Q + \pi , \label{ch2_decompop}
\end{eqnarray*}
where, for all $t\in (0,T)$, the couple $(\w,\pi)$ satisfies
\begin{eqnarray}
-  \nu \Delta \w(t) + \nabla \pi(t) = 0, & \quad & \textrm{in
$\mathcal{F}$}, \label{ch2_jpr1}  \\
\div \  \w(t) = 0 , & \quad & \textrm{in $\mathcal{F}$},  \label{ch2_jpr2}  \\
\w(t)  =  W(\cdot,t) , & \quad & \textrm{on $\p \mathcal{S}$}, \label{ch2_jpr25} \\
\w(t) = 0 , & \quad & \textrm{on $\p \mathcal{O}$} , \label{ch2_jpr3}
\end{eqnarray}
and where the couple $(V,Q)$ satisfies
\begin{eqnarray}
\frac{\p V}{\p t} - \nu \Delta V + \nabla Q = F , & \quad & \textrm{in $\mathcal{F} \times (0,T)$}, \label{ch2_premlinp}\\
\div \  V = 0 ,  & \quad & \textrm{in $\mathcal{F} \times (0,T)$},\\
V = 0 , & \quad & \textrm{on $\p \mathcal{O} \times (0,T)$}, \\
V = H'(t) + \Omega (t)\wedge y  , & \quad &  y\in \ \p \mathcal{S} ,\quad t\in (0,T),
\label{ch2superS}\\
M H''(t) = - \int_{\p \mathcal{S}} \sigma(V,Q) n \d \Gamma  + F_M, & \quad &  t\in (0,T),\\
I_0\Omega' (t) = -  \int_{\p \mathcal{S}} y\wedge \sigma(V,Q) n \d \Gamma + F_I  , & \quad &  t\in(0,T),
\end{eqnarray}
\begin{eqnarray}
V(y,0) = u_0 (y) - \w(y,0), \  y\in \mathcal{F}, \quad H'(0)=h_1 \in \R^3 ,  \quad \Omega(0) = \omega_0 \in \R^3, \label{ch2_derlinp}
\end{eqnarray}
with
\begin{eqnarray*}
\begin{array} {ll}
F  =  \hat{F} - \displaystyle \frac{\p \w}{\p t}  , & W  =  \hat{W}, \\
F_M  =  \hat{F}_M + \displaystyle \int_{\p \mathcal{S}}\sigma(\w, \pi)n\d \Gamma , & F_I  =  \hat{F}_I + \displaystyle \int_{\p \mathcal{S}} y\wedge \sigma(\w,\pi)n\d \Gamma.
\end{array}
\end{eqnarray*}

To sum up, we have as right-hand-sides:
\begin{eqnarray}
& & F  =  \mathbb{F} - \frac{\p \mathbb{G}}{\p t} + \nu \Delta \mathbb{G} - \frac{\p \w}{\p t} ,
\label{ch2_secm1} \\
& & W  =  \mathbb{W}-\mathbb{G} , \\
& & F_M  =  \mathbb{F}_M - 2\nu \int_{\p \mathcal{S}} D(\mathbb{G})n\d \Gamma + \displaystyle \int_{\p \mathcal{S}}\sigma(\w, \pi)n\d \Gamma,\\
& & F_I  =  \mathbb{F}_I - 2\nu \int_{\p \mathcal{S}} y\wedge D(\mathbb{G})n\d \Gamma
+ \displaystyle \int_{\p \mathcal{S}} y\wedge \sigma(\w,\pi)n\d \Gamma . \label{ch2_secm4}
\end{eqnarray}

\subsubsection{Stokes problem} \label{ch2_transpart}
We now look at the problem \eqref{ch2_jpr1}--\eqref{ch2_jpr3}. Let us keep in mind that we have the compatibility condition
\begin{eqnarray*}
 \int_{\p \mathcal{S}}\left(\mathbb{W}(y) -\mathbb{G}(y)\right)\cdot n \d \Gamma & = & 0 . \label{ch2_condadm}
\end{eqnarray*}
Let us set a result of existence and uniqueness in $\mathcal{U}(0,T;\mathcal{F}) \times \L^2(0,T;\H^1(\mathcal{F} ))$ for this nonhomogeneous boundary problem, which is a consequence of a result stated in \cite{Galdi1}, Theorem 6.1, Chapter IV.\\

\begin{proposition} \label{ch2_propw}
There exists a unique couple $(\w,\nabla\pi)\in \mathcal{U}(0,T;\mathcal{F}) \times
\L^2(0,T;\L^2(\mathcal{F} ))$ solution of system \eqref{ch2_jpr1}--\eqref{ch2_jpr3}
for almost all $t\in (0,T)$. Moreover, there exists a positive constant $C$ such that
\begin{eqnarray*}
& & \| \w \|_{\mathcal{U}(0,T;\mathcal{F})}
\text{\textcolor{black}{$+  \| \w_{|\p \mathcal{S}} \|_{\H^1(0,T;\H^{\varepsilon}(\p \mathcal{S} ))}$}}
+ \| \nabla \pi \|_{\L^2(0,T;\L^2(\mathcal{F} ))} \leq \\
& & C \left( \|\mathbb{W}\|_{\L^2(0,T;\mathbf{H}^{3/2}(\p \mathcal{S}))\cap \H^1(0,T;\mathbf{H}^{1/2}(\p \mathcal{S}))} \right.   \text{\textcolor{black}{$ \displaystyle
\left. + \| \mathbb{G} \|_{\mathcal{U}(0,T;\mathcal{F})} + 
\|\mathbb{G}_{|\p \mathcal{S}}\|_{\H^1(0,T;\mathbf{H}^{\varepsilon}(\p \mathcal{S}))} \right)$}}.
\end{eqnarray*}
\end{proposition}
The estimate we give in this proposition is not sharp, but it is sufficient for what will follow.

\subsubsection{Semigroup approach} \label{ch2_paraoui}
We solve \eqref{ch2_premlinp}--\eqref{ch2_derlinp} by proceeding as in \cite{Tucsnak}. We
project the unknown $V$ on the space
\begin{displaymath}
\mathcal{H} = \displaystyle \left\{ \phi \in  \mathbf{L}^2(\mathcal{O})
\mid \ \div \ \phi = 0 \text{ in } \mathcal{O}, \ D(\phi) = 0 \text{ in } \mathcal{S},
\ \phi \cdot n = 0 \text{ on $\p \mathcal{O}$}\right\},
\end{displaymath}
and we consider
\begin{displaymath}
\mathcal{V} = \displaystyle \left\{ \phi \in \mathbf{H}^1(\mathcal{O})
\mid \ \div \ \phi = 0 \text{ in } \mathcal{O}, \ D(\phi) = 0 \text{ in } \mathcal{S},
\ \phi \cdot n = 0 \text{ on $\p \mathcal{O}$}\right\}.
\end{displaymath}

Let us remind a lemma stated in \cite[page 18]{Temam}.

\begin{lemma} \label{ch2_lemmeVW}
For all $\phi \in \mathcal{H}$, there exists $l_{\phi} \in \R^3$ and
$\omega_{\phi} \in \R$ such that
\begin{eqnarray*}
\phi (y) = l_{\phi} + \omega_{\phi}\wedge y \text{ for all } y\in \mathcal{S}.
\end{eqnarray*}
\end{lemma}
This result allows us to extend $V$ in $\mathcal{S}$ and then consider the system in the
whole domain $\mathcal{O}$. Indeed, for $V \in \mathcal{H}$, this lemma gives us two vectors $H'$ and $\Omega$ such that
\begin{eqnarray*}
V  =  H_V'(t) + \Omega_V(t) \wedge y  =  H'(t) + \Omega(t) \wedge y.
\end{eqnarray*}
Let us now define a new inner product on $ \mathbf{L}^2(\mathcal{O})$ by setting
\begin{eqnarray*}
(\psi , \phi)_{\mathbf{L}^2(\mathcal{O})}  = \int_{\mathcal{F}} (\psi \cdot  \phi)\d y
+ \rho_{\mathcal{S}} \int_{\mathcal{S}} \psi(y) \cdot
\phi(y) \d y . \label{ch2_ps}
\end{eqnarray*}
Let us remind that $\rho_{\mathcal{S}} >0$ is the constant density of the rigid body $\mathcal{S}$.
The corresponding Euclidean norm is equivalent to the usual one in $\mathbf{L}^2(\mathcal{O})$. If two functions $\psi$ et $\phi$
lie in $\mathcal{H}$, then a simple calculation leads us to
\begin{displaymath}
(\psi , \phi)_{\mathbf{L}^2(\mathcal{O})} = \int_{\mathcal{F}} (\psi \cdot \phi) \d y
+ Ml_{\phi} \cdot l_{\psi} + I_0\omega_{\phi} \cdot\omega_{\psi}.
\end{displaymath}
In order to solve \eqref{ch2_premlinp}--\eqref{ch2_derlinp} we use a semigroup approach. We define
\begin{eqnarray*}
D(A) & = & \left\{\phi \in  \mathbf{H}^1(\mathcal{O}) \mid \
\phi_{|\mathcal{F}}
\in \mathbf{H}^2(\mathcal{F}) , \ \div \ \phi = 0 \text{ in } \mathcal{O},
\ D(\phi)=0 \text{ in } \mathcal{S} , \ \phi \cdot n = 0 \text{ on $\p
\mathcal{O}$}\right\}. \label{ch2_defdom}
\end{eqnarray*}
For all $V \in D(A)$ we set
\begin{eqnarray*}
\mathcal{A}V =
\left\{\displaystyle\begin{array} {llll}
- \nu \Delta V \text{  in } \mathcal{F} , \\

\displaystyle\frac{2\nu}{M} \displaystyle\int_{\p \mathcal{S}} D(V)n\d \Gamma + \left( 2\nu {I_0}^{-1} \displaystyle \int_{\p \mathcal{S}}
y\wedge  D(V)n\d \Gamma \right)\wedge y \text{  in } \mathcal{S},
\end{array}
\right.
\end{eqnarray*}
and
\begin{eqnarray*}
AV = \mathbb{P} \mathcal{A} V ,
\end{eqnarray*}
where $\mathbb{P}$ is the orthogonal projection from $\mathbf{L}^2(\mathcal{O})$ onto $\mathcal{H}$. Then we get a unique solution $(V,Q,H',\Omega)$ in $\mathcal{U}(0,T;\mathcal{F}) \times \L^2(0,T;\H^1(\mathcal{F}  )) \times \H^1(0,T;\R^3 ) \times \H^1(0,T;\R^3 )$ by following the steps of \cite{Tucsnak}.

\subsection{The main result for the linearized system}
\begin{proposition} \label{ch2_thsolfaible}
Let $\mathbb{F} \in \L^2(0,T;\mathbf{L}^2(\mathcal{F}))$, $\mathbb{F}_M \in\L^2(0,T;\R^3)$, $\mathbb{F}_I \in \L^2(0,T;\R^3)$, $\mathbb{G} \in \mathcal{U}(0,T;\mathcal{F})$ and $\mathbb{W}\in \L^2(0,T;\mathbf{H}^{3/2}(\p \mathcal{S}))\cap \H^1(0,T;\mathbf{H}^{1/2}(\p \mathcal{S}))$ be given. Let us assume that  $\mathbb{G}$ satisfy $\mathbb{G} = 0$ on $\p \mathcal{O}$, \textcolor{black}{$\mathbb{G}_{| \p \mathcal{S}}\in \H^1(0,T;\mathbf{H}^{\varepsilon}(\p \mathcal{S}))$ for $0 < \varepsilon < 1/2$,} and the compatibility condition
\begin{eqnarray*}
\int_{\p \mathcal{S}} \mathbb{G} \cdot n \d \Gamma & = & \int_{\p \mathcal{S}} \mathbb{W} \cdot n \d \Gamma.
\end{eqnarray*}
Assume that $0<  \dist(\mathcal{S},\p\mathcal{O})$ and that $u_0\in \mathbf{H}^1(\mathcal{F})$ with
\begin{eqnarray*}
\div \  u_0  =  0   \textrm{ in $\mathcal{F}$}, \quad
u_0  =  0  \textrm{ on $\p \mathcal{O}$}, \quad
u_0(y)  =  h_1 + \omega_0 \wedge y  \textrm{ on $\p \mathcal{S}$}.
\end{eqnarray*}
Then system \eqref{premslin}--\eqref{derslin} admits a unique solution $(\tilde{U},\tilde{P},\tilde{h}',\tilde{\omega})$ in 
\begin{eqnarray*}
\mathcal{U}(0,T;\mathcal{F}) \times \L^2(0,T;\H^1(\mathcal{F} )) \times \H^1(0,T;\R^3 ) \times \H^1(0,T;\R^3 ),
\end{eqnarray*}
\textcolor{black}{up to a constant for $\tilde{P}$ that we choose such that $\int_{\mathcal{F}} \tilde{P}  =  0$.} Moreover, \textcolor{black}{$\tilde{U}_{|\p \mathcal{S}} \in \H^1(0,T;\mathbf{H}^{\varepsilon}(\p \mathcal{S}))$} and there exists a positive constant $K$ such that
\begin{eqnarray*}
& & \|\tilde{U}\|_{\mathcal{U}(0,T;\mathcal{F})} 
\text{\textcolor{black}{$+ \|\tilde{U}_{|\p \mathcal{S}}\|_{\H^1(0,T;\mathbf{H}^{\varepsilon}(\p \mathcal{S}))}$}}
+\| \nabla \tilde{P} \|_{\L^2(0,T;\L^2(\mathcal{F}))} + \|
\tilde{h}'\|_{ \H^1(0,T;\R^3)} + \|\tilde{\omega} \|_{ \H^1(0,T;\R^3 )}  \\
& & \leq K\left( \|u_0\|_{\mathbf{H}^1(\mathcal{O})} + |h_1|_{\R^3} + |\omega |_{\R^3}
+  \| \mathbb{F} \|_{\L^2(0,T;\mathbf{L}^2)} 
+ \text{\textcolor{black}{$\| \mathbb{G} \|_{\mathcal{U}(0,T;\mathcal{F})} +\|\mathbb{G}_{|\p \mathcal{S}}\|_{ \H^1(0,T;\mathbf{H}^{\varepsilon}(\p \mathcal{S}))}$}} \right. \\
& & \left. + \|\mathbb{W}\|_{\L^2(0,T;\mathbf{H}^{3/2}(\p \mathcal{S}))\cap \H^1(0,T;\mathbf{H}^{1/2}(\p \mathcal{S}))}+ \| \mathbb{F}_M \|_{\L^2(0,T;\R^3)}+ \| \mathbb{F}_I \|_{ \L^2(0,T;\R^3 )} \right).
\end{eqnarray*}
The constant $K$ depends only on $T$, and is nondecreasing with respect to~$T$.
\end{proposition}

\begin{proof}
Proposition \ref{ch2_propw} provides us a solution $(\w,\pi)\in \mathcal{U}(0,T;\mathcal{F}) \times \L^2(0,T;\H^1(\mathcal{F} ))$ for the nonhomogeneous
Stokes problem \eqref{ch2_jpr1}--\eqref{ch2_jpr3}. Let us remind the expressions \eqref{ch2_secm1}--\eqref{ch2_secm4} of the quantities which
appear in some second members of the system \eqref{ch2_premlinp}--\eqref{ch2_derlinp}:
\begin{eqnarray*}
& & F  =  \mathbb{F} - \displaystyle \frac{\p \mathbb{G}}{\p t} + \nu \Delta \mathbb{G} - \frac{\p \w}{\p t} ,\\
& & W  =  \mathbb{W}-\mathbb{G} , \\
& & F_M  =  \mathbb{F}_M - \displaystyle 2\nu \int_{\p \mathcal{S}} D(\mathbb{G})n\d \Gamma + \displaystyle \int_{\p \mathcal{S}}\sigma(\w, \pi)n\d\Gamma, \\
& & F_I  =  \mathbb{F}_I - \displaystyle 2\nu \int_{\p \mathcal{S}} y\wedge  D(\mathbb{G})n\d \Gamma + \displaystyle \int_{\p \mathcal{S}}y\wedge \sigma(\w, \pi)n\d \Gamma.
\end{eqnarray*}
Then the semigroup approach \ref{ch2_paraoui} provides us a solution $(V,Q,H',\Omega)$ for the problem \eqref{ch2_premlinp}--\eqref{ch2_derlinp} \textcolor{black}{(see \cite{Tucsnak} for more details)}, with
\begin{eqnarray*}
& & V\in \L^2\left( 0,T;\mathbf{H}^2(\mathcal{F})  \right) \cap C\left( [0,T];\mathbf{H}^1(\mathcal{F}) \right) \cap \H^1\left( 0,T;\mathbf{L}^2(\mathcal{F})
\right), \\
& & \nabla Q\in \L^2\left( 0,T; \L^2(\mathcal{F}) \right), \quad H' \in \H^1\left( 0,T; \R^3 \right), \quad \Omega \in \H^1\left( 0,T; \R^3 \right).
\end{eqnarray*}
We get then
\begin{displaymath}
(U,P) = (V,Q) + (\w,\pi),
\end{displaymath}
so by setting $(\tilde{U},\tilde{P},\tilde{h}',\tilde{\omega}) = (U + \mathbb{G},P,H',\Omega)$ we get a solution for the problem \eqref{premslin}--\eqref{derslin}. \textcolor{black}{The estimate of $\tilde{U}_{|\p \mathcal{S}}$ in $\H^1(0,T;\mathbf{H}^{\varepsilon}(\p \mathcal{S}))$ is directly deduced from the equality
\begin{eqnarray*}
\tilde{U}_{|\p \mathcal{S}} & = & V_{|\p \mathcal{S}} + \w_{|\p \mathcal{S}} + \mathbb{G}_{|\p \mathcal{S}} \\
& = & H' + \Omega \wedge y + \w_{|\p \mathcal{S}} + \mathbb{G}_{|\p \mathcal{S}},
\end{eqnarray*}
an estimate of $H'$ and $\Omega$ in $\H^1(0,T;\R^3)$ provided by the semigroup theory and the estimate of $\w$ given by Proposition \ref{ch2_propw}.} For the \textcolor{black}{rest of the} announced estimate, we first write
\begin{eqnarray*}
\|\tilde{U}\|_{\mathcal{U}(0,T;\mathcal{F})} & \leq & \|U\|_{\mathcal{U}(0,T;\mathcal{F})} + \| \mathbb{G} \|_{\mathcal{U}(0,T;\mathcal{F})} \\
 & \leq & \|V\|_{\mathcal{U}(0,T;\mathcal{F})} + \| \w \|_{\mathcal{U}(0,T;\mathcal{F})}+ \| \mathbb{G} \|_{\mathcal{U}(0,T;\mathcal{F})}
\end{eqnarray*}
and
\begin{eqnarray*}
\| \nabla \tilde{P}\|_{\L^2(0,T;\L^2(\mathcal{F}))}
& \leq & \| \nabla Q\|_{\L^2(0,T;\L^2(\mathcal{F}))} + \| \nabla \pi \|_{\L^2(0,T;\L^2(\mathcal{F}))}.
\end{eqnarray*}
Then we use the estimate of Proposition \ref{ch2_propw} to get
\begin{eqnarray*}
& & \|\tilde{U}\|_{\mathcal{U}(0,T;\mathcal{F})} + \| \nabla \tilde{P}\|_{\L^2(0,T;\L^2(\mathcal{F}))} \leq
\|V\|_{\mathcal{U}(0,T;\mathcal{F})} + \| \nabla Q\|_{\L^2(0,T;\L^2(\mathcal{F}))} + \| \mathbb{G} \|_{\mathcal{U}(0,T;\mathcal{F})}   \\
& & + C \left( \| \mathbb{G}\|_{\mathcal{U}(0,T;\mathcal{F})} +
\text{\textcolor{black}{
$\|\mathbb{G}\|_{ \H^1(0,T;\mathbf{H}^{\varepsilon}(\p \mathcal{S}))}$
}} \right. 
\left. +
\|\mathbb{W}\|_{\L^2(0,T;\mathbf{H}^{3/2}(\p \mathcal{S}))\cap \H^1(0,T;\mathbf{H}^{1/2}(\p \mathcal{S}))} \right). 
\end{eqnarray*}
It remains us to use the estimate of the semigroup theory for estimating $\|V\|_{\mathcal{U}(0,T;\mathcal{F})} + \| \nabla Q\|_{\L^2(0,T;\L^2(\mathcal{F}))}$, and to use again the estimate of Proposition \ref{ch2_propw} to \textcolor{black}{get the desired estimate. The uniqueness is due to the linearity of system \eqref{premslin}--\eqref{derslin}, and the fact that without right-hand-sides we have for this system the energy estimate
\begin{eqnarray*}
\frac{\d}{\d t}\left(
\| \tilde{U}\|^2_{\mathbf{L}^2(\mathcal{F})} + M|\tilde{h}'(t)|^2
+ I_0\tilde{\omega}\cdot \tilde{\omega}
\right) & = & -4\nu \|D(\tilde{U})\|^2_{\mathbf{L}^2(\mathcal{F})} ;
\end{eqnarray*}
With null initial conditions, the Gr\"{o}nwall's lemma applied to this estimate leads to
\begin{eqnarray*}
\tilde{U} = 0, \quad \tilde{P} = 0, \quad \tilde{h}' = 0, \quad \tilde{\omega} = 0.
\end{eqnarray*}
Thus the proof is complete
}
\end{proof}

\section{Local existence of strong solutions} \label{secfixe}
\subsection{Statement}
\begin{theorem} \label{ch3_thlocex} 
Assume that $X^{\ast}-\Id_{\mathcal{S}} \in \mathcal{W}_0(0,\infty;\mathcal{S})$ satisfies the hypotheses $\mathbf{H1}-\mathbf{H4}$. Assume that $0<  \dist(\mathcal{S},\p\mathcal{O})$, and that $u_0 \in \mathbf{H}^1(\mathcal{F})$ satisfies
\begin{eqnarray*}
\div \  u_0  =  0   \textrm{ in $\mathcal{F}$}, \quad
u_0  =  0  \textrm{ on $\p \mathcal{O}$}, \quad
u_0(y)  =  h_1 + \omega_0 \wedge y  \textrm{ on $\p \mathcal{S}$}.
\end{eqnarray*}
Then there exists $T_0 > 0$ such that problem \eqref{prems}--\eqref{ders} admits a unique strong solution
$(u,p,h,\omega)$ in
\begin{eqnarray*}
\mathcal{U} (0,T_0;\mathcal{F}(t) ) \times \L^2(0,T_0;\H^1(\mathcal{F}(t)  )) \times \H^2(0,T_0;\R^3 ) \times \H^1(0,T_0;\R^3 ).
\end{eqnarray*}
\end{theorem}

\begin{remark}
The existence of a local strong solution for system \eqref{premsfix}--\eqref{dersfix} is going to be obtained by a fixed point method for some time $T_0$ small enough. This system is equivalent to system \eqref{prems}--\eqref{ders}, up to the change of variables whose existence is conditioned by an other time $T$ small enough (see Lemma \ref{ch3_lemmaxtension} and the change of unknowns given in \eqref{tildeu}). So, by reducing the existence time $T_0$ to $T$, we can get the desired local strong solution for system \eqref{prems}--\eqref{ders}.
\end{remark}

\subsection{Proof} 
\textcolor{black}{Remind that $0 < \varepsilon < 1/2$.} Let us set
\begin{eqnarray*}
\mathbb{H}_T & = &
\left\{(U,P,H',\Omega) \in \mathcal{U}(0,T;\mathcal{F}) \times \L^2(0,T;\H^1(\mathcal{F})) \times \H^1(0,T;\R^3) \times \H^1(0,T;\R^3) \mid \right. \\
& & \left. \quad U_{| \p \mathcal{O}} = 0 \text{ in } \p \mathcal{O} \times (0,T), \ \text{\textcolor{black}{$ \displaystyle
U_{| \p \mathcal{S}} \in \H^1(0,T;\mathbf{H}^{\varepsilon}(\p \mathcal{S}))$}}
\right\}.
\end{eqnarray*}
\textcolor{black}{We endow this space with the natural norm on the Cartesian product to which we add the norm of $U_{| \p \mathcal{S}}$ in $\H^1(0,T;\mathbf{H}^{\varepsilon}(\p \mathcal{S}))$:
\begin{eqnarray*}
 \| (U,P,H',\Omega) \|_{\mathbb{H}_T} & := &
 \| U \|_{\mathcal{U}(0,T;\mathcal{F})} + \| U_{| \p \mathcal{S}} \|_{\H^1(0,T;\mathbf{H}^{\varepsilon}(\p \mathcal{S}))} \\
& &  +\|P \|_{\L^2(0,T;\H^1(\mathcal{F}))} +
\| H' \|_{\H^1(0,T;\R^3)} + \| \Omega \|_{\H^1(0,T;\R^3)}.
\end{eqnarray*}
}
The equivalence of the solutions of systems \eqref{prems}--\eqref{ders} and \eqref{premsfix}--\eqref{dersfix} has been explained in section \ref{sectionequiv}. A solution of system \eqref{premsfix}--\eqref{dersfix} is seen as a fixed point of the mapping
\begin{eqnarray*}
\begin{array} {cccc}
\mathcal{N} : & \mathbb{H}_T & \rightarrow & \mathbb{H}_T \\
& (V,Q,K',\varpi) & \mapsto & (U, P, H', \Omega)
\end{array}
\end{eqnarray*}
where $(U, P, H', \Omega)$ satisfies
\begin{eqnarray*}
\frac{\p U}{\p t} -\nu \Delta U + \nabla P =  F(V,Q,K',\varpi), & \quad & \text{in } \mathcal{F}\times (0,T), \\
\div \ U  =  \div \ G(V,K',\varpi), & \quad & \text{in } \mathcal{F}\times (0,T),
\end{eqnarray*}
\begin{eqnarray*}
U = 0 , & \quad & \text{in } \p \mathcal{O}\times (0,T),  \\
U =  H'(t) + \Omega (t) \wedge y+ W(\varpi) , & \quad & (y,t)\in \p \mathcal{S}\times (0,T),
\end{eqnarray*}
\begin{eqnarray*}
M H''   =  - \int_{\p \mathcal{S}} \sigma(U,P) n  \d \Gamma + F_M(V,Q,K',\varpi), & \quad & \text{in } (0,T)   \\
I_0\Omega' (t)   =   -  \int_{\p \mathcal{S}} y \wedge \sigma(U,P) n  \d \Gamma + F_I(V,Q,K',\varpi), & \quad & \text{in } (0,T)
\end{eqnarray*}
\begin{eqnarray*}
U(y,0)  =  u_0 (y), \  y\in \mathcal{F} , \quad H'(0)=h_1 \in \R^3 ,\quad \Omega(0) =\omega_0 \in \R^3.
\end{eqnarray*}
The expressions of the right-hand-side are given by
\begin{eqnarray}
F(V,Q,K',\varpi) & = & \nu (\LL_{(K',\varpi)} - \Delta) V - \MM_{(K',\varpi)} (V,K',\varpi) - \NN_{(K',\varpi)} V \nonumber \\
& & - (\GG_{(K',\varpi)}-\nabla) Q - \varpi \wedge V, \label{ch3_rhsF} \\
G(V,K',\varpi) & = & \left(\I_{\R^3}- \nabla \tilde{Y}(\tilde{X}(y,t),t)\right)V, \label{ch3_rhsG} \\
W(\varpi) & = & \varpi \wedge \left(X^{\ast} - \Id\right) + \frac{\p X^{\ast}}{\p t}, \label{ch3_rhsW} \\
F_M(V,Q,K',\varpi) & = & -M \varpi \wedge K'(t) \nonumber \\
 & & - \nu\int_{\p \mathcal{S}}\left(\nabla V \left(\nabla \tilde{Y}(\tilde{X}) - \I_{\R^3}\right) + \left(\nabla \tilde{Y}(\tilde{X}) - \I_{\R^3}\right)^T\nabla V^T \right) \nabla \tilde{Y}(\tilde{X})^T n \d \Gamma \nonumber \\
 & & - \int_{\p \mathcal{S}}\sigma(V,Q)\left(\nabla \tilde{Y}(\tilde{X})-\I_{\R^3}\right)^T n\d \Gamma, \label{ch3_rhsFM} \\
F_I(V,Q,K',\varpi) & = & -\left(I^{\ast} - I_0\right) \Omega'  - {I^{\ast}}'\Omega + I^{\ast}\Omega \wedge \Omega \nonumber \\
& &  - \nu\int_{\p \mathcal{S}}y\wedge \left(\nabla V \left(\nabla \tilde{Y}(\tilde{X}) - \I_{\R^3}\right) + ({\nabla \tilde{Y}(\tilde{X})} - \I_{\R^3})^T\nabla V^T\right)\nabla \tilde{Y}(\tilde{X})^Tn \d \Gamma \nonumber \\
& & - \int_{\p \mathcal{S}}y\wedge \sigma(V,Q)\left(\nabla \tilde{Y}(\tilde{X})-\I_{\R^3}\right)n\d \Gamma \nonumber \\
& & + \int_{\p \mathcal{S}}\left(X^{\ast}-\Id\right)\wedge \left(\tilde{\sigma}(V,Q)\nabla \tilde{Y}(\tilde{X})^T n\right)\d \Gamma. \label{ch3_rhsFI}
\end{eqnarray}
The mapping $\tilde{X}$ is given by Lemma \ref{ch3_lemmaxtension}, with $(K',\varpi,X^{\ast})$ as data.
For the expression of $F(V,Q,K',\varpi)$, let us remind that
\begin{eqnarray*}
[\LL_{(K',\varpi)} (V)]_i(y,t) & = & [ \nabla V(y,t) \Delta \tilde{Y}(\tilde{X}(y,t),t)]_i  +  \nabla^2 V_i(y,t) : \left(\nabla \tilde{Y} \nabla \tilde{Y}^T \right)(\tilde{X}(y,t),t), \\
\MM_{(K',\varpi)} (V,K',\varpi)(y,t) & = & -\nabla V (y,t) \nabla \tilde{Y}(\tilde{X}(y,t),t)\left(K'(t) + \varpi \wedge \tilde{X}(y,t) + \frac{\p \tilde{X}}{\p t}(y,t)\right), \\
\NN_{(K',\varpi)} V(y,t) & = & \nabla V(y,t) \nabla \tilde{Y}(\tilde{X}(y,t),t) V(y,t), \\
\GG_{(K',\varpi)} Q(y,t) & = & \nabla \tilde{Y}(\tilde{X}(y,t),t)^T \nabla Q(y,t).
\end{eqnarray*}

\subsubsection{Preliminary estimates}

The estimates given in the lemmas below are not necessarily sharp, but they are sufficient to prove the desired result.

\begin{lemma}
There exists a positive constant $C$ such that for all $(V,Q,K',\varpi)$ in $\mathbb{H}_T$ we have
\begin{eqnarray*}
& &  \left\| \left(\Delta - \LL\right)V \right\|_{\L^2(0,T;\mathbf{L}^2(\mathcal{F}))}  \leq   C
\left\| V \right\|_{\L^2(0,T;\mathbf{H}^2(\mathcal{F}))} \times \\
& &  \left( \| \nabla \tilde{Y}(\tilde{X})\nabla \tilde{Y}(\tilde{X})^T - \I_{\R^3} \|_{\L^{\infty}(0,T;\mathbf{H}^{2}(\mathcal{F}))}
+ \| \Delta \tilde{Y}(\tilde{X}(\cdot,t),t) \|_{\L^{\infty}(0,T;\mathbf{H}^{1}(\mathcal{F}))} \right),  \\
& & \| \Delta \tilde{Y}(\tilde{X}(\cdot,t),t) \|_{\L^{\infty}(0,T;\mathbf{H}^{1}(\mathcal{F}))} \leq
C \| \nabla \tilde{Y}(\tilde{X}) -\I_{\R^3} \|_{\L^{\infty}(0,T;\mathbf{H}^{2}(\mathcal{F}))} \| \nabla \tilde{Y}(\tilde{X}) \|_{\L^{\infty}(0,T;\mathbf{H}^{2}(\mathcal{F}))},     \nonumber \\
& &  \left\| (\nabla - \GG)Q \right\|_{\L^2(0,T;\mathbf{L}^2(\mathcal{F}))} \leq C\| \nabla \tilde{Y}(\tilde{X}) - \I_{\R^3} \|_{\L^{\infty}(0,T;\H^{2}(\mathcal{F}))} \left\| \nabla Q \right\|_{\L^2(0,T;\mathbf{L}^2(\mathcal{F}))}.
\end{eqnarray*}
\end{lemma}

\begin{proof}
Given the regularities stated in Lemma \ref{ch3_lemmeinverse} and the continuous embedding $\mathbf{H}^{2}(\mathcal{F}) \hookrightarrow \mathbf{L}^{\infty}(\mathcal{F})$, the only delicate point that has to be verified is $\Delta \tilde{Y}(\tilde{X}) \in \L^{\infty}(0,T;\mathbf{H}^{1}(\mathcal{F}))$. For that, let us consider the i-th component of $\Delta Y(X)$; We write
\begin{eqnarray*}
\Delta \tilde{Y}_i(X(\cdot,t),t) & = & \trace \left( \nabla^2 \tilde{Y}_i(X(\cdot,t),t) \right)
\end{eqnarray*}
with
\begin{eqnarray*}
\nabla^2 \tilde{Y}_i(\tilde{X}(\cdot,t),t) & = & \left( \nabla \left( \nabla \tilde{Y}_i(\tilde{X}(\cdot,t),t) \right)\right) \nabla \tilde{Y}(\tilde{X}(\cdot,t),t)\\
& = & \left( \nabla \left( \nabla \tilde{Y}_i(\tilde{X}(\cdot,t),t) - \I_{\R^3} \right)\right) \nabla \tilde{Y}(\tilde{X}(\cdot,t),t),
\end{eqnarray*}
and we apply Lemma \ref{ch3_lemmaGrubb} with $s=1$, $\mu = 0$ and $\kappa = 1$ to obtain
\begin{eqnarray*}
\| \Delta \tilde{Y}_i(\tilde{X}(\cdot,t),t) \|_{\mathbf{H}^{1}(\mathcal{F})} & \leq & C \| \nabla \tilde{Y}(\tilde{X}) -\I_{\R^3} \|_{\mathbf{H}^{2}(\mathcal{F})} \| \nabla \tilde{Y}(\tilde{X}) \|_{\mathbf{H}^{2}(\mathcal{F})}.
\end{eqnarray*}
\end{proof}

\begin{corollary} \label{ch3_lemmaH301}
There exists a positive constant $C$ such that for all $(V,Q,K',\varpi)$ in $\mathbb{H}_T$ we have
\begin{eqnarray*}
& &  \left\| \left(\Delta - \LL\right)V \right\|_{\L^2(0,T;\mathbf{L}^2(\mathcal{F}))}  \leq   C \sqrt{T}
\left\| V \right\|_{\L^2(0,T;\mathbf{H}^2(\mathcal{F}))} \times \nonumber  \\
& &  \left( \| \nabla \tilde{Y}(\tilde{X}) - \I_{\R^3} \|_{\H^{1}(0,T;\mathbf{H}^{2}(\mathcal{F}))}
\left(1 + \|  \nabla \tilde{Y}(\tilde{X})  \|_{\L^{\infty}(0,T;\mathbf{H}^{2}(\mathcal{F}))} \right) \right), \label{ch3_estLL} \\
& & \| \Delta \tilde{Y}(\tilde{X}(\cdot,t),t) \|_{\L^{\infty}(0,T;\mathbf{H}^{1}(\mathcal{F}))} \leq
C \sqrt{T}\| \nabla \tilde{Y}(\tilde{X}) -\I_{\R^3} \|_{\H^{1}(0,T;\mathbf{H}^{2}(\mathcal{F}))} \| \nabla \tilde{Y}(\tilde{X}) \|_{\L^{\infty}(0,T;\mathbf{H}^{2}(\mathcal{F}))},     \nonumber \\
& &  \left\| (\nabla - \GG)Q \right\|_{\L^2(0,T;\mathbf{L}^2(\mathcal{F}))} \leq C \sqrt{T}
\| \nabla \tilde{Y}(\tilde{X}) - \I_{\R^3} \|_{\H^{1}(0,T;\H^{2}(\mathcal{F}))}
\left\| \nabla Q \right\|_{\L^2(0,T;\mathbf{L}^2(\mathcal{F}))}. \nonumber  \label{ch3_estGG}
\end{eqnarray*}
\end{corollary}

\begin{proof}
Since $\nabla \tilde{Y}(\tilde{X}(\cdot,0),0) - \I_{\R^3} = 0$, we have
\begin{eqnarray*}
\| \nabla \tilde{Y}(\tilde{X}) - \I_{\R^3} \|_{\L^{\infty}(0,T;\mathbf{H}^{2}(\mathcal{F}))} & \leq &
\sqrt{T} \| \nabla \tilde{Y}(\tilde{X}) - \I_{\R^3} \|_{\H^{1}(0,T;\mathbf{H}^{2}(\mathcal{F}))}.
\end{eqnarray*}
The following quadratic term is treated as follows
\begin{eqnarray*}
\nabla \tilde{Y}(\tilde{X})\nabla \tilde{Y}(\tilde{X})^T - \I_{\R^3} & = & \left(\nabla \tilde{Y}(\tilde{X}) - \I_{\R^3}\right)\nabla \tilde{Y}(\tilde{X})^T + \left(\nabla \tilde{Y}(\tilde{X})^T - \I_{\R^3} \right).
\end{eqnarray*}
\end{proof}

\begin{lemma} \label{ch3_lemmaH302}
There exists a positive constant $C$ such that for all $(V,K',\varpi)$ in $\mathcal{U}(0,T;\mathcal{F}) \times \H^1(0,T;\R^3) \times \H^1(0,T;\R^3)$ we have
\begin{eqnarray}
\left\| \MM (V,K',\varpi) \right\|_{\L^2(0,T;\mathbf{L}^2(\mathcal{F}))} & \leq & CT^{1/10} \| \nabla  \tilde{Y}(\tilde{X}) \|_{\L^{\infty}(0,T;\mathbf{H}^{2}(\mathcal{F}))}
 \left\| K' + \varpi \wedge \tilde{X} + \frac{\p \tilde{X}}{\p t} \right\|_{\L^{\infty}(0,T;\mathbf{H}^1(\mathcal{F}))} \times \nonumber \\
 & &  \left\| V \right\|^{1/5}_{\L^{\infty}(0,T;\mathbf{H}^1(\mathcal{F}))}  \left\| V \right\|^{4/5}_{\L^{2}(0,T;\mathbf{H}^2(\mathcal{F}))},  \label{ch3_estiMM} \\
 \left\| \NN V \right\|_{\L^2(0,T;\mathbf{L}^2(\mathcal{F}))} & \leq & C T^{1/10} \| \nabla  \tilde{Y}(\tilde{X})\|_{\L^{\infty}(0,T;\mathbf{H}^{2}(\mathcal{F}))}
 \left\| V \right\|^{6/5}_{\L^{\infty}(0,T;\mathbf{H}^1(\mathcal{F}))}  \left\| V \right\|^{4/5}_{\L^{2}(0,T;\mathbf{H}^2(\mathcal{F}))}, \nonumber \\ \label{ch3_estiNN}  \\
 \| \varpi \wedge V \|_{\L^2(0,T;\mathbf{L}^2(\mathcal{F}))} & \leq & C\sqrt{T} \| \varpi \|_{\L^{\infty}(0,T;\R^3)}
 \| V \|_{\L^{\infty}(0,T;\mathbf{L}^2(\mathcal{F}))}. \label{ch3_estiOU}
\end{eqnarray}
\end{lemma}

\begin{proof}
Let us remind an estimate proved in \cite{Tucsnak} (Lemma 5.2) which is still true in dimension 3, \textcolor{black}{because we have the continuous embeddings:
\begin{eqnarray*}
\H^1(\mathcal{F}) & \hookrightarrow & \mathbf{L}^q(\mathcal{F}), \quad \forall \ 2 \leq q \leq 6.
\end{eqnarray*}
}
So there exists a positive constant $C$ such that for all $v$, $w$ in $\mathcal{U}(0,T;\mathcal{F})$ we have
\begin{eqnarray}
\|(w\cdot \nabla)v \|_{\L^{5/2}(0,T,\mathbf{L}^2(\mathcal{F}))} & \leq & C \|w \|_{\L^{\infty}(0,T,\mathbf{H}^1(\mathcal{F}))}
\|v \|^{1/5}_{\L^{\infty}(0,T,\mathbf{H}^1(\mathcal{F}))} \| v \|^{4/5}_{\L^{2}(0,T,\mathbf{H}^2(\mathcal{F}))}.\nonumber \\ \label{ch3_estTT}
\end{eqnarray}
By applying the estimate \eqref{ch3_estTT} with $v = V$ and $w = -\nabla \tilde{Y}(\tilde{X}) \left( H' + \Omega \wedge \tilde{X} + \displaystyle \frac{\p \tilde{X}}{\p t} \right)$, combined with the H\"{o}lder inequality which gives
\begin{eqnarray*}
 \|(w\cdot \nabla)v \|_{\L^{2}(0,T,\mathbf{L}^2(\mathcal{F}))} & \leq & T^{1/10}\|(w\cdot \nabla)v \|_{\L^{5/2}(0,T,\mathbf{L}^2(\mathcal{F}))},
\end{eqnarray*}
we get
\begin{eqnarray*}
& & \left\| \MM V \right\|_{\L^2(0,T;\mathbf{L}^2(\mathcal{F}))}  \leq \\
& & CT^{1/10}
\left\| \nabla  \tilde{Y}(\tilde{X}) \left( K' +\varpi \wedge \tilde{X} + \frac{\p \tilde{X}}{\p t} \right) \right\|_{\L^{\infty}(0,T;\mathbf{H}^1(\mathcal{F}))}
\left\| V \right\|^{1/5}_{\L^{\infty}(0,T;\mathbf{H}^1(\mathcal{F}))}  \left\| V \right\|^{4/5}_{\L^{2}(0,T;\mathbf{H}^2(\mathcal{F}))}.
\end{eqnarray*}
We apply Lemma \ref{ch3_lemmaGrubb} on $w$ with $s = 1$, $\mu = 1$ and $\kappa = 0$, and then we obtain \eqref{ch3_estiMM}. For the estimate \eqref{ch3_estiNN}, we proceed similarly; We use the inequality \eqref{ch3_estTT} with $v = V$ and $w = \nabla \tilde{Y}(\tilde{X})V$, and we apply Lemma \ref{ch3_lemmaGrubb} on $w$ with $s = 1$, $\mu = 1$ and $\kappa = 0$. For the estimate \eqref{ch3_estiOU}, we simply write
\begin{eqnarray*}
\| \varpi \wedge V \|_{\L^2(0,T;\mathbf{L}^2(\mathcal{F}))}
& \leq & C \| \varpi \|_{\L^{2}(0,T;\R^3)} \| V \|_{\L^{\infty}(0,T;\mathbf{L}^2(\mathcal{F}))} \\
& \leq & C\sqrt{T} \| \varpi \|_{\L^{\infty}(0,T;\R^3)} \| V \|_{\L^{\infty}(0,T;\mathbf{L}^2(\mathcal{F}))}.
\end{eqnarray*}
\end{proof}

\begin{lemma} \label{ch3_lemmaH3}
There exists a positive constant $C$ such that for all $(V,K',\varpi)$ in $\mathcal{U}(0,T;\mathcal{F}) \times \H^1(0,T;\R^3) \times \H^1(0,T;\R^3)$ we have
\begin{eqnarray*}
\left\| G(V,K',\varpi) \right\|_{\L^2(0,T;\mathbf{H}^2(\mathcal{F}))} & \leq & C\sqrt{T}
\|V \|_{\L^2(0,T;\mathbf{H}^2(\mathcal{F}))} \| \nabla \tilde{Y}(\tilde{X}) - \I_{\R^3} \|_{\H^{1}(0,T;\mathbf{H}^2(\mathcal{F}))}, \\
 \left\| G(V,K',\varpi) \right\|_{\H^1(0,T;\mathbf{L}^2(\mathcal{F}))} & \leq & C\sqrt{T}
 \left( \| V \|_{\H^1(0,T;\mathbf{L}^{2}(\mathcal{F}))} \|  \nabla \tilde{Y}(\tilde{X}) - \I_{\R^3} \|_{\H^{1}(0,T;\mathbf{H}^2(\mathcal{F}))} \right. \nonumber \\
  & &   + \left. \| V \|_{\L^2(0,T;\mathbf{H}^2(\mathcal{F}))}
  \| \nabla \tilde{Y}(\tilde{X}) - \I_{\R^3} \|_{\H^{2}(0,T;\mathbf{L}^2(\mathcal{F}))} \right),\\
  \text{\textcolor{black}{$\displaystyle 
\left\| G(V,K',\varpi) \right\|_{\L^{\infty}(0,T;\mathbf{H}^1(\mathcal{F}))} $}} 
& \text{\textcolor{black}{$\leq$}} & 
\text{\textcolor{black}{$C\sqrt{T} \|  \nabla \tilde{Y}(\tilde{X}) - \I_{\R^3} \|_{\H^{1}(0,T;\mathbf{H}^2(\mathcal{F}))}
\| V \|_{\L^{\infty}(0,T;\mathbf{H}^1(\mathcal{F}))}.
  $}}
\end{eqnarray*}
\textcolor{black}{If furthermore $V_{|\p \mathcal{S}} \in \H^1(0,T;\mathbf{H}^{\varepsilon}(\p \mathcal{S}))$, then $G(V,K',\varpi)_{| \p \mathcal{S}}$ lies also in this space and we have
\begin{eqnarray*}
& &  \| G(V,K',\varpi)_{| \p \mathcal{S}}\|_{ \H^1(0,T;\mathbf{H}^{\varepsilon}(\p \mathcal{S}))}  \leq   C T^{1/2-\varepsilon}
 \| \nabla \tilde{Y}(\tilde{X}) - \I_{\R^3} \|_{\H^{2}(0,T;\mathbf{L}^2(\mathcal{F}))\cap \H^1(0,T;\mathbf{H}^{2}(\mathcal{F}))}  \\
& & \hspace*{180pt} \times \left( \| V_{|\p \mathcal{S}} \|_{\H^1(0,T;\mathbf{H}^{\varepsilon}(\p \mathcal{S}))}
+ \| V \|_{\mathcal{U}(0,T;\mathcal{F})}
\right).
\end{eqnarray*}
}
\end{lemma}

\begin{proof}
Observe that $\nabla \tilde{Y}(\tilde{X})$ lies in $\H^1(0,T;\mathbf{H}^{2}(\mathcal{F}))$. We apply Lemma \ref{ch3_lemmaGrubb} with $s =2$, $\mu = 0$ and $\kappa = 0$, and we get
\begin{eqnarray*}
\left\|G(V,K',\varpi)(\cdot,t) \right\|_{\mathbf{H}^{2}(\mathcal{F})} & \leq & C \left\| V \right\|_{\mathbf{H}^{2}(\mathcal{F})}
\| \nabla \tilde{Y}(\tilde{X}(\cdot,t),t) - I \|_{\mathbf{H}^{2}(\mathcal{F})}, \\
\left\|G(V,K',\varpi) \right\|_{\L^2(0,T;\mathbf{H}^{2}(\mathcal{F}))} & \leq & C \left\| V \right\|_{\L^2(0,T;\mathbf{H}^{2}(\mathcal{F}))}
\| \nabla \tilde{Y}(\tilde{X}) - \I_{\R^3} \|_{\L^{\infty}(0,T;\mathbf{H}^{2}(\mathcal{F}))} \\
& \leq & C\sqrt{T} \left\| V \right\|_{\L^2(0,T;\mathbf{H}^{2}(\mathcal{F}))}
\| \nabla \tilde{Y}(\tilde{X}) - \I_{\R^3} \|_{\H^{1}(0,T;\mathbf{H}^{2}(\mathcal{F}))}.
\end{eqnarray*}
\textcolor{black}{The estimate in $\L^{\infty}(0,T;\mathbf{H}^1(\mathcal{F}))$ is obtained by applying \ref{ch3_lemmaGrubb} with $s =1$, $\mu = 1$ and $\kappa = 0$, and the arguments used above}. For proving the regularity in $\H^1(0,T;\mathbf{L}^2(\mathcal{F}))$, we first write
\begin{eqnarray}
\frac{\p G(V,K',\varpi)}{\p t}  & = & \left( \nabla \tilde{Y}(\tilde{X}) - \I_{\R^3} \right)\frac{\p V}{\p t}
+ \frac{\p }{\p t}\left( \nabla \tilde{Y}(\tilde{X}) \right) V. \label{useful}
\end{eqnarray}
\textcolor{black}{Let us keep in mind that we have the continuous embedding
\begin{eqnarray*}
\mathbf{H}^{2}(\mathcal{F}) \hookrightarrow 
\mathbf{L}^{\infty}(\mathcal{F}),
\end{eqnarray*}
and thus we have the estimate}
\begin{eqnarray*}
\left\| \frac{\p G(V,K',\varpi)}{\p t} \right\|_{\L^2(0,T;\mathbf{L}^2(\mathcal{F}))} & \leq & C \left( 
\| \nabla \tilde{Y}(\tilde{X}) - \I_{\R^3} \|_{\L^{\infty}(0,T;\mathbf{H}^{2}(\mathcal{F}))} \left\| \frac{\p V}{\p t} \right\|_{\L^2(0,T;\mathbf{L}^2(\mathcal{F}))} \right. \\
& & + \left.  \left\| \frac{\p }{\p t}\left( \nabla \tilde{Y}(\tilde{X}) \right) \right\|_{\L^{\infty}(0,T;\mathbf{L}^{2}(\mathcal{F}))}
\| V \|_{\L^{2}(0,T;\mathbf{H}^{2}(\mathcal{F}))}  \right).
\end{eqnarray*}
\textcolor{black}{For the estimate in $\H^1(0,T; \mathbf{H}^{\varepsilon}(\p \mathcal{S}))$, let us consider again the equality \eqref{useful}. We rather consider an extension of $\overline{V}$ of $V_{\p \mathcal{S}}$ in $\mathcal{F}$ such that
\begin{eqnarray*}
\| \overline{V} \|_{\H^1(0,T; \mathbf{H}^{\varepsilon +1/2}(\mathcal{F}))} & \leq & C\| V_{| \p \mathcal{S}} \|_{\H^1(0,T; \mathbf{H}^{\varepsilon}(\p \mathcal{S}))}, \\
\| \overline{V} \|_{\L^{\infty}(0,T; \mathbf{H}^{1}(\mathcal{F}))} & \leq & C\| V \|_{\L^{\infty}(0,T; \mathbf{H}^{1}(\mathcal{F}))}.
\end{eqnarray*}
We then set 
\begin{eqnarray*}
\overline{G}_t & = & \left( \nabla \tilde{Y}(\tilde{X}) - \I_{\R^3} \right)\frac{\p \overline{V}}{\p t}
+\frac{\p }{\p t}\left( \nabla \tilde{Y}(\tilde{X}) \right) \overline{V}
\end{eqnarray*}
which can be seen as an extension in $\mathcal{F}$ of $\frac{\p G(V,K',\varpi)}{\p t}_{| \p \mathcal{S}}$.\\
We then want to estimate $\overline{G}_t$ in $\L^2(0,T; \mathbf{H}^{\varepsilon +1/2}(\mathcal{F}))$. The first term in the definition of $\overline{G}_t$ can be treated by using Lemma \ref{ch3_lemmaGrubb} with $s = 1/2+\varepsilon$, $\mu = 3/2-\varepsilon$ and $\kappa = 0$, as follows
\begin{eqnarray*}
 \left\| \left( \nabla \tilde{Y}(\tilde{X}) - \I_{\R^3} \right)\frac{\p \overline{V}}{\p t} \right\|_{\L^2(0,T; \mathbf{H}^{\varepsilon +1/2}(\mathcal{F}))} & \leq &
 C \| \tilde{Y}(\tilde{X}) - \I_{\R^3} \|_{\L^{\infty}(0,T;\mathbf{H}^{2}(\mathcal{F}))} \| \overline{V} \|_{\H^1(0,T; \mathbf{H}^{\varepsilon +1/2}(\mathcal{F}))} \\
 & \leq & 
 C \sqrt{T}\| \tilde{Y}(\tilde{X}) - \I_{\R^3} \|_{\H^{1}(0,T;\mathbf{H}^{2}(\mathcal{F}))} \| \overline{V} \|_{\H^1(0,T; \mathbf{H}^{\varepsilon +1/2}(\mathcal{F}))}.
\end{eqnarray*}
For the other term, we first use Lemma \ref{ch3_lemmaGrubb} with $s = 1/2+\varepsilon$, $\mu = 1/2+\varepsilon$ and $\kappa = 1/2-\varepsilon$ to get
\begin{eqnarray*}
\left\|\frac{\p }{\p t}\left( \nabla \tilde{Y}(\tilde{X}) \right) \overline{V} \right\|_{\mathbf{H}^{1/2+\varepsilon}(\mathcal{F})} & \leq & C
\left\| \frac{\p }{\p t}\left( \nabla \tilde{Y}(\tilde{X}) \right) \right\|_{\mathbf{H}^{1+2\varepsilon}(\mathcal{F})} \| \overline{V} \|_{\mathbf{H}^{1}(\mathcal{F})}
\end{eqnarray*}
Hence from the H\"older inequality we deduce
\begin{eqnarray*}
\left\|\frac{\p }{\p t}\left( \nabla \tilde{Y}(\tilde{X}) \right) \overline{V} \right\|_{\L^2(0,T;\mathbf{H}^{1/2+\varepsilon}(\mathcal{F}))} & \leq & C
\left\| \frac{\p }{\p t}\left( \nabla \tilde{Y}(\tilde{X}) \right) \right\|_{\L^{2p}(0,T;\mathbf{H}^{1+2\varepsilon}(\mathcal{F}))} \| \overline{V} \|_{\L^{2q}(0,T;\mathbf{H}^{1}(\mathcal{F}))} \\
& \leq & C T^{1/2q}
\left\| \frac{\p }{\p t}\left( \nabla \tilde{Y}(\tilde{X}) \right) \right\|_{\L^{2p}(0,T;\mathbf{H}^{1+2\varepsilon}(\mathcal{F}))} 
\| \overline{V} \|_{\L^{\infty}(0,T;\mathbf{H}^{1}(\mathcal{F}))},
\end{eqnarray*}
with $1/p = 1/2 + \varepsilon$ and $1/q = 1/2 - \varepsilon$. Then by interpolation we estimate
\begin{eqnarray*}
\left\| \frac{\p }{\p t}\left( \nabla \tilde{Y}(\tilde{X}) \right) \right\|_{\L^{2p}(0,T;\mathbf{H}^{1+2\varepsilon}(\mathcal{F}))} & \leq & 
\left\| \frac{\p }{\p t}\left( \nabla \tilde{Y}(\tilde{X}) \right) \right\|_{\L^{2}(0,T;\mathbf{H}^{2}(\mathcal{F}))}^{1/2+\varepsilon}
\left\| \frac{\p }{\p t}\left( \nabla \tilde{Y}(\tilde{X}) \right) \right\|_{\L^{\infty}(0,T;\mathbf{L}^{2}(\mathcal{F}))}^{1/2-\varepsilon}.
\end{eqnarray*}
Because of the equality 
\begin{eqnarray*}
\nabla \tilde{Y}(\tilde{X}(\cdot,t),t) \nabla \tilde{X}(\cdot,t) & = & \I_{\R^3}
\end{eqnarray*}
and the fact that $\tilde{X}(\cdot,0) = \Id_{\mathcal{F}}$, we have $\displaystyle \frac{\p }{\p t}\left( \nabla \tilde{Y}(\tilde{X}) \right)(\cdot,0)  =0$ and so we can deduce
\begin{eqnarray*}
\left\| \frac{\p }{\p t}\left( \nabla \tilde{Y}(\tilde{X}) \right) \right\|_{\L^{\infty}(0,T;\mathbf{L}^{2}(\mathcal{F}))}^{1/2-\varepsilon}
& \leq &  T^{1/4-\varepsilon/2}
\left\| \frac{\p }{\p t}\left( \nabla \tilde{Y}(\tilde{X}) \right) \right\|_{\H^{1}(0,T;\mathbf{L}^{2}(\mathcal{F}))}^{1/2-\varepsilon},
\\
\left\| \frac{\p }{\p t}\left( \nabla \tilde{Y}(\tilde{X}) \right) \right\|_{\L^{2p}(0,T;\mathbf{H}^{1+2\varepsilon}(\mathcal{F}))}
& \leq & T^{1/4-\varepsilon/2}\left\| \frac{\p }{\p t}\left( \nabla \tilde{Y}(\tilde{X}) \right) \right\|_{\L^{2}(0,T;\mathbf{H}^{2}(\mathcal{F}))\cap \H^{1}(0,T;\mathbf{L}^{2}(\mathcal{F})) }.
\end{eqnarray*}
Finally we have
\begin{eqnarray*}
\left\|\frac{\p }{\p t}\left( \nabla \tilde{Y}(\tilde{X}) \right) \overline{V} \right\|_{\L^2(0,T;\mathbf{H}^{1/2+\varepsilon}(\mathcal{F}))} & \leq & 
C T^{1/2-\varepsilon} \| \overline{V} \|_{\L^{\infty}(0,T;\mathbf{H}^{1}(\mathcal{F}))}
\times \\
& & \left\| \frac{\p }{\p t}\left( \nabla \tilde{Y}(\tilde{X}) \right) \right\|_{\L^{2}(0,T;\mathbf{H}^{2}(\mathcal{F}))\cap \H^{1}(0,T;\mathbf{L}^{2}(\mathcal{F})) }.
\end{eqnarray*}
}
\end{proof}

\begin{lemma} \label{ch3_estFMFI}
There exists a positive constant $C$ such that for all $(V,Q,K',\varpi)$ in $\mathbb{H}_T$ we have
\begin{eqnarray*}
& & \left\| W(\varpi) \right\|_{\L^2(0,T;\mathbf{H}^{3/2}(\p \mathcal{S}))\cap \H^1(0,T;\mathbf{H}^{1/2}(\p \mathcal{S}))} \leq C\left(\sqrt{T}\|\varpi\|_{\H^1(0,T;\R^3)} + 1 \right)\| X^{\ast} - \Id_{\mathcal{S}}\|_{\mathcal{W}_0(0,\infty;\mathcal{S})}, \label{ch3_estW} \\
& & \left\| F_M(V,Q,K',\varpi) \right\|_{\L^2(0,T;\R^3)}  \leq \nonumber \\
& & C\left( \sqrt{T} \|K'\|_{\L^{\infty}(0,T;\R^3)} \|\varpi\|_{\L^{\infty}(0,T;\R^3)} +
\left( \|V \|_{\L^2(0,T;\mathbf{H}^2(\mathcal{F}))} + \|Q \|_{\L^2(0,T;\mathbf{H}^1(\mathcal{F}))} \right) \times \right. \nonumber \\
& &   \sqrt{T}\|  \nabla \tilde{Y}(\tilde{X}) - \I_{\R^3} \|_{\H^{1}(0,T;\mathbf{L}^{\infty}(\p \mathcal{S}))}
\left( \| \nabla \tilde{Y}(\tilde{X}) \|_{\L^{\infty}((0,T) \times \p \mathcal{S})} +1 \right)
 , \nonumber \\ \label{ch3_estFM} \\
& & \left\| F_I(V,Q,K',\varpi) \right\|_{\L^2(0,T;\R^3)}  \leq \nonumber \\
& & C \left( T \|{I^{\ast}}' \|_{\L^{\infty}(0,T;\R^9)} \| \varpi \|_{\H^1(0,T;\R^3)} \right. \nonumber \\
& & + \sqrt{T}\|{I^{\ast}}' \|_{\L^{\infty}(0,T;\R^9)} \| \varpi \|_{\L^{\infty}(0,T;\R^3)}
+ \sqrt{T}\| I^{\ast}\|_{\L^{\infty}(0,T;\R^9)} \| \varpi \|^2_{\L^{\infty}(0,T;\R^3)} \nonumber \\
& & +  \sqrt{T}\left( \| V \|_{\L^2(0,T;\mathbf{H}^2(\mathcal{F}))} + \|Q \|_{\L^2(0,T;\mathbf{H}^1(\mathcal{F}))} \right) \times \nonumber \\
& &  \left(1 + \| \nabla \tilde{Y}(\tilde{X}) \|_{\L^{\infty}((0,T) \times \p \mathcal{S})} \right)
\left(\|  \nabla \tilde{Y}(\tilde{X}) - \I_{\R^3} \|_{\H^{1}(0,T;\mathbf{L}^{\infty}(\p \mathcal{S}))} +
\|  \nabla X^{\ast} - \Id_{\p \mathcal{S}} \|_{\H^{1}(0,T;\mathbf{L}^{\infty}(\p \mathcal{S}))} \right)  , \nonumber
\end{eqnarray*}
\begin{eqnarray}
 \|{I^{\ast}}' \|_{\L^{\infty}(0,T;\R^9)} \leq C , & \quad & \| I^{\ast} - I_0 \|_{\L^{\infty}(0,T;\R^{9})} \leq  C T. \label{estmatrix}
\end{eqnarray}
\end{lemma}

\begin{proof}
For the first estimate, we write (for $m \geq 3$)
\begin{eqnarray*}
\| W(\varpi) \|_{\L^2(0,T;\mathbf{H}^{3/2}(\mathcal{F}))}
& \leq & C\|\varpi\|_{\L^2(0,T;\R^3)} \| X^{\ast} - \Id \|_{\L^{\infty}(0,T;\mathbf{H}^{3/2}(\p \mathcal{S}))}
+ C \left\| \frac{\p X^{\ast}}{\p t} \right\|_{\L^{2}(0,T;\mathbf{H}^{3/2}(\p \mathcal{S}))} \\
& \leq & C \sqrt{T} \|\varpi\|_{\L^2(0,T;\R^3)} \left\| \frac{\p X^{\ast}}{\p t} \right\|_{\L^{2}(0,T;\mathbf{H}^{3}(\mathcal{S}))}
+ C \left\| \frac{\p X^{\ast}}{\p t} \right\|_{\L^{2}(0,T;\mathbf{H}^{3}(\mathcal{S}))},
\end{eqnarray*}
and
\begin{eqnarray*}
\left\| \frac{\p W(\varpi)}{\p t} \right\|_{\L^2(0,T;\mathbf{H}^{1/2}(\p \mathcal{S}))}
& \leq & C\|\varpi'\|_{\L^2(0,T;\R^3)} \| X^{\ast} - \Id \|_{\L^{\infty}(0,T;\mathbf{H}^{1/2}(\p \mathcal{S}))} \\
&  & + C\|\varpi\|_{\L^{2}(0,T;\R^3)} \left\| \frac{\p X^{\ast}}{\p t} \right\|_{\L^{\infty}(0,T;\mathbf{H}^{1/2}(\p \mathcal{S}))}
+ \left\| \frac{\p^2 X^{\ast}}{\p t^2} \right\|_{\L^{2}(0,T;\mathbf{H}^{1/2}(\p \mathcal{S}))}, \\
& \leq & C\sqrt{T} \|\varpi'\|_{\L^2(0,T;\R^3)} \left\| \frac{\p X^{\ast}}{\p t} \right\|_{\L^{2}(0,T;\mathbf{H}^{1}(\mathcal{S}))} \\
&  & + C\sqrt{T} \|\varpi\|_{\L^{2}(0,T;\R^3)} \left\| \frac{\p^2 X^{\ast}}{\p t^2} \right\|_{\L^{2}(0,T;\mathbf{H}^{1}(\mathcal{S}))}
+ \left\| \frac{\p^2 X^{\ast}}{\p t^2} \right\|_{\L^{2}(0,T;\mathbf{H}^{1}(\mathcal{S}))}.
\end{eqnarray*}
There is no particular difficulty for proving the other two estimates, if we refer to the respective expressions of $F_M$ and $F_I$ given by \eqref{ch3_rhsFM} and \eqref{ch3_rhsFI}. However, let us detail the terms due to the inertia matrices. We have
\begin{eqnarray*}
{I^{\ast}}'(t) & = & \rho_{\mathcal{S}}\int_{\mathcal{S}} \left( 2\left(\frac{\p X^{\ast}}{\p t} \cdot X^{\ast} \right)\I_{\R^3} - \frac{\p X^{\ast}}{\p t} \otimes X^{\ast} - X^{\ast} \otimes \frac{\p X^{\ast}}{\p t}  \right)(y,t) \d y,
\end{eqnarray*}
and thus
\begin{eqnarray}
\left| {I^{\ast}}'(t) \right|_{\R^9}
& \leq & C \left\| \frac{\p X^{\ast}}{\p t}(\cdot,t) \right\|_{\mathbf{L}^2(\mathcal{S})} \| X^{\ast}(\cdot,t) \|_{\mathbf{L}^2(\mathcal{S})}, \label{supstarI} \\
\left\| {I^{\ast}}' \right\|_{\L^{\infty}(0,T;\R^9)} & \leq &
C \left\| \frac{\p X^{\ast}}{\p t} \right\|_{\L^{\infty}(0,T;\mathbf{L}^2(\mathcal{S}))}
\| X^{\ast} \|_{\L^{\infty}(0,T;\mathbf{L}^2(\mathcal{S}))}, \nonumber \\
\left\| {I^{\ast}} - I_0 \right\|_{\L^{\infty}(0,T;\R^9)} & \leq &  T \left\| {I^{\ast}}' \right\|_{\L^{\infty}(0,T;\R^9)}. \nonumber
\end{eqnarray}
\end{proof}

\subsubsection{The mapping $\mathcal{N}$ is well-defined}
From Remark \ref{remarkcc} and from the estimates of the previous subsection (see Corollary \ref{ch3_lemmaH301} and Lemmas \ref{ch3_lemmaH302}, \ref{ch3_lemmaH3}, \ref{ch3_estFMFI}), we can first claim that the assumptions of Proposition \ref{ch2_thsolfaible} are satisfied. Then from Proposition \ref{ch2_thsolfaible} the mapping $\mathcal{N}$ is well-defined. Moreover we have the following estimate
\begin{eqnarray}
\| \mathcal{N}(V,Q,K',\varpi) \|_{ \mathbb{H}_T} & \leq &
C^{(0)}_T\left(1 +\| F(V,Q,K',\varpi) \|_{\L^2(0,T;\mathbf{L}^2(\mathcal{F}))} + \| G(V,K',\varpi) \|_{\mathcal{U}(0,T;\mathcal{F})} \right. \nonumber \\
& & + \text{\textcolor{black}{$ \displaystyle \left\| G(V,K',\varpi)_{|\p \mathcal{S}} \right\|_{\H^1(0,T;\mathbf{H}^{\varepsilon}(\p \mathcal{S}))} $}} + \sqrt{T}\|\varpi\|_{\H^1(0,T;\R^3)} \nonumber \\
& &  \left. + \| F_M(V,Q,K',\varpi) \|_{\L^2(0,T;\R^3)} + \| F_I(V,Q,K',\varpi) \|_{\L^2(0,T;\R^3)} \right), \nonumber \\ \label{ch3_estNB}
\end{eqnarray}
where the constant $C^{(0)}_T$ is nondecreasing with respect to $T$, and depends on the data
\begin{eqnarray*}
\|u_0\|_{\mathbf{H}^1(\mathcal{F})}, \quad |h_1|_{\R^3}, \quad |\omega_0|_{\R^3}, \quad
\| X^{\ast} - \Id_{\mathcal{S}}\|_{\mathcal{W}_0(0,T;\mathcal{S})}.
\end{eqnarray*}
For $R>0$, we set the ball
\begin{eqnarray*}
& & \mathcal{B}_R = \left\{ (U,P,H',\Omega) \in  \mathbb{H}_T \mid 
\| (U,P,H',\Omega) \|_{\mathbb{H}_T}
\leq R  \right\}
\end{eqnarray*}
which is clearly a closed subset of $\mathbb{H}_T$. The rest of this section is devoted to proving that for $R$ large enough and $T$ small enough the ball $B_R$ is stable by $\mathcal{N}$, and $\mathcal{N}$ is a contraction in $B_R$.

\subsubsection{Stability of the set $B_R$ by the mapping $\mathcal{N}$}
We are in position to claim that for $R$ large enough and $T$ small enough the ball $\mathcal{B}_R$ is stable by $\mathcal{N}$.

\begin{lemma} \label{ch3_estFFF}
Let us assume that $T\leq 1$ and $R\geq 1$. There exists a positive constant $C$, which does not depend on $T$ or $R$, such that for $(V,Q,K',\varpi) \in B_R$ we have
\begin{eqnarray*}
\left\| F(V,P,K',\varpi) \right\|_{\L^2(0,T;\mathbf{L}^2(\mathcal{F}))} & \leq & CT^{1/10}R^3 ,\\
\left\| G(V,K',\varpi) \right\|_{\mathcal{U}(0,T;\mathcal{F})} & \leq & C\sqrt{T}R^2, \\
\text{\textcolor{black}{$\displaystyle
\left\| G(V,K',\varpi)_{|\p \mathcal{S}} \right\|_{\H^1(0,T;\mathbf{H}^{\varepsilon}(\p \mathcal{S}))}$}} & \leq & \text{\textcolor{black}{$\displaystyle CT^{1/2-\varepsilon}R^2,$}} \\
\left\| F_M(V,P,K',\varpi) \right\|_{\L^2(0,T;\R^3)} & \leq & C \sqrt{T}R^3, \\
\left\| F_I(V,P,K',\varpi) \right\|_{\L^2(0,T;\R^3)} & \leq &  C \sqrt{T}R^3.
\end{eqnarray*}
\end{lemma}

\begin{proof}
These estimates follow from Corollary \ref{ch3_lemmaH301} and Lemmas \ref{ch3_lemmaH302}, \ref{ch3_lemmaH3}, \ref{ch3_estFMFI} combined with the estimates \eqref{estlip0} and \eqref{ch3_eqtrick} (given in Appendix A).
\end{proof}

By combining Lemma \ref{ch3_estFFF} and the estimate \eqref{ch3_estNB}, we have for $R$ large enough ($R > C_T^{(0)}$) and $T$ small enough
\begin{eqnarray*}
\mathcal{N}\left(B_R\right) & \subset & B_R.
\end{eqnarray*}

\subsubsection{Lipschitz stability for the mapping $\mathcal{N}$}
Let $(V_1,P_1,K_1',\varpi_1)$ and $(V_2,P_2,K_2',\varpi_2)$ be in $B_R$. We set
\begin{eqnarray*}
(U_1,P_1,H_1',\Omega_1)  =  \mathcal{N}(V_1,Q_1,K_1',\varpi_1), & \quad &
(U_2,P_2,H_2',\Omega_2)  =  \mathcal{N}(V_2,Q_2,K_2',\varpi_2),
\end{eqnarray*}
and
\begin{eqnarray*}
\begin{array} {ccccccc}
U = U_2 - U_1 ,& \quad & P = P_2 - P_1, & \quad & H' = H_2' - H_1', & \quad & \Omega = \Omega_2 - \Omega_1, \\
V = V_2 - V_1 ,& \quad & Q = Q_2 - Q_1, & \quad & K' = K_2' - K_1', & \quad & \varpi = \varpi_2 - \varpi_1.
\end{array}
\end{eqnarray*}
We also denote by $\tilde{X}_1, \ \nabla \tilde{Y}_1(\tilde{X}_1)$ the mappings provided by Lemma \ref{ch3_lemmaxtension} with $(K'_1,\varpi_1,X^{\ast})$ as data, and similarly $\tilde{X}_2, \ \nabla \tilde{Y}_2(\tilde{X}_2)$ the mappings provided by $(K'_2,\varpi_2,X^{\ast})$. \\
The quadruplet $(U, P, H', \Omega)$ satisfies the system
\begin{eqnarray*}
\frac{\p U}{\p t}  -\nu \Delta U + \nabla P
 =  \overline{F}, & \quad & \text{in } \mathcal{F}\times (0,T), \\
\div \ U  =  \div \ \overline{G}, & \quad & \text{in } \mathcal{F}\times (0,T),
\end{eqnarray*}
\begin{eqnarray*}
U = 0 , & \quad & \text{in } \p \mathcal{O}\times (0,T),  \\
U  =  H'(t) + \Omega (t) \wedge y + \overline{W} , & \quad & (y,t)\in \p \mathcal{S}\times (0,T),
\end{eqnarray*}
\begin{eqnarray*}
M H''  & = & - \int_{\p \mathcal{S}} \sigma(U,P) n  \d \Gamma + \overline{F}_M, \ \text{in } (0,T)   \\
I_0\Omega' (t) & = &  -  \int_{\p \mathcal{S}} y \wedge \sigma(U,P) n  \d \Gamma + \overline{F}_I, \ \text{in } (0,T)
\end{eqnarray*}
\begin{eqnarray*}
U(y,0)  = 0, \  \text{in } \mathcal{F} , \quad H'(0)=0 \in \R^3 ,\quad \Omega(0) = 0 \in \R^3 ,
\end{eqnarray*}
with
\begin{eqnarray*}
\overline{F} & = & F(V_2,Q_2,K'_2,\Omega_2) - F(V_1,Q_1,K'_1,\Omega_1), \\
\overline{G} & = & G(V_2,K'_2,\varpi_2) - G(V_1,K'_1,\varpi_1),\\
\overline{W} & = & W(\varpi_2) - W(\varpi_1) = \varpi \wedge (X^{\ast} - \Id_{\mathcal{S}}),  \\
\overline{F}_M & = & F_M(V_2,Q_2,K'_2,\varpi_2) - F_M(V_1,Q_1,K'_1,\varpi_1), \\
\overline{F}_I & = & F_I(V_2,Q_2,K'_2,\varpi_2) - F_I(V_1,Q_1,K'_1,\varpi_1).
\end{eqnarray*}
In particular, Proposition \ref{ch2_thsolfaible} provides for this nonhomogeneous linear system the following estimate
\begin{eqnarray}
\| (U,P,H',\Omega) \|_{\mathbb{H}_T} & \leq & C_T^{(0)}
\left( \| \overline{F} \|_{\L^2(0,T;\R^3)} + \| \overline{G} \|_{\mathcal{U}(0,T;\mathcal{F})}
+ \text{\textcolor{black}{$\displaystyle
\left\| \overline{G}_{|\p \mathcal{S}} \right\|_{\H^1(0,T;\mathbf{H}^{\varepsilon}(\p \mathcal{S}))}$}}
 \right. \nonumber \\
& & + \| \overline{W} \|_{\H^1(0,T;\mathbf{H}^{3/2}(\p \mathcal{S}))}  \left. +\| \overline{F}_M \|_{\L^2(0,T;\R^3)} + \| \overline{F}_I \|_{\L^2(0,T;\R^3)} \right). \nonumber \\ \label{ch3_estNB21}
\end{eqnarray}
Notice that the right-hand-sides $\overline{F}$, $\overline{G}$, $\overline{F}_M$ and $\overline{F}_I$ can be written as polynomial differential forms, multiplicative of one of the quantities
\begin{eqnarray*}
V, \quad Q, \quad K', \quad \varpi, \quad (\tilde{X}_2 - \tilde{X}_1),
\quad \left(\nabla \tilde{Y}_2(\tilde{X}_2) - \nabla \tilde{Y}_1(\tilde{X}_1)\right).
\end{eqnarray*}
For instance, the nonhomogeneous divergence condition $\overline{G}$ can be written as
\begin{eqnarray*}
\overline{G} & = & \left( \nabla \tilde{Y}_2(\tilde{X}_2) - \nabla \tilde{Y}_1(\tilde{X}_1) \right) V_2 + (\nabla \tilde{Y}_1(\tilde{X}_1)- \I_{\R^3})V.
\end{eqnarray*}
We have in particular
\begin{eqnarray*}
\tilde{X}_2(\cdot, 0) - \tilde{X}_1(\cdot,0)  =  0, & \quad &
\nabla \tilde{Y}_2 ( \tilde{X}_2(\cdot,0),0) - \nabla \tilde{Y}_1 ( \tilde{X}_1(\cdot,0),0)  =  0.
\end{eqnarray*}
The mapping $\nabla \tilde{Y}_2(\tilde{X}_2) - \nabla \tilde{Y}_1(\tilde{X}_1)$ satisfies the estimate \eqref{ch3_eqtrickbis} stated in Lemma \ref{ch3_lemmeinverse}, which is useful in order to make $\mathcal{N}$ a contraction. More specifically, the estimates \eqref{estlip12} and \eqref{ch3_eqtrickbis} are rewritten as
\begin{eqnarray*}
\| \tilde{X}_2 - \tilde{X}_1 \|_{\H^{1}(\mathbf{H}^{3}) \cap \H^{2}(\mathbf{H}^1)} & \leq & \tilde{C} \left(
\|K'\|_{\H^1(0,T_0;\R^3)} + \| \varpi \|_{\H^1(0,T_0;\R^3)} \right), \label{ch3_superX21} \\
\| \nabla \tilde{Y}(\tilde{X})_2 - \nabla \tilde{Y}_1(\tilde{X}_1) \|_{\H^{1}(\mathbf{H}^{2}) \cap \H^{2}(\mathbf{L}^2)} & \leq & \tilde{C} \left(
\| K' \|_{\H^1(0,T_0;\R^3)} + \| \varpi \|_{\H^1(0,T_0;\R^9)} \right).  \label{ch3_superY21}
\end{eqnarray*}
Then we state the following result, which can be proven with the same techniques that have been used for obtaining Lemma \ref{ch3_estFFF}.

\begin{lemma}
For $R$ large enough and $T$ small enough, there exists a positive constant $\tilde{C}$ - which does not depend on $T$ or $R$ - such that
\begin{eqnarray*}
\left\| \overline{F} \right\|_{\L^2(0,T;\mathbf{L}^2(\mathcal{F}))} & \leq & \tilde{C} T^{1/10} R^2 \| (V , Q, K', \varpi)\|_{\mathbb{H}_T}, \\
\left\| \overline{G} \right\|_{\mathcal{U}(0,T;\mathcal{F})} & \leq & \tilde{C} \sqrt{T} R
\left( \| V \|_{\mathcal{U}(0,T;\mathcal{F})} + \| K' \|_{\H^1(0,T;\R^3)} + \| \varpi \|_{\H^1(0,T;\R^3)} \right), \\
\text{\textcolor{black}{$\displaystyle 
\left\| \overline{G}_{|\p \mathcal{S}} \right\|_{\H^1(0,T;\mathbf{H}^{\varepsilon}(\p \mathcal{S}))}$}} & \leq & 
\text{\textcolor{black}{$\displaystyle 
\tilde{C} T^{1/2-\varepsilon} R
\left( \| V \|_{\mathcal{U}(0,T;\mathcal{F})} 
+ \| V_{| \p \mathcal{S}} \|_{\H^1(0,T;\mathbf{H}^{\varepsilon}(\p \mathcal{S}))}
\right.$}} \\
& & \text{\textcolor{black}{$\displaystyle 
\left.+ \| K' \|_{\H^1(0,T;\R^3)} + \| \varpi \|_{\H^1(0,T;\R^3)} \right),
$}} \\
\left\| \overline{W} \right\|_{\H^1(0,T;\mathbf{L}^{3/2}(\p \mathcal{S}))} & \leq &
\tilde{C} \sqrt{T} \| \varpi \|_{\H^1(0,T;\R^3)}, \\
\left\| \overline{F}_M \right\|_{\L^2(0,T;\R^3)} & \leq & \tilde{C} \sqrt{T}R^2 \|( V , Q, K', \varpi)\|_{\mathbb{H}_T}, \\
\left\| \overline{F}_I \right\|_{\L^2(0,T;\R^3)} & \leq & \tilde{C} \sqrt{T}R^2 \|( V , Q, K', \varpi)\|_{\mathbb{H}_T}.
\end{eqnarray*}
\end{lemma}

With regards to the estimate \eqref{ch3_estNB21}, we deduce from this lemma that for $T$ small enough the mapping $\mathcal{N}$ is a contraction in $B_R$. \textcolor{black}{Then Theorem \ref{ch3_thlocex} is proven.} 

\section{Global existence of strong solutions} \label{secglobale}
\subsection{Statement}
\begin{theorem} \label{ch3_existglobale}
Assume that the hypotheses in Theorem \ref{ch3_thlocex} hold true. Assume that $\|u_0\|_{\mathbf{H}^1(\mathcal{F})}$, $|h_1|_{\R^3}$ and $|\omega_0|_{\R^3}$ are small enough, and that the displacement $X^{\ast}- \Id_{\mathcal{S}}$ is small enough in $\mathcal{W}_0(0,\infty;\mathcal{S})$. \textcolor{black}{Let us still denote by $T_0$ the maximal time of existence of the local-in-time strong solution provided by Theorem \ref{ch3_thlocex}. Then the following alternative holds:
\begin{eqnarray*}
\begin{array} {ll}
\textrm{(a)} & \textrm{Either $T_0 = +\infty$ (that is to say the solution is global in time)} \\
\textrm{(b)} & \textcolor{black}{\textrm{or $\displaystyle \lim_{t \rightarrow T_0} \dist(\mathcal{S}(t), \p \mathcal{O}) = 0$.}}
\end{array}
\end{eqnarray*}
}
\end{theorem}

For proving this theorem, let us \textcolor{black}{proceed by contradiction}. Assume that \textcolor{black}{the existence time $T_0$ of Theorem \ref{ch3_thlocex} is finite, while there exists $\eta$ such that
\begin{eqnarray*}
\dist(\mathcal{S}(t), \p \mathcal{O}) \geq \eta > 0, & \quad & t\in [0, T_0).
\end{eqnarray*}
} Let us show then that the functions
\begin{displaymath}
t \mapsto \|u(t)\|_{\mathbf{H}^1(\mathcal{F}(t))}, \quad t \mapsto |h'(t)|, \quad
t\mapsto |\omega(t)|
\end{displaymath}
are bounded in $[0,T_0)$. For that, let us give some first results.\\

\subsection{Preliminary results}

\begin{lemma} \label{Xwstar}
Let $X^{\ast} - \Id_{\mathcal{S}} \in \mathcal{W}_0(0,T;\mathcal{S})$ such that for all $t\in [0,T)$ the mapping $X^{\ast}(\cdot,t)$ is a $C^1$-diffeomorphism from $\mathcal{S}$ onto $\mathcal{S}^{\ast}(t)$. Then the function defined by
\begin{eqnarray*}
w^{\ast}(x^{\ast},t) & = & \frac{\p X^{\ast}}{\p t}(Y^{\ast}(x^{\ast},t),t),  \quad x^{\ast} \in \mathcal{S}^{\ast}(t), \quad t \in [0,T),
\end{eqnarray*}
satisfies
\begin{eqnarray*}
w^{\ast} & \in & \L^2(0,T;\mathbf{H}^3(\mathcal{S}^{\ast}(t))) \cap \H^1(0,T;\mathbf{H}^1(\mathcal{S}^{\ast}(t))).
\end{eqnarray*}
Moreover, $\| w^{\ast} \|_{\L^2(0,T;\mathbf{H}^3(\mathcal{S}^{\ast}(t))) \cap \H^1(0,T;\mathbf{H}^1(\mathcal{S}^{\ast}(t)))}$ is an increasing function of
\begin{eqnarray*}
\left\| \frac{\p X^{\ast}}{\p t} \right\|_{\L^2(\mathbf{H}^3(\mathcal{S})) \cap \H^1(\mathbf{H}^1(\mathcal{S}))}, \quad
\| \nabla Y^{\ast}(X^{\ast}) \|_{\L^{\infty}(\mathbf{H}^2(\mathcal{S}))}, \quad
\| \det \nabla X^{\ast}(\cdot,t) \|_{\L^{\infty}(\mathbf{L}^{\infty}(\mathcal{S}))},
\end{eqnarray*}
and tends to $0$ when $\displaystyle \left\| \frac{\p X^{\ast}}{\p t} \right\|_{\L^2(\mathbf{H}^3(\mathcal{S})) \cap \H^1(\mathbf{H}^1(\mathcal{S}))}$ goes to $0$.
\end{lemma}

The proof of this lemma is given in Appendix B. The aim of this lemma is to show that by assuming smallness on $\|X^{\ast} - \Id_{\mathcal{S}} \|_{\mathcal{W}_0(0,\infty;\mathcal{S})}$, we impose automatically smallness on the velocity $w^{\ast}$ in $\L^2(0,T;\mathbf{H}^3(\mathcal{S}^{\ast}(t))) \cap \H^1(0,T;\mathbf{H}^1(\mathcal{S}^{\ast}(t)))$. Thus in the proof of Theorem \ref{ch3_existglobale} it is sufficient to consider that $w^{\ast}$ is small enough in $\L^2(0,T;\mathbf{H}^3(\mathcal{S}^{\ast}(t))) \cap \H^1(0,T;\mathbf{H}^1(\mathcal{S}^{\ast}(t)))$ for all $T>0$.\\

\begin{proposition} \label{ch2_zprop1}
Let $(u,p,h,\omega)$ be a strong solution of the system \eqref{prems}--\eqref{ders} defined on $[0,T_0)$ with $T_0>0$. Furthermore assume that there exists $\eta>0$ such that for all $t\in [0,T_0)$
\begin{eqnarray*}
\dist (\mathcal{S}(t),\p \mathcal{O}) \geq \eta .
\end{eqnarray*}
Then there exists a positive constant $K$ (depending on $T_0$ and $\eta$) such that
\begin{eqnarray*}
\|u\|_{\L^{\infty}(0,T_0;\mathbf{L}^2(\mathcal{F}(t)))}+\|u\|_{\L^2(0,T_0;\mathbf{H}^1(\mathcal{F}(t)))}
+\|h'\|_{\L^{\infty}(0,T_0;\R^3)} + \|\omega\|_{\L^{\infty}(0,T_0;\R^3)} & \leq & K C_0^2,
\end{eqnarray*}
with
\begin{eqnarray*}
\begin{array} {l}
\displaystyle  C_0  :=
\exp \left(K  \left(
\|w^{\ast}\|^2_{\L^2(0,T_0;\mathbf{H}^3(\mathcal{S}^{\ast}(t)))} 
+ \| {I^{\ast}}' \|_{\L^2(0,T_0;\R^9)}^2
\right) \right) \times \\
 \displaystyle \left( \|u_0\|_{\mathbf{L}^2(\mathcal{F})}^2 + |h_1|^2 + |\omega_0|^2 +
\| w^{\ast} \|^2_{\H^1(0,T_0;\mathbf{H}^1(\mathcal{S}^{\ast}(t)))} \left(1+\| w^{\ast} \|^2_{\L^2(0,T_0;\mathbf{H}^3(\mathcal{S}^{\ast}(t)))} \right)
 \right)^{1/2}.
\end{array}
\end{eqnarray*}
\end{proposition}

\begin{proof}
We need to define an extension to $\mathcal{F}(t)$ of the velocity $w(\cdot,t)$ which is define on $\mathcal{S}(t)$. For that, let us first define an extension to \textcolor{black}{$\tilde{\mathcal{F}}(t) = \mathcal{O} \setminus \overline{\mathcal{S}^{\ast}(t)}$} of the velocity $w^{\ast}(\cdot,t)$, defined on $\mathcal{S}^{\ast}(t)$ and whose the regularity has been given in Lemma \ref{Xwstar}; This extension is denoted $\overline{w}^{\ast}(\cdot,t)$ and is chosen as solution of the following divergence problem
\begin{eqnarray}
\left\{\begin{array} {lcl}
\div \ \overline{w}^{\ast}  =  0 & \ &  \text{ in } \R^3 \setminus \overline{\mathcal{S}^{\ast}(t)}, \ t\in (0,T_0), \\
\overline{w}^{\ast}(x^{\ast},t)  =  0 & \ & \text{ if } \dist(x^{\ast}, \mathcal{S}^{\ast}(t)) \geq \eta > 0, \ t\in (0,T_0), \\
\overline{w}^{\ast}(x^{\ast},t)  =  w^{\ast}(x^{\ast},t) & \ & \text{ if } x^{\ast} \in \mathcal{S}^{\ast}(t), \ t\in (0,T_0).
\end{array}\right. \label{pbdiridiv}
\end{eqnarray}
A solution of this problem can be obtained by using some results of \cite{Galdi1} for instance: The nonhomogeneous Dirichlet condition can be lifted (see Theorem 3.4, Chapter II) and the resolution made by using Exercise 3.4 and Theorem 3.2 of Chapter III. Then this extension $\overline{w}^{\ast}$ of the datum $w^{\ast}$ obeys the following estimate:
\begin{eqnarray}
\| \overline{w}^{\ast}(\cdot ,t) \|_{\mathbf{H}^3(\tilde{\mathcal{F}}(t))}
 & \leq & C \| w^{\ast}(\cdot,t) \|_{\mathbf{H}^{5/2}(\p \mathcal{S}^{\ast}(t))}
 \leq  \tilde{C} \| w^{\ast}(\cdot,t) \|_{\mathbf{H}^{3}(\mathcal{S}^{\ast}(t))}, \label{estwwast1}
 \end{eqnarray}
 \textcolor{black}{On the boundary $\p \mathcal{S}^{\ast}(t)$ the equality $\overline{w}^{\ast} = w^{\ast}$ derived in time is written with the material derivative as
 \begin{eqnarray*}
 \frac{\p \overline{w}^{\ast}}{\p t} + (w^{\ast}\cdot \nabla)\overline{w}^{\ast}  = 
  \frac{\p w^{\ast}}{\p t} + (w^{\ast}\cdot \nabla)w^{\ast}, & \quad & x^{\ast} \in \p \mathcal{S}^{\ast}(t). 
 \end{eqnarray*}
 Hence from the system \eqref{pbdiridiv} derived in time we deduce the estimate:
 \begin{eqnarray}
 \left\| \frac{\p \overline{w}^{\ast}}{\p t}(\cdot,t) \right\|_{\mathbf{H}^1(\tilde{\mathcal{F}}(t))}
 & \leq & C \left(
 \left\| \frac{\p w^{\ast}}{\p t}(\cdot,t) \right\|_{\mathbf{H}^{1}( \mathcal{S}^{\ast}(t))} \right. \nonumber \\
 & & \left.+ \| w^{\ast}(\cdot,t) \|_{\mathbf{W}^{1,\infty}( \mathcal{S}^{\ast}(t))}\left( \| w^{\ast}(\cdot,t) \|_{\mathbf{H}^{1}( \mathcal{S}^{\ast}(t))}
 + \| \overline{w}^{\ast}(\cdot,t) \|_{\mathbf{H}^{1}( \tilde{\mathcal{F}}(t))} \right)
 \right) \nonumber \\
 & \leq & \tilde{C} \left(  \left\| \frac{\p w^{\ast}}{\p t}(\cdot,t) \right\|_{\mathbf{H}^{1}(\mathcal{S}^{\ast}(t))} +
 \| w^{\ast}(\cdot,t) \|_{\mathbf{H}^{3}( \mathcal{S}^{\ast}(t))}^2
 \right).
 \label{estwwast2}
\end{eqnarray}
}
The constant $\tilde{C}$ does not depend on time\footnote{\textcolor{black}{This constant is nondecreasing with respect to the size of the domain, and so can be considered as the one associated with the whole domain $\mathcal{O}$.}}, since we have assumed that $\dist (\mathcal{S}(t),\p \mathcal{O}) \geq \eta > 0$ for all $t \in [0,T_0)$. Then we set as an extension of $w$ in $\mathcal{F}(t)$:
\begin{eqnarray*}
\overline{w}(x,t) & = & \mathbf{R}(t)\overline{w}^{\ast}(\mathbf{R}(t)^T(x-h(t)),t),  \quad  x\in \mathcal{F}(t).
\end{eqnarray*}
This relation yields the following properties
\begin{eqnarray*}
\left\{\begin{array} {rcl}
\div \ \overline{w}  =  0 & \ &  \text{ in } \mathcal{F}(t), \ t\in (0,T_0), \\
\overline{w}  =  0 & \ & \text{ on } \p \mathcal{O}, \ t\in (0,T_0), \\
\overline{w}  =  w & \ & \text{ on } \p \mathcal{S}(t), \ t\in (0,T_0),
\end{array}\right.
\end{eqnarray*}
and the following estimates, for some positive constant $C$ independent of time
\begin{eqnarray}
\| \overline{w} \|_{\L^2(0,T;\mathbf{L}^2(\mathcal{F}(t)))} & = & \| \overline{w}^{\ast} \|_{\L^2(0,T;\mathbf{L}^2(\tilde{\mathcal{F}}(t)))}, \nonumber \\
\| \nabla \overline{w} \|_{\mathbf{L}^2(\mathcal{F}(t))} & \leq & C\| \nabla \overline{w}^{\ast} \|_{\mathbf{L}^2(\tilde{\mathcal{F}}(t))}, \label{estwwbar1} \\
\| (\overline{w} \cdot \nabla ) \overline{w} \|_{\mathbf{L}^2(\mathcal{F}(t))} & \leq &
C\| \overline{w}^{\ast} \|_{\mathbf{H}^3(\tilde{\mathcal{F}}(t))} \| \overline{w}^{\ast} \|_{\mathbf{H}^1(\tilde{\mathcal{F}}(t))}, \\
\| \overline{w} \|_{\mathbf{W}^{1,\infty}(\mathcal{F}(t))} & \leq & C \| \overline{w}^{\ast} \|_{\mathbf{H}^3(\tilde{\mathcal{F}}(t))}, \\
\left\|\frac{\p \overline{w}}{\p t} \right\|_{\mathbf{L}^2(\mathcal{F}(t))} & \leq &
C \left( \left\|\frac{\p \overline{w}^{\ast}}{\p t} \right\|_{\mathbf{L}^2(\tilde{\mathcal{F}}(t))} +
\| \overline{w}^{\ast} \|_{\mathbf{H}^1(\tilde{\mathcal{F}}(t))} ( |h'|_{\R^3} + |\omega |_{\R^3})
\right). \label{estwwbar4}
\end{eqnarray}
Let us now set $v=u-\overline{w}$. The function $v$ satisfies the following system
\begin{eqnarray}
\frac{\p v}{\p t} + (u\cdot \nabla )v - \nu \Delta u + \nabla p & = &
-(v\cdot \nabla)\overline{w} - (\overline{w}\cdot \nabla)\overline{w} - \frac{\p \overline{w}}{\p t}, \quad x \in \mathcal{F}
(t),\quad t\in (0,T), \nonumber \\ \label{ch2_z0}\\
\div \ u  & = & 0 ,  \quad  x\in \mathcal{F} (t), \quad t\in (0,T),
\end{eqnarray}
\begin{eqnarray}
v = 0 , & \quad & x\in \p \mathcal{O} , \  t\in (0,T),  \\
v = h'(t) + \omega (t)\wedge(x-h(t)) , & \quad &  x\in \ \p \mathcal{S} (t), \ t\in (0,T),
\end{eqnarray}
\begin{eqnarray}
M h''(t)  =  - \int_{\p \mathcal{S}(t)} \sigma(u,p) n \d \Gamma , & \quad & t\in (0,T), \label{ch2_z1} \\
\left(I\omega\right)' (t) = - \int_{\p \mathcal{S}(t)} (x-h(t))\wedge \sigma(u,p) n \d \Gamma , & \quad & t\in (0,T), \label{ch2_z2}
\end{eqnarray}
\begin{eqnarray}
h(0) = h_0 \in \R^3 ,\quad h'(0)=h_1 \in \R^3 ,\quad \omega(0) = \omega_0 \in
\R^3 ,
\end{eqnarray}
\begin{eqnarray}
v(x,0)  =  v_0(x) := u_0 (x)-\overline{w}(x,0), \quad  x\in \mathcal{F}. \label{ch2_z4}
\end{eqnarray}
\textcolor{black}{In the equation \eqref{ch2_z0} we take the inner product with $v$ and we integer on $\mathcal{F}(t)$ to get \\
$\displaystyle \int_{\mathcal{F}(t)} \left( \frac{\p v}{\p t} + (u\cdot
\nabla)v \right)\cdot v \ \d x -
\int_{\mathcal{F}(t)}  \div \left(\sigma(u,p)\right)\cdot  v \ \d x $ \\
\begin{eqnarray}
& = &\int_{\mathcal{F}(t)}f\cdot v\ \d x - \int_{\mathcal{F}(t)} \left( (v\cdot \nabla)\overline{w}
\right)\cdot v \ \d x -
\int_{\mathcal{F}(t)} \left( (\overline{w}\cdot \nabla)\overline{w} \right)\cdot v \ \d x -
\int_{\mathcal{F}(t)} \frac{\p \overline{w}}{\p t}\cdot v \ \d x. \nonumber \\ \label{ch2_z3}
\end{eqnarray}
On one hand, we have by using the Reynolds transport theorem
\begin{eqnarray*}
\int_{\mathcal{F}(t)} \left( \frac{\p v}{\p t} + (u\cdot \nabla)v \right)\cdot v \
\d x & = & \frac{1}{2} \frac{\d}{\d t} \left(\int_{\mathcal{F}(t)}|v|_{\R^3}^2 \d x\right) .
\end{eqnarray*}
On the other hand, since $\div \ v = 0$, we have
\begin{eqnarray*}
\div \left(\sigma(u,p)\right)\cdot  v  & = &  \ \div \left(\sigma(u,p)v\right) - 2
\nu D(u) : D(v),
\end{eqnarray*}
which implies - by using the divergence formula and the fact that $v$ is equal to
$0$ on $\p \mathcal O$ - that
\begin{eqnarray*}
\int_{\mathcal{F}(t)}  \div \left(\sigma(u,p)\right)\cdot  v
& = & \int_{\p \mathcal{S}(t)} \sigma(u,p)v\cdot n\d \Gamma - 2\nu\int_{\mathcal{F}(t)}
D(u):D(v) \\
& = & -\int_{\p \mathcal{S}(t)} \sigma(u,p)v\cdot n\d \Gamma - 2\nu\int_{\mathcal{F}(t)}
D(\overline{w}):D(v)-
2\nu\int_{\mathcal{F}(t)} |D(v)|_{\R^9}^2 ;
\end{eqnarray*}
And yet on $\p \mathcal{S}(t)$ we have $v(t) = h'(t) +\omega(t)\wedge(x-h(t))$. Thus
\begin{eqnarray*}
\int_{\mathcal{F}(t)}  \div \left(\sigma(u,p)\right)\cdot  v
& = & h'(t)\cdot \int_{\p \mathcal{S}(t)} \sigma(u,p)n\d \Gamma +\omega(t) \cdot \int_{\p
\mathcal{S}(t)}
(x-h(t))\wedge \sigma(u,p)n \d \Gamma \\\
& & - 2\nu\int_{\mathcal{F}(t)} D(\overline{w}):D(v) -  
2\nu\int_{\mathcal{F}(t)} |D(v)|_{\R^9}^2.
\end{eqnarray*}
By using the equations \eqref{ch2_z1} and \eqref{ch2_z2} we deduce}\\
$\displaystyle \frac{1}{2}\frac{\d}{\d t} \int_{\mathcal{F}(t)}|v|_{\R^3}^2 \d x
+2\nu \int_{\mathcal{F}(t)}|D(v)|_{\R^9}^2 \d x + \frac{M}{2}\frac{\d}{\d
t}\left(|h'(t)|_{\R^3}^2\right)$ + 
\textcolor{black}{$(I\omega)'(t)\cdot \omega(t) $} \\
\begin{eqnarray*}
& = &
-2\nu \int_{\mathcal{F}(t)}D(\overline{w}):D(v) \d
x \\
&  & -\int_{\mathcal{F}(t)} \left( (v\cdot \nabla)\overline{w} \right)\cdot v \ \d x
 - \int_{\mathcal{F}(t)} \left( (\overline{w}\cdot \nabla)\overline{w} \right)\cdot v \ \d x -
\int_{\mathcal{F}(t)} \frac{\p \overline{w}}{\p t}\cdot v \ \d x.
\end{eqnarray*}
\textcolor{black}{The term involving the moment of inertia tensor\footnote{Remind that this tensor is an invertible symmetric matrix} can be transformed as follows
\begin{eqnarray*}
(I\omega)'\cdot \omega & = & \frac{1}{2}\frac{\d}{\d t}\left( (I\omega)\cdot \omega \right)
+ \frac{1}{2} I'\omega\cdot \omega,\\
I'\omega & = & \left(\mathbf{R}I^{\ast}\mathbf{R}^T \right)'\omega = 
\omega \wedge (I\omega) + \mathbf{R}{I^{\ast}}'\mathbf{R}^T\omega - I(\omega \wedge \omega),\\
I'\omega\cdot \omega & = & \mathbf{R}{I^{\ast}}'\mathbf{R}^T\omega \cdot \omega = 
{I^{\ast}}'\tilde{\omega}\cdot\tilde{\omega},
\end{eqnarray*}
where we use the notation $\tilde{\omega} = \mathbf{R}^T\omega$. So we have\\
$\displaystyle \frac{1}{2}\frac{\d}{\d t} \int_{\mathcal{F}(t)}|v|_{\R^3}^2 \d x
+2\nu \int_{\mathcal{F}(t)}|D(v)|_{\R^9}^2 \d x + \frac{M}{2}\frac{\d}{\d
t}\left(|h'(t)|_{\R^3}^2\right) + \frac{1}{2}\frac{\d}{\d t}\left( \left| \left(\sqrt{I}\omega\right)(t) \right|_{\R^3}^2 \right)$ \\
\begin{eqnarray*}
& = &
-\frac{1}{2}{I^{\ast}}'\tilde{\omega}\cdot\tilde{\omega}-2\nu
\int_{\mathcal{F}(t)}D(\overline{w}):D(v) \d
x \\
&  & -\int_{\mathcal{F}(t)} \left( (v\cdot \nabla)\overline{w} \right)\cdot v \ \d x
 - \int_{\mathcal{F}(t)} \left( (\overline{w}\cdot \nabla)\overline{w} \right)\cdot v \ \d x -
\int_{\mathcal{F}(t)} \frac{\p \overline{w}}{\p t}\cdot v \ \d x.
\end{eqnarray*}
}
It follows that there exists $C>0$ such that
\begin{eqnarray*}
&  & \frac{1}{2}\frac{\d}{\d t} \| v\|^2_{\mathbf{L}^2(\mathcal{F}(t))} +2\nu
\| D(v)\|^2_{\mathbf{L}^2(\mathcal{F}(t))} + \frac{M}{2}\frac{\d}{\d t}\left(|h'(t)|_{\R^3}^2\right)
+ \frac{1}{2}\frac{\d}{\d t}\left( \left| \left(\sqrt{I}\omega\right)(t)\right|_{\R^3}^2 \right) \\
& &  \leq C\left(\left\| \frac{\p \overline{w}}{\p t}\right\|^2_{\mathbf{L}^2(\mathcal{F}(t))}
+ \| (\overline{w}\cdot \nabla) \overline{w}\|^2_{\mathbf{L}^2(\mathcal{F}(t))} + \| D(\overline{w})\|^2_{\mathbf{L}^2(\mathcal{F}(t))}
 \right. \\
& & \quad \left.  + \| v\|^2_{\mathbf{L}^2(\mathcal{F}(t))}
\left(1 + \| \nabla \overline{w} \|^2_{\mathbf{L}^{\infty}(\mathcal{F}(t))}\right) + |\omega |^2(1+ |{I^{\ast}}'|_{\R^9}^2 ) \right).
\end{eqnarray*}
Using the estimates \eqref{estwwbar1}--\eqref{estwwbar4} combined with \eqref{estwwast1}-\eqref{estwwast2}, we get
\begin{eqnarray}
&  & \frac{1}{2}\frac{\d}{\d t} \| v\|^2_{\mathbf{L}^2(\mathcal{F}(t))} +2\nu
\| D(v)\|^2_{\mathbf{L}^2(\mathcal{F}(t))} + \frac{M}{2}\frac{\d}{\d t}\left(|h'(t)|_{\R^3}^2\right)
+ \frac{1}{2}\frac{\d}{\d t}\left( \left| \left(\sqrt{I}\omega\right)(t)\right|_{\R^3}^2 \right) \nonumber \\
& &  \leq \tilde{C} \left(\left\| \frac{\p w^{\ast}}{\p t}\right\|^2_{\mathbf{H}^1(\mathcal{S}^{\ast}(t))}
+ \| w^{\ast} \|^2_{\mathbf{H}^3(\mathcal{S}^{\ast}(t))} 
\left(1+\| w^{\ast} \|^2_{\mathbf{H}^1(\mathcal{S}^{\ast}(t))} \right)
+ \| w^{\ast} \|^2_{\mathbf{H}^1(\mathcal{S}^{\ast}(t))} \right. \nonumber \\
& & \quad \left.  +\left( \| v\|^2_{\mathbf{L}^2(\mathcal{F}(t))} + |\omega |_{\R^3}^2 + |h'|_{\R^3}^2 \right)
\left(1 + \| w^{\ast} \|^2_{\mathbf{H}^3(\mathcal{S}^{\ast}(t))} + |{I^{\ast}}'|_{\R^9}^2 \right) \right). \label{ch2_lastineq}
\end{eqnarray}
\textcolor{black}{Finally, let us control $\|  v \|_{\mathbf{H}^1(\mathcal{F}(t))}$ by $\| D(v) \|_{\mathbf{L}^2(\mathcal{F}(t))}$.} For that, we can extend the velocity field $v$ into $\mathcal{S}(t)$ by setting $v(x,t) = h'(t) + \omega(t) \wedge (x-h(t))$ for $x \in \mathcal{S}(t)$. Thus we have $v \in \mathbf{H}_0^1(\mathcal{O})$ with $\div \ v = 0$, and the following formula
\begin{eqnarray*}
\nabla v : \nabla v - 2D(v):D(v) & = & - \div \left((v\cdot\nabla) v -(\div \ v)v\right) - (\div \ v)^2
\end{eqnarray*}
combined with the Poincar\'e inequality enables us to write
\begin{eqnarray}
\|  v \|_{\mathbf{H}^1(\mathcal{F}(t))} & \leq & \|  v \|_{\mathbf{H}^1(\mathcal{O})} \nonumber \\
& \leq & C\| \nabla v \|_{\mathbf{L}^2(\mathcal{O})}
 =  2 C\| D(v) \|_{\mathbf{L}^2(\mathcal{O})} = 
 2 C\| D(v) \|_{\mathbf{L}^2(\mathcal{F}(t))}. \label{ineqDG}
\end{eqnarray}
Then we can conclude by using inequality \eqref{estmatrix} and the Gr\"{o}nwall's lemma on \eqref{ch2_lastineq}.
\end{proof}

Proposition \ref{ch2_zprop1} is the analogous adaptation of Proposition 1 and Lemma 4.1 of \cite{SMSTT} and \cite{Cumsille} respectively. The difference with \cite{SMSTT} is that in dimension 2 we do not have to assume smallness on the data, whereas in our case we need to quantify the regularity of the deformation for which we need to assume smallness in dimension 3.\\
\textcolor{black}{
\begin{corollary} \label{corodist}
Assume that the hypotheses in Theorem \ref{ch3_thlocex} hold true. Assume that the maximal time of existence $T_0$ of the strong solution provided by this theorem is finite. Then there exists $d \geq 0$ such that
\begin{eqnarray*}
\lim_{t \rightarrow T_0} \dist(\mathcal{S}(t), \p \mathcal{O}) & = & d.
\end{eqnarray*}
\end{corollary}
}
\textcolor{black}{
\begin{proof}
From Proposition \ref{ch2_zprop1} we know that the functions $t\mapsto |h'(t)|_{\R^3}$ and $t\mapsto | \omega(t)|_{\R^3}$ are bounded in $[0,T_0)$. The function
\begin{eqnarray*}
\left| \frac{\p \mathbf{R}}{\p t}\right|_{\R^9} & = & 
|\mathbb{S}(\omega) \mathbf{R} |_{\R^9} = | \mathbb{S}(\omega) |_{\R^9}
\end{eqnarray*}
is then also bounded in $[0,T_0)$. Hence it follows that the limits
\begin{eqnarray*}
\lim_{t \rightarrow T_0} h(t) & \text{ and } & \lim_{t \rightarrow T_0} \mathbf{R}(t)
\end{eqnarray*}
exist. Since $t \mapsto X^{\ast}(\cdot,t)$ is continuous on $[0,+\infty)$, the mapping $t  \mapsto X_{\mathcal{S}}(\cdot,t) = h(t) + \mathbf{R}(t)X^{\ast}(\cdot,t)$ admits a limit when $t$ goes to $T_0$.
\end{proof}
}

\textcolor{black}{From this corollary, it is sufficient to show that if the limit $d$ is not equal to $0$, then necessarily the maximal time of existence $T_0$ cannot be finite. So let us assume that $T_0$ is finite and that $d >0$. Since the function $t \mapsto \dist(\mathcal{S}(t), \p \mathcal{O})$ is continuous, there exists $\eta > 0$ such that for all $t \in [0, T_0]$ we have
\begin{eqnarray*}
\dist(\mathcal{S}(t), \p \mathcal{O}) & \geq & \eta.
\end{eqnarray*}
It remains us to show that under these conditions the norm $\| u \|_{\mathbf{H}^1(\mathcal{F}(t))}$ is bounded in $[0,T_0]$. Then we will be able to extend the strong solution to a time larger than $T_0$, which is a contradiction. 
}

\textcolor{black}{
\begin{remark}
Note that the argument of extension of the solution to a time larger than $T_0$ is submitted to assumptions on the displacement $X^{\ast}-\Id_{\mathcal{S}}$; In particular, in the statement of Theorem \ref{ch3_thlocex}, we assume the initial conditions $X^{\ast}(\cdot,0) = \Id_{\mathcal{S}}$ and $\frac{\p X^{\ast}}{\p t}(\cdot,0) = 0$. So for extending these solution after $T_0$ we would need the {\it a priori} the same conditions at time $T_0$. However, we do not need to assume that. Firstly, the estimates we need for proving this theorem deal only with the time derivative of $X^{\ast}$, so the first condition is not important and can be relaxed. Secondly, the homogeneous condition on the time derivative can be replaced by a non-null velocity; In that case, this initial condition would behave just like the other initial conditions on the state of the system, $\|u_0\|_{\mathbf{H}^1(\mathcal{F})}$, $| h_1|_{\R^3}$ and $| \omega_0 |_{\R^3}$.
\end{remark}
}

\subsection{Rest of the proof}
Likewise for the proof of Proposition \ref{ch2_zprop1}, let us consider the system \eqref{ch2_z0}--\eqref{ch2_z4} satisfied by $v = u - \overline{w}$. By taking the inner product of \eqref{ch2_z0} with $\displaystyle \frac{\p v}{\p t}$ and by integrating on
$\mathcal{F}(t)$ we get\\
$\displaystyle \left\| \frac{\p v}{\p t} \right\|_{\mathbf{L}^2(\mathcal{F}(t))}^2 -
\int_{\mathcal{F}(t)} \div \ \sigma(u,p)\cdot \frac{\p v}{\p t} \ \d x
\quad =$  
\begin{eqnarray*}
& & -\int_{\mathcal{F}(t)} [(u\cdot\nabla)v]\cdot\frac{\p v}{\p t} \ \d x -
\int_{\mathcal{F}(t)} [(v\cdot\nabla)\overline{w}]\cdot\frac{\p v}{\p t} \ \d x \\
& &  -\int_{\mathcal{F}(t)} [(\overline{w}\cdot\nabla)\overline{w}]\cdot\frac{\p v}{\p t} \ \d x -
\int_{\mathcal{F}(t)} \frac{\p \overline{w}}{\p t}\cdot\frac{\p v}{\p t} \ \d x ,
\end{eqnarray*}
for almost all $t\in (0,T_0)$, and furthermore by replacing $u$ by $v+\overline{w}$\\
$\displaystyle \left\| \frac{\p v}{\p t} \right\|_{\mathbf{L}^2(\mathcal{F}(t))}^2 -
\int_{\mathcal{F}(t)} \div \ \sigma(u,p)\cdot \frac{\p v}{\p t} \ \d x \quad = $
\begin{eqnarray}
& & -\int_{\mathcal{F}(t)} [(v\cdot\nabla)v]\cdot\frac{\p v}{\p t} \ \d x
-\int_{\mathcal{F}(t)} [(\overline{w}\cdot\nabla)v]\cdot\frac{\p v}{\p t} \ \d x
 -\int_{\mathcal{F}(t)} [(v\cdot\nabla)\overline{w}]\cdot\frac{\p v}{\p t} \ \d x \nonumber \\
& & -\int_{\mathcal{F}(t)} [(\overline{w}\cdot\nabla)\overline{w}]\cdot\frac{\p v}{\p t} \ \d x  - \int_{\mathcal{F}(t)} \frac{\p \overline{w}}{\p t}\cdot\frac{\p v}{\p t} \ \d x  +\nu \int_{\mathcal{F}(t)} \Delta \overline{w} \cdot\frac{\p v}{\p t} \ \d x
 \label{ch2_zz1}
\end{eqnarray}
almost everywhere in $(0,T_0)$. \\
Up to a density argument - as the one which is detailed in \cite{Cumsille} - and by using the relations $v=h'(t) + \omega (t)\wedge(x-h(t))$ on $\p \mathcal{S}(t)$ and $v=0$ on $\p \mathcal{O}$, we have\\
$- \displaystyle \int_{\mathcal{F}(t)}  \div \left(\sigma(u,p)\right)\cdot
\frac{\p v}{\p t} \quad = $
\begin{eqnarray}
&  & \nu \frac{\d}{\d t}\int_{\mathcal{F}(t)} |D(v)|_{\R^9}^2 + M|h''|_{\R^3}^2 +
\left|\sqrt{I} \omega'\right|_{\R^3}^2 - Mh''\cdot \left( h'\wedge \omega\right)
+\left(\omega \wedge \left(I\omega\right) \right)\cdot \omega' +
{I^{\ast}}'\tilde{\omega}\cdot \tilde{\omega}' \nonumber \\
& & +2\nu \int_{\p \mathcal{S}(t)} D(\overline{w})n\cdot\left(h''+h' \wedge \omega + \omega'
\wedge(x-h) \right) \d \Gamma  . \label{ch2_zz2}
\end{eqnarray}
By combining the equalities \eqref{ch2_zz1} and \eqref{ch2_zz2} we obtain almost
everywhere in $(0,T_0)$ \\
$\displaystyle \left\| \frac{\p v}{\p t}\right\|_{\mathbf{L}^2(\mathcal{F}(t))}^2
+\nu \frac{\d}{\d t} \int_{\mathcal{F}(t)} |D(v)|_{\R^9}^2 \ \d x
+M|h''(t)|_{\R^3}^2 + \left|\left(\sqrt{I} \omega'\right)(t) \right|_{\R^3}^2 $\\
\begin{eqnarray*}
= \int_{\mathcal{F}(t)} \mathbb{F}_1\cdot \frac{\p v}{\p t} \ \d x + \mathbb{F}_2
\cdot h''(t) +\mathbb{F}_3\cdot \omega'(t) + \mathbb{F}_4,
\end{eqnarray*}
where
\begin{eqnarray*}
\mathbb{F}_1 & = & -[(v\cdot \nabla)v]+[(\overline{w}\cdot \nabla)v]+[(v\cdot
\nabla)\overline{w}]+[(\overline{w}\cdot \nabla)\overline{w}]-\frac{\p \overline{w}}{\p t}+\nu\Delta \overline{w}, \\
\mathbb{F}_2 & = & M h'(t) \wedge \omega(t)-2\nu \int_{\p \mathcal{S}(t)} D(\overline{w})n\d \Gamma
, \\
\mathbb{F}_3 & = & \left(I\omega\right)\wedge \omega -
\mathbf{R}{I^{\ast}}'\tilde{\omega} -2\nu \int_{\p \mathcal{S}(t)}
(x-h)\wedge D(\overline{w})n\d \Gamma , \\
\mathbb{F}_4 & = & -2\nu \int_{\p \mathcal{S}(t)} D(\overline{w})n \cdot \left(h' \wedge \omega
\right)\d \Gamma.
\end{eqnarray*}
By using the Cauchy-Schwartz inequality combined with the Young inequality and the
fact that $\|u\|_{\mathbf{L}^2(\mathcal{F}(t))}$, $h'$ and
$\omega$ are bounded in $[0,T_0)$ (by Proposition \ref{ch2_zprop1}), we deduce that
there exists a positive constant $C_3$ such that for almost all $t\in (0,T_0)$
\begin{eqnarray}
\left\| \frac{\p v}{\p t}\right\|_{\mathbf{L}^2(\mathcal{F}(t))}^2 + 2\nu
\frac{\d}{\d t} \int_{\mathcal{F}(t)} |D(v)|_{\R^9}^2 \ \d x
+M|h''(t)|_{\R^3}^2 + \left|\left(\sqrt{I} \omega'\right)(t) \right|_{\R^3}^2 \nonumber \\
\leq C_3 \left( C_0^2 + C_0^2\|\nabla
v\|_{\mathbf{L}^2(\mathcal{F}(t))}^2+\| (v\cdot \nabla)v
\|_{\mathbf{L}^2(\mathcal{F}(t))}^2 \right),
\label{ch2_zzz3}
\end{eqnarray}
by reminding that the constant $C_0$ is the one which appears in Proposition
\ref{ch2_zprop1}.\\
Then, in order to handle the nonlinear term, we first use the H\"{o}lder
inequality
\begin{eqnarray*}
\| (v\cdot \nabla)v \|_{\mathbf{L}^2(\mathcal{F}(t))}^2 \leq \| v
\|_{\mathbf{L}^4(\mathcal{F}(t))}^2 \| \nabla v \|_{\mathbf{L}^4(\mathcal{F}(t))}^2
\end{eqnarray*}
and we remind the continuous embedding of $\mathbf{H}^{3/4}(\mathcal{F}(t))$ in
$\mathbf{L}^4(\mathcal{F}(t))$, so
\begin{eqnarray*}
\| (v\cdot \nabla)v \|_{\mathbf{L}^2(\mathcal{F}(t))}^2 \leq C \| v
\|_{\mathbf{H}^{3/4}(\mathcal{F}(t))}^2 \| \nabla v \|_{\mathbf{H}^{3/4}(\mathcal{F}(t))}^2.
\end{eqnarray*}
Thus, by interpolation we obtain
\begin{eqnarray}
\| (v\cdot \nabla)v \|_{\mathbf{L}^2(\mathcal{F}(t))}^2
& \leq & \tilde{C} \| v \|_{\mathbf{H}^{1}(\mathcal{F}(t))}^{3/2} \| v
\|_{\mathbf{L}^{2}(\mathcal{F}(t))}^{1/2}
 \| \nabla v \|_{\mathbf{H}^{1}(\mathcal{F}(t))}^{3/2} \| \nabla v
\|_{\mathbf{L}^{2}(\mathcal{F}(t))}^{1/2} \nonumber \\
& \leq & C_1 \| v \|_{\mathbf{L}^{2}(\mathcal{F}(t))}^{1/2} \| \nabla v
\|_{\mathbf{L}^{2}(\mathcal{F}(t))}^{1/2}
\left( \| v \|_{\mathbf{L}^{2}(\mathcal{F}(t))} + \| \nabla v \|_{\mathbf{L}^{2}(\mathcal{F}(t))}
\right)^{3/2} \times \nonumber \\
& & \left( \| \nabla v \|_{\mathbf{L}^{2}(\mathcal{F}(t))} + \sum_{i=1}^3 \| \nabla^2 v_i
\|_{\mathbf{L}^{2}(\mathcal{F}(t))} \right)^{3/2}. \label{ch2_zzz4}
\end{eqnarray}
From Proposition \ref{ch2_zprop1}, $\|v \|_{\mathbf{L}^2(\mathcal{F}(t))}$ is bounded, so we have
\begin{eqnarray*}
\| (v\cdot \nabla)v \|_{\mathbf{L}^2(\mathcal{F}(t))}^2 & \leq &
\tilde{C}_1 \| \nabla v \|_{\mathbf{L}^{2}(\mathcal{F}(t))}^{1/2}
\left( 1 + \| \nabla v \|_{\mathbf{L}^{2}(\mathcal{F}(t))} \right)^{3/2} \times \\
& & \left( \| \nabla v \|_{\mathbf{L}^{2}(\mathcal{F}(t))} + \sum_{i=1}^3 \| \nabla^2 v_i
\|_{\mathbf{L}^{2}(\mathcal{F}(t))} \right)^{3/2}.
\end{eqnarray*}
On the other hand, if we consider the following Stokes system at some fixed
time $t>0$
\begin{eqnarray*}
-\nu \Delta v + \nabla p  =  \overline{f} \quad & \textrm{ in } & \mathcal{F}(t), \\
\div \ v  =  0 \quad & \textrm{ in } & \mathcal{F}(t), \\
v  =  h' + \omega\wedge(x-h) \quad & \textrm{ on } & \p \mathcal{S}(t), \\
v  =  0 \quad & \textrm{ on } & \p \mathcal{O},
\end{eqnarray*}
with $\displaystyle \overline{f} : =  -\frac{\p v}{\p t} - (v\cdot \nabla)v - (v\cdot \nabla)\overline{w} - (\overline{w}\cdot \nabla)v$, we have the following classical estimate
\begin{eqnarray*}
\| \nabla v \|_{\mathbf{L}^{2}(\mathcal{F}(t))} +\sum_{i=1}^3 \| \nabla^2 v_i \|_{\mathbf{L}^{2}(\mathcal{F}(t))} & \leq & C_2 \left( \|
\overline{f} \|_{\mathbf{L}^{2}(\mathcal{F}(t))}
 +\| h' \|_{\R^3} + \|\omega \|_{\R^3} \right).
\end{eqnarray*}
Since from Proposition \ref{ch2_zprop1} the quantities $h'$ and $\omega$ are bounded by some constant $K_1>0$, we have
\begin{eqnarray*}
\| \nabla v \|_{\mathbf{L}^{2}(\mathcal{F}(t))} +\sum_{i=1}^3 \| \nabla^2 v_i \|_{\mathbf{L}^{2}(\mathcal{F}(t))} & \leq & C_2 \left( \|
\overline{f} \|_{\mathbf{L}^{2}(\mathcal{F}(t))}
 + K_1 \right).
\end{eqnarray*}
\begin{remark}
Notice that the constant $C_2$ does not depend on time, since we have $\dist(\mathcal{S}(t), \p \mathcal{O}) >0$ for all $t\in[]0,T_0)$.
\end{remark}

Consequently, by considering that
\begin{eqnarray*}
\| \overline{f} \|_{\mathbf{L}^{2}(\mathcal{F}(t))} & \leq &  \left\| \frac{\p v}{\p t} \right\|_{\mathbf{L}^{2}(\mathcal{F}(t))} + \|(v\cdot \nabla)v
\|_{\mathbf{L}^{2}(\mathcal{F}(t))} 
 + K_2 K_1 + K_3\|\nabla v\|_{\mathbf{L}^{2}(\mathcal{F}(t))},
\end{eqnarray*}
we get the following inequality
\begin{eqnarray*}
& & \displaystyle \| (v\cdot \nabla)v \|_{\mathbf{L}^2(\mathcal{F}(t))}^2
 \leq  \hat{C} \| \nabla v \|_{\mathbf{L}^{2}(\mathcal{F}(t))}^{1/2}
\left( 1 + \| \nabla v \|_{\mathbf{L}^{2}(\mathcal{F}(t))} \right)^{3/2} \times \\
& &  \left( K_2 K_1 +K_3 \| \nabla v \|_{\mathbf{L}^2(\mathcal{F}(t))} +
 \left\| \frac{\p v}{\p t} \right\|_{\mathbf{L}^2(\mathcal{F}(t))} +
\|(v\cdot \nabla)v \|_{\mathbf{L}^2(\mathcal{F}(t))} \right)^{3/2}.
\end{eqnarray*}
By a classical convexity inequality we can develop
\begin{eqnarray*}
\left( 1 + \| \nabla v \|_{[\L^{2}(\mathcal{F}(t))]^9} \right)^{3/2}
& \leq & \tilde{C}_1 \left( 1 + \| \nabla v \|_{\mathbf{L}^{2}(\mathcal{F}(t))}^{3/2}
\right)
\end{eqnarray*}
and also
\begin{eqnarray*}
\left( K_2 K_1 + K_3 \| \nabla v \|_{\mathbf{L}^2(\mathcal{F}(t))}  + \left\| \frac{\p v}{\p t}
\right\|_{\mathbf{L}^2(\mathcal{F}(t))} +
\|(v\cdot \nabla)v \|_{\mathbf{L}^2(\mathcal{F}(t))} \right)^{3/2} \leq  \\
 \tilde{C}_2 \left( 1 + \| \nabla v \|_{\mathbf{L}^2(\mathcal{F}(t))}^{3/2} +
 \left\| \frac{\p v}{\p t}
\right\|_{\mathbf{L}^2(\mathcal{F}(t))}^{3/2} +
\|(v\cdot \nabla)v \|_{\mathbf{L}^2(\mathcal{F}(t))}^{3/2} \right). 
\end{eqnarray*}
Thus
\begin{eqnarray*}
& & \displaystyle \| (v\cdot \nabla)v \|_{\mathbf{L}^2(\mathcal{F}(t))}^2
 \leq  \hat{C}  \| \nabla v \|_{\mathbf{L}^{2}(\mathcal{F}(t))}^{1/2}
\left( 1 + \| \nabla v \|_{\mathbf{L}^{2}(\mathcal{F}(t))}^{3/2}\right) \times \\
& & \tilde{C}_2 \left( 1 + \| \nabla v \|_{\mathbf{L}^2(\mathcal{F}(t))}^{3/2} + \left\| \frac{\p v}{\p t}
\right\|_{\mathbf{L}^2(\mathcal{F}(t))}^{3/2} + \|(v\cdot \nabla)v
\|_{\mathbf{L}^2(\mathcal{F}(t))}^{3/2} \right).
\end{eqnarray*}
Then, if we set $\hat{\hat{C}} = \tilde{C}_1 \tilde{C}_2 \hat{C}$ and
\begin{eqnarray*}
\mathbb{A} & = & \hat{\hat{C}} \| \nabla v \|_{\mathbf{L}^{2}(\mathcal{F}(t))}^{1/2}
\left( 1 + \| \nabla v \|_{\mathbf{L}^{2}(\mathcal{F}(t))}^{3/2} \right)  \left( 1 + \| \nabla v \|_{\mathbf{L}^2(\mathcal{F}(t))}^{3/2} +
 \left\| \frac{\p v}{\p t}
\right\|_{\mathbf{L}^2(\mathcal{F}(t))}^{3/2}\right) \\
\mathbb{B} & = & \hat{\hat{C}} \| \nabla v \|_{\mathbf{L}^{2}(\mathcal{F}(t))}^{1/2}
\left( 1 + \| \nabla v \|_{\mathbf{L}^{2}(\mathcal{F}(t))}^{3/2} \right),
\end{eqnarray*}
we have
\begin{eqnarray*}
\|(v\cdot \nabla)v \|_{\mathbf{L}^2(\mathcal{F}(t))}^{2}
& \leq & \mathbb{A} + \mathbb{B}\|(v\cdot \nabla)v
\|_{\mathbf{L}^2(\mathcal{F}(t))}^{3/2},
\end{eqnarray*}
and thus by the Young inequality we get
\begin{eqnarray*}
\|(v\cdot \nabla)v \|_{\mathbf{L}^2(\mathcal{F}(t))}^{2} & \leq & 4\mathbb{A} +
\mathbb{B}^4 
\end{eqnarray*}
with
\begin{eqnarray*}
\mathbb{B}^4 & \leq & \bar{C}  \| \nabla v \|_{\mathbf{L}^{2}(\mathcal{F}(t))}^{2} \left(
1 + \| \nabla v \|_{\mathbf{L}^{2}(\mathcal{F}(t))}^{6} \right).
\end{eqnarray*}
Therefore by returning to \eqref{ch2_zzz3} we obtain for $\tilde{C}_3$ large enough
\begin{eqnarray*}
\left\| \frac{\p v}{\p t}\right\|_{\mathbf{L}^2(\mathcal{F}(t))}^2 + 2\nu
\frac{\d}{\d t} \int_{\mathcal{F}(t)} |D(v)|_{\R^9}^2 \ \d x
+M|h''(t)|_{\R^3}^2 + \left|\left(\sqrt{I} \omega'\right)(t) \right|_{\R^3}^2 \\
\leq \tilde{C}_3 \left( 1+ 4\mathbb{A} + \mathbb{B}^4 \right).
\end{eqnarray*}
This leads us to \\

\noindent $\displaystyle \left\| \frac{\p v}{\p
t}\right\|_{\mathbf{L}^2(\mathcal{F}(t))}^2 + 2\nu \frac{\d}{\d t} \int_{\mathcal{F}(t)}
|D(v)|_{\R^9}^2 \ \d x
+M|h''(t)|_{\R^3}^2 + \left|\left(\sqrt{I} \omega'\right)(t) \right|_{\R^3}^2 $ \\
\begin{eqnarray*}
& \leq & \overline{C}_3\left( C_0^2 + C_0^2\|\nabla
v\|_{\mathbf{L}^2(\mathcal{F}(t))}^2 \right.\\
& & \left.+ \| \nabla v
\|_{\mathbf{L}^{2}(\mathcal{F}(t))}^{1/2}\left( 1+\| \nabla v
\|_{\mathbf{L}^{2}(\mathcal{F}(t))}^{3/2}\right)
\left(1+\| \nabla v \|_{\mathbf{L}^{2}(\mathcal{F}(t))}^{3/2} +
\left\| \frac{\p v}{\p t}\right\|_{\mathbf{L}^2(\mathcal{F}(t))}^{3/2} \right)
\right.\\
& & \left. +\| \nabla v \|_{\mathbf{L}^{2}(\mathcal{F}(t))}^{2}\left(1+ \| \nabla v
\|_{\mathbf{L}^{2}(\mathcal{F}(t))}^{6} \right)\right) \\
& \leq &  \overline{C}_3\left(C_0^2+\left(C_0^2 + 3\right)\| \nabla v
\|_{\mathbf{L}^{2}(\mathcal{F}(t))}^{2}
+ \| \nabla v \|_{\mathbf{L}^{2}(\mathcal{F}(t))}^{1/2} + \| \nabla v \|_{\mathbf{L}^{2}(\mathcal{F}(t))}^{7/2} \right.\\
& & \left.+\| \nabla v \|_{\mathbf{L}^{2}(\mathcal{F}(t))}^{8} +
\left( \| \nabla v \|_{\mathbf{L}^{2}(\mathcal{F}(t))}^{1/2}+\| \nabla v
\|_{\mathbf{L}^{2}(\mathcal{F}(t))}^{2}\right)
\left\| \frac{\p v}{\p t}\right\|_{\mathbf{L}^2(\mathcal{F}(t))}^{3/2}\right).
\end{eqnarray*}
Then, still by the Young inequality, we get \\

\noindent $\displaystyle \left\| \frac{\p v}{\p
t}\right\|_{\mathbf{L}^2(\mathcal{F}(t))}^2 + 2\nu \frac{\d}{\d t} \int_{\mathcal{F}(t)}
|D(v)|_{\R^9}^2 \ \d x
+M|h''(t)|_{\R^3}^2 + \left|\left(\sqrt{I} \omega'\right)(t) \right|_{\R^3}^2 $ \\
\begin{eqnarray*}
& \leq & \overline{C}_4C_0^2+ \overline{C}_5 \left( \| \nabla v
\|_{\mathbf{L}^{2}(\mathcal{F}(t))}^{1/2} +
\| \nabla v\|_{\mathbf{L}^{2}(\mathcal{F}(t))}^{2}
+ \| \nabla v
\|_{\mathbf{L}^{2}(\mathcal{F}(t))}^{7/2} +
\| \nabla v \|_{\mathbf{L}^{2}(\mathcal{F}(t))}^{8} \right). 
\end{eqnarray*}
By integrating this inequality on $[0,t]$, by using \eqref{ineqDG}, and with the Cauchy-Schwartz inequality we get \\

\noindent $\displaystyle \int_0^t\left\| \frac{\p v}{\p
t}(s)\right\|_{\mathbf{L}^2(\mathcal{F}(s))}^2 \d s
+ \nu \| \nabla v(t)
\|_{\mathbf{L}^{2}(\mathcal{F}(t))}^{2} +M\int_0^t|h''(s)|_{\R^3}^2\d s +
\int_0^t\left|\left(\sqrt{I} \omega'\right)(s) \right|_{\R^3}^2 \d
s$ \\
\begin{eqnarray*}
 & \leq & \nu \| \nabla v_0 \|_{\mathbf{L}^{2}(\mathcal{F})}^{2} + T_0\overline{C}_4C_0^2
+ \overline{C}_5T_0^{3/4} 
\| \nabla v \|_{\L^2(0,T_0;\mathbf{L}^{2}(\mathcal{F}(t)))}^{1/2} 
+ \overline{C}_5\| \nabla v \|_{\L^2(0,T_0;\mathbf{L}^{2}(\mathcal{F}(t)))}^{2} \\
& & + \overline{C}_5\int_0^t \| \nabla v (s)\|_{\mathbf{L}^{2}(\mathcal{F}(s))}^{7/2} \d
s  + \overline{C}_5\int_0^t \| \nabla v (s)\|_{\mathbf{L}^{2}(\mathcal{F}(s))}^{8}
\d s .
\end{eqnarray*}
Since $\nabla v = \nabla u - \nabla \overline{w}$, from \eqref{estwwbar1}, \eqref{estwwast1} we have
\begin{eqnarray*}
\| \nabla v \|_{\L^2(0,T_0;\mathbf{L}^{2}(\mathcal{F}(t)))} & \leq & 
\| \nabla u \|_{\L^2(0,T_0;\mathbf{L}^{2}(\mathcal{F}(t)))} + C_0^2
\end{eqnarray*}
by replacing $C_0$ on $\max\left(C_0, \| w^{\ast} \|^{1/2}_{\L^2(0,T_0;\mathbf{H}^{3}(\mathcal{S}^{\ast}(t)))} \right)$, and from Proposition \ref{ch2_zprop1} we deduce 
\begin{eqnarray}
\| \nabla v \|_{\L^2(0,T_0;\mathbf{L}^{2}(\mathcal{F}(t)))} & \leq & 
(K+1)C_0^2 \label{lastcont}
\end{eqnarray}
and so\\

\noindent $\displaystyle \int_0^t\left\| \frac{\p v}{\p
t}(s)\right\|_{\mathbf{L}^2(\mathcal{F}(s))}^2 \d s
+ \nu \| \nabla v(t)
\|_{\mathbf{L}^{2}(\mathcal{F}(t))}^{2} +M\int_0^t|h''(s)|_{\R^3}^2\d s +
\int_0^t\left|\left(\sqrt{I} \omega'\right)(s) \right|_{\R^3}^2 \d
s$ \\
\begin{eqnarray}
 & \leq & \nu \| \nabla v_0 \|_{\mathbf{L}^{2}(\mathcal{F})}^{2} + T_0\overline{C}_4C_0^2
+ \overline{C}_5T_0^{3/4} (K+1)^{1/2}C_0  + \overline{C}_5(K+1)^2C_0^4 \nonumber \\
& & + \overline{C}_5\int_0^t \| \nabla v
(s)\|_{\mathbf{L}^{2}(\mathcal{F}(s))}^{7/2} \d s  +
\overline{C}_5\int_0^t \| \nabla v
(s)\|_{\mathbf{L}^{2}(\mathcal{F}(s))}^{8} \d s .
\label{ch2_ineqsup}
\end{eqnarray}
For small data we can have in particular $C_0$ and $\|\nabla v_0\|_{\mathbf{L}^{2}(\mathcal{F})}^{2}$ small enough to satisfy
\begin{eqnarray}
\nu \| \nabla v_0 \|_{\mathbf{L}^{2}(\mathcal{F})}^{2} + T_0\overline{C}_4C_0^2
+ \overline{C}_5T_0^{3/4} (K+1)^{1/2}C_0  + 3\overline{C}_5(K+1)^2C_0^4 & < & \nu.  \label{ch2_hypglob}
\end{eqnarray}
For such initial data we notice in particular that $\|\nabla
v_0\|_{\mathbf{L}^{2}(\mathcal{F})}^{2} < 1$, and then by continuity there exists a
maximal time $\tilde{T}_0$ such that for all $t\in[0,\tilde{T}_0]$ we have
$\|\nabla
v(t)\|_{[\L^{2}(\mathcal{F}(t))]^9}^{2} \leq 1$. Let us show that $\tilde{T}_0 = T_0$;
By contradiction, let us assume that $\tilde{T}_0 < T_0$. For all
$s\in[0,\tilde{T}_0]$
we have for $\displaystyle r\in \left\{\frac{7}{2},8\right\}$
\begin{eqnarray*}
 \| \nabla v(s)\|_{\mathbf{L}^{2}(\mathcal{F}(s))}^{r} & \leq & \|
\nabla
v(s)\|_{\mathbf{L}^{2}(\mathcal{F}(s))}^{2},
\end{eqnarray*}
and by returning to \eqref{ch2_ineqsup} with these inequalities and \eqref{lastcont} we deduce
\begin{eqnarray*}
\nu \| \nabla v(\tilde{T}_0) \|_{\mathbf{L}^{2}(\mathcal{F}(\tilde{T}_0))}^{2} & \leq
& \nu \| \nabla v_0 \|_{\mathbf{L}^{2}(\mathcal{F})}^{2} + T_0\overline{C}_4C_0^2
+ \overline{C}_5T_0^{3/4} (K+1)^{1/2}C_0  \\
& & + \overline{C}_5(K+1)^2C_0^4 + 2\overline{C}_5(K+1)^2C_0^4 .
\end{eqnarray*}
Then under the hypothesis \eqref{ch2_hypglob} we deduce that $\| \nabla
v(\tilde{T}_0) \|_{\mathbf{L}^{2}(\mathcal{F}(\tilde{T}_0))}^{2} < 1$, and by continuity
we can find $\epsilon > 0$ such that
$\| \nabla v(t) \|_{\mathbf{L}^{2}(\mathcal{F}(t))}^{2} \leq 1$ for $t\in [\tilde{T}_0,
\tilde{T}_0 + \epsilon]$. This belies the definition of $\tilde{T}_0$ as an upper bound, and thus
$\tilde{T}_0 = T_0$.\\
This shows that by assuming the initial data small enough, $\| \nabla v(t)\|_{\mathbf{L}^{2}(\mathcal{F}(t))}$ is bounded (in $\L^{\infty}(0,T_0)$). Since $u = v + \overline{w}$, it follows from the estimate \eqref{estwwbar1} that $\| \nabla u(t)\|_{\mathbf{L}^{2}(\mathcal{F}(t))}$ is also bounded. By adding Proposition \ref{ch2_zprop1}, the function $\displaystyle t \mapsto \|  u \|_{\mathbf{H}^1(\mathcal{F}(t))}$ is bounded on $[0,T_0]$. On the other hand, from Corollary \ref{corodist}, if we assume that
\begin{eqnarray*}
\lim_{t\rightarrow T_0} \dist(\mathcal{S}(t), \p \mathcal{O}) & > & 0,
\end{eqnarray*}
it follows from the continuity of the function $t \mapsto \dist(\mathcal{S}(t), \p \mathcal{O})$ that
\begin{eqnarray*}
\dist(\mathcal{S}(t), \p \mathcal{O}) >  0, & \quad & t \in [0,T_0].
\end{eqnarray*}
Thus the strong solution of Theorem \ref{ch3_thlocex} can be extended to $[0,T_0']$, with $T_0' > T_0$. This belies the fact that $T_0$ could be the maximal time of existence of this solution. Finally we can conclude that if $T_0$ is finite, then necessarily
\begin{eqnarray*}
\lim_{t\rightarrow T_0} \dist(\mathcal{S}(t), \p \mathcal{O}) & = & 0,
\end{eqnarray*}
and the proof is complete.

\newpage
\section{Appendix A: The changes of variables}

\subsection{Preliminary results}
Let us remind a result stated in the Appendix B of \cite{Grubb} (Proposition B.1), which treats of Sobolev regularities for products of functions, and that we state in the particular case of dimension 3 as follows:

\begin{lemma} \label{ch3_lemmaGrubb}
Let $s$, $\mu$ and $\kappa$ in $\R$. If $f\in \H^{s+\mu}(\mathcal{F})$ and $g\in \H^{s+\kappa}(\mathcal{F})$, then there exists a positive constant $C$ such that
\begin{eqnarray*}
\| f g \|_{\H^s(\mathcal{F})} & \leq & C \| f  \|_{\H^{s+\mu}(\mathcal{F})} \| g \|_{\H^{s+\kappa}(\mathcal{F})},
\end{eqnarray*}
(i) when $s + \mu + \kappa \geq 3/2$, \\
(ii) with $\mu \geq 0$, $\kappa \geq 0$, $2s + \mu + \kappa \geq 0$, \\
(iii) except that $s + \mu + \kappa > 3/2$ if equality holds somewhere in (ii).

\end{lemma}

A consequence of this Lemma is the following result.
\textcolor{black}{
\begin{lemma} \label{ch3_lemmecomatrice}
Let be $T>0$ and $\tilde{X}$ in $\mathcal{W}(Q_T^0)$ satisfying $\tilde{X}(\cdot,0) = \Id_{\mathcal{F}}$. Then
\begin{eqnarray}
\com \nabla \tilde{X} & \in & \H^{1}(0,T;\mathbf{H}^{2}(\mathcal{F})) \cap \H^{2}(0,T;\mathbf{L}^2(\mathcal{F})), \label{ch3_resfond1}
\end{eqnarray}
and there exists a positive constant $C$ such that
\begin{eqnarray}
& & \| \com \nabla \tilde{X} - \I_{\R^3} \|_{\H^{1}(\mathbf{H}^{2}) \cap \H^{2}(\mathbf{L}^2)}
 \nonumber \\
& & \leq C \| \nabla \tilde{X} - \I_{\R^3} \|_{\H^{1}(\mathbf{H}^{2}) \cap \H^{2}(\mathbf{L}^2)}
\left(1+ \| \nabla \tilde{X} - \I_{\R^3} \|_{\H^{1}(\mathbf{H}^{2}) \cap \H^{2}(\mathbf{L}^2)}  \right) \nonumber \\
& & \leq C \sqrt{1+T^2} \left\| \frac{\p \tilde{X}}{\p t}\right\|_{\mathcal{H}(Q_T^0)} \left(1+ \sqrt{1+T^2}
\left\| \frac{\p \tilde{X}}{\p t}\right\|_{\mathcal{H}(Q_T^0)} \right).
 \label{ch3_estcof}
\end{eqnarray}
Moreover, if $\tilde{X}_1, \ \tilde{X}_2 \in \mathcal{W}(Q_T^0)$, then
\begin{eqnarray}
& & \| \com \nabla \tilde{X}_2 -\com \nabla \tilde{X}_1 \|_{\H^{1}(\mathbf{H}^{2}) \cap \H^{2}(\mathbf{L}^2)}
  \nonumber \\
& & \leq C \| \nabla \tilde{X}_2 - \nabla \tilde{X}_1 \|_{\H^{1}(\mathbf{H}^{2}) \cap \H^{2}(\mathbf{L}^2)}
\left(1+ \| \nabla \tilde{X}_1\| + \|\nabla \tilde{X}_2 \|_{\H^{1}(\mathbf{H}^{2}) \cap \H^{2}(\mathbf{L}^2)}  \right) \nonumber \\
& & \leq C \sqrt{1+T^2}\left\| \frac{\p (\tilde{X}_2-\tilde{X}_1)}{\p t}\right\|_{\mathcal{H}(Q_T^0)}\left(1+\sqrt{1+T^2} \left(
\left\| \frac{\p \tilde{X}_1}{\p t}\right\|_{\mathcal{H}(Q_T^0)}
+ \left\| \frac{\p \tilde{X}_1}{\p t}\right\|_{\mathcal{H}(Q_T^0)} \right) \right).
\nonumber \\ \label{ch3_estcof12}
\end{eqnarray}
\end{lemma}
}

\begin{proof}
For proving \eqref{ch3_resfond1}, it is sufficient to show that the space $\H^{1}(0,T;\H^{2}(\mathcal{F})) \cap \H^{2}(0,T;\L^2(\mathcal{F}))$ is stable by product. For that, let us consider two functions $f$ and $g$ which lie in this space. We write
\begin{eqnarray*}
\frac{\p (fg)}{\p t} & = & \frac{\p f}{\p t}g + f\frac{\p g}{\p t}. \label{ch3_eqtriv}
\end{eqnarray*}
Applying Lemma \ref{ch3_lemmaGrubb} with $s = 2$ and $\mu = \kappa = 0$, we get
\begin{eqnarray*}
\left\|\frac{\p (fg)}{\p t} \right\|_{\L^2(0,T;\H^{2}(\mathcal{F}))} & \leq & C\left(
\left\| \frac{\p f}{\p t} \right\|_{\L^2(0,T;\H^{2}(\mathcal{F}))} \left\| g \right\|_{\L^{\infty}(0,T;\H^{2}(\mathcal{F}))} \right. \\
 & & + \left. \left\| \frac{\p g}{\p t} \right\|_{\L^2(0,T;\H^{2}(\mathcal{F}))} \left\| f \right\|_{\L^{\infty}(0,T;\H^{2}(\mathcal{F}))} \right)
\end{eqnarray*}
and thus $fg \in \H^1(0,T;\H^{2}(\mathcal{F}))$. For the regularity of $fg$ in $\H^{2}(0,T;\L^2(\mathcal{F}))$, we consider
\begin{eqnarray*}
\frac{\p^2 (fg)}{\p t^2} & = & \frac{\p^2 f}{\p t^2}g + f\frac{\p^2 g}{\p t^2} + 2\frac{\p f}{\p t} \frac{\p g}{\p t}
\end{eqnarray*}
with
\begin{eqnarray*}
\begin{array} {lll}
\displaystyle \frac{\p^2 f}{\p t^2}, \ \frac{\p^2 g}{\p t^2} \in \L^2(0,T;\L^{2}(\mathcal{F})), & \quad &
\displaystyle f, \ g \in \L^{\infty}(0,T;\L^{\infty}(\mathcal{F})), \\
 & & \\
\displaystyle \frac{\p f}{\p t} \in \L^2(0,T;\L^{\infty}(\mathcal{F})), & \quad &
\displaystyle \frac{\p g}{\p t} \in \L^{\infty}(0,T;\L^{2}(\mathcal{F})),
\end{array}
\end{eqnarray*}
because of the embedding $\mathbf{H}^{2}(\mathcal{F}) \hookrightarrow \mathbf{L}^{\infty}(\mathcal{F})$, and so we get
\begin{eqnarray*}
\frac{\p^2 (fg)}{\p t^2} & \in & \L^{2}(0,T;\L^2(\mathcal{F}))
\end{eqnarray*}
and the desired regularity. This shows in particular that $\H^1(0,T;\H^{m-1}(\mathcal{F})) \cap \H^{2}(0,T;\L^2(\mathcal{F}))$ is an algebra, and we can show the estimate \eqref{ch3_estcof} by noticing that the cofactor matrix is made of quadratic terms (in dimension 3), and thus we can estimate
\begin{eqnarray*}
& & \| \com \nabla \tilde{X} - \I_{\R^3} \|_{\H^{1}(\mathbf{H}^{2})\cap \H^{2}(\mathbf{L}^2)} \\
& & \leq  \tilde{C}\| \nabla \tilde{X} - \I_{\R^3} \|_{\H^{1}(\mathbf{H}^{2})\cap \H^{2}(\mathbf{L}^2)}
\left(\| \nabla \tilde{X} \|_{\H^{1}(\mathbf{H}^{2})\cap \H^{2}(\mathbf{L}^2)} + 1 \right), \\
& & \leq  C\| \nabla \tilde{X} - \I_{\R^3} \|_{\H^{1}(\mathbf{H}^{2})\cap \H^{2}(\mathbf{L}^2)}
\left(\| \nabla \tilde{X}-\I_{\R^3} \|_{\H^{1}(\mathbf{H}^{2})\cap \H^{2}(\mathbf{L}^2)} +1 \right).
\end{eqnarray*}
The arguments for proving \eqref{ch3_estcof12} is the same.
\end{proof}

\subsection{Existence of a change of variables}
Let be $T_0 \geq T > 0$. Let $h\in \H^2(0,T_0;\R^3)$ be a vector and $\mathbf{R} \in \H^2(0,T_0;\R^9)$ a rotation which provides the angular velocity $\omega \in \H^1(0,T_0;\R^3)$ given by
\begin{eqnarray*}
\mathbb{S}\left( \omega\right) & = & \displaystyle \frac{\d \mathbf{R}}{\d t} \mathbf{R}^T,
\quad \text{with }
\mathbb{S}(\omega) = \left(
\begin{matrix}
0 & -\omega_3 & \omega_2 \\
\omega_3 & 0 & -\omega_1 \\
-\omega_2 & \omega_1 & 0
\end{matrix} \right).
\end{eqnarray*}
We assume that $h_0 = 0,  \ \mathbf{R}(0) = \I_{\R^3}$ and we still use the notation
\begin{eqnarray*}
\tilde{h}'(t)  =  \mathbf{R}(t)^T h(t), & \quad & \tilde{\omega}(t) =  \mathbf{R}(t)^T \omega(t).
\end{eqnarray*}
Let us remind and prove Lemma \ref{ch3_lemmaxtension}:

\begin{lemma} \label{lemmasuper}
Let $X^{\ast}$ be a mapping lying in $\mathcal{W}_0(0,\infty;\mathcal{S})$ and satisfying for all $t\geq 0$ the equality
\begin{eqnarray}
\int_{\p \mathcal{S}} \frac{\p X^{\ast}}{\p t} \cdot \left(\com \nabla X^{\ast} \right)n \d \Gamma &  = & 0. \label{eqvolume}
\end{eqnarray}
Then for $T >0$ small enough, there exists a mapping $\tilde{X} \in \mathcal{W}(Q_T^0) $ satisfying
\textcolor{black}{
\begin{eqnarray}
\left\{ \begin{array} {lcl}
\det \nabla \tilde{X}  = 1 & \quad & \text{in } \mathcal{F} \times (0,T), \\
\tilde{X} = X^{\ast} & \quad & \text{on } \p \mathcal{S} \times (0,T), \\
\tilde{X} = \mathbf{R}^T(\Id-h) & \quad & \text{on } \p \mathcal{O} \times (0,T),\\
\text{\textcolor{black}{$\tilde{X}(\cdot,0) = \Id_{\mathcal{F}}$}} & &
\end{array} \right. \label{superpb}
\end{eqnarray}
}
and the estimate
\begin{eqnarray}
\| \tilde{X} - \Id_{\mathcal{F}} \|_{\mathcal{W}(Q_T^0)} & \leq & C \left(1+
\left\| X^{\ast} - \Id_{\mathcal{S}} \right\|_{\mathcal{W}(S_{T_0}^0)}
+ \|\tilde{h}'\|_{\H^1(0,T_0;\R^3)} + \|\tilde{\omega}\|_{\H^1(0,T_0;\R^3)} \right), \nonumber \label{estlip0} \\
\end{eqnarray}
for some independent positive constant $C$ - which in particular does not depend on $T$. Besides, if $\tilde{X}_1$ and $\tilde{X}_2$ are the solutions - for $T$ small enough - of problem \eqref{superpb} corresponding to the data $(X^{\ast},h_1,\mathbf{R}_1)$ and $(X^{\ast},h_2,\mathbf{R}_2)$ respectively, with
\begin{eqnarray*}
h_1(0) = h_2(0) = 0, \quad \mathbf{R}_1(0) = \mathbf{R}_2(0) = \I_{\R^3}, \quad h'_1(0) = h'_2(0), \quad \omega_1(0) = \omega_2(0),
\end{eqnarray*}
then the difference $\tilde{X}_2 - \tilde{X}_1$ satisfies
\begin{eqnarray}
\| \tilde{X}_2 - \tilde{X}_1 \|_{\mathcal{W}(Q_T^0)} & \leq & \tilde{C}\left(
\|\tilde{h}'_2 -\tilde{h}'_1 \|_{\H^1(0,T_0;\R^3)} + \|\tilde{\omega}_2 - \tilde{\omega}_1 \|_{\H^1(0,T_0;\R^3)} \right), \label{estlip12}
\end{eqnarray}
where the constant $\tilde{C}$ does not depend on $T$.
\end{lemma}


\begin{proof}
Given the initial data $X^{\ast}(\cdot,0) = \Id_{\mathcal{S}}$, $h_0 = 0$, $\mathbf{R}(0) = \I_{\R^3}$, $h'(0) = h_1$, $\omega(0) = \omega_0$ \textcolor{black}{and the initial condition $\tilde{X}(\cdot,0) = \Id_{\mathcal{F}}$}, let us consider the system \eqref{superpb} derived in time, as
\begin{eqnarray*}
\left\{ \begin{array} {lcl}
\displaystyle \left(\com \nabla \tilde{X} \right) : \frac{\p \nabla \tilde{X} }{\p t} = 0 & \quad & \text{in } \mathcal{F} \times (0,T), \\
\displaystyle \frac{\p \tilde{X}}{\p t} = \frac{\p X^{\ast}}{\p t} & \quad & \text{on } \p \mathcal{S} \times (0,T)
, \\
\displaystyle \frac{\p \tilde{X}}{\p t}(y,t) = -\tilde{h}'(t)- \tilde{\omega}(t)\wedge \text{\textcolor{black}{$\mathbf{R}^T(t)(y-h(t))$}}
& \quad & (y,t) \in \p \mathcal{O} \times (0,T).
\end{array} \right.
\end{eqnarray*}
This system can be seen as a modified nonlinear divergence problem, that we state as
\begin{eqnarray*}
\left\{ \begin{array} {lcl}
\displaystyle \div \ \frac{\p \tilde{X}}{\p t} = f(\tilde{X})  & \quad & \text{in } \mathcal{F} \times (0,T), \\
\displaystyle \frac{\p \tilde{X}}{\p t} = \frac{\p X^{\ast}}{\p t} & \quad & \text{on } \p \mathcal{S} \times (0,T), \\
\displaystyle \frac{\p \tilde{X}}{\p t}(y,t) = -\tilde{h}'(t)- \tilde{\omega}(t)\wedge \text{\textcolor{black}{$\mathbf{R}^T(t)(y-h(t))$}}
 & \quad & (y,t) \in \p \mathcal{O} \times (0,T),
\end{array} \right.
\end{eqnarray*}
with
\begin{eqnarray*}
f(\tilde{X}) & = & \left(\I_{\R^3} - \com \nabla \tilde{X} \right) : \frac{\p \nabla \tilde{X} }{\p t}.
\end{eqnarray*}
\textcolor{black}{Let us notice that if we assume in addition that the condition below is satisfied
\begin{eqnarray*}
\int_{\p \mathcal{S}}  \frac{\p \tilde{X}}{\p t}\cdot \left(\com  \nabla \tilde{X} \right)n \d \Gamma &  = & 0,
\end{eqnarray*}
then from the Piola identity we have
\begin{eqnarray*}
\int_{\mathcal{F}}f(\tilde{X}) = \int_{\mathcal{F}} \div\left(
\left(\I_{\R^3} - \com \nabla \tilde{X}^T\right)\frac{\p \tilde{X}}{\p t}
\right)
 = \int_{\p \mathcal{S}} \frac{\p \tilde{X}}{\p t}\cdot n \d \Gamma = 
 \int_{\p \mathcal{S}} \frac{\p X^{\ast}}{\p t}\cdot n \d \Gamma
\end{eqnarray*}
and thus the compatibility condition for this divergence system is satisfied (the contribution on $\p \mathcal{O}$ vanishes automatically).}
A solution of this system can be seen as a fixed point of the mapping
\begin{eqnarray}
\begin{array} {cccc}
\mathfrak{T} : & \text{\textcolor{black}{$\mathfrak{W}_T$}} & \rightarrow & \text{\textcolor{black}{$\mathfrak{W}_T$}} \\
& \tilde{X}_1 & \mapsto & \tilde{X}_2,
\end{array}
\end{eqnarray}
\textcolor{black}{where\footnote{\textcolor{black}{The set $\mathfrak{W}_T$ is non-trivial; Indeed, it contains extensions of $X^{\ast}$ obtained by the use of {\it plateau fonctions} (see \cite{SMSTT} for instance). The difficulty here is to consider the condition on $\p \mathcal{O}$.}}
\begin{eqnarray*}
\text{\textcolor{black}{$\mathfrak{W}_T$}} & = & \text{\textcolor{black}{
$\displaystyle \left\{\tilde{X} \in \mathcal{W}(Q_T^0) \mid  \tilde{X} = X^{\ast}
\text{ on } \p \mathcal{S}, \ 
\int_{\p \mathcal{S}}  \frac{\p \tilde{X}}{\p t}\cdot \left(\com  \nabla \tilde{X} \right)n \d \Gamma  = 0, \ \tilde{X}(\cdot,0)=\Id_{\mathcal{F}}
\right\}$
}}
\end{eqnarray*}
}
and where $\tilde{X}_2$ is a solution of the classical divergence problem
\begin{eqnarray*}
\left\{ \begin{array} {lcl}
\displaystyle \div \ \frac{\p \tilde{X}_2}{\p t} = f(\tilde{X}_1)  & \quad & \text{in } \mathcal{F} \times (0,T), \\
\displaystyle \frac{\p \tilde{X}_2}{\p t} = \frac{\p X^{\ast}}{\p t} & \quad & \text{on } \p \mathcal{S} \times (0,T), \\
\displaystyle \frac{\p \tilde{X}_2}{\p t} = -\tilde{h}'- \tilde{\omega}\wedge
\text{\textcolor{black}{$\mathbf{R}^T(t)(y-h(t))$}}
& \quad & \text{on } \p \mathcal{O} \times (0,T), \\
\text{\textcolor{black}{$\displaystyle \tilde{X}_2(\cdot,0) = \Id_{\mathcal{F}}$.}} & &
\end{array} \right.
\end{eqnarray*}
A solution of this problem can be obtained by using some results of \cite{Galdi1} for instance: The nonhomogeneous Dirichlet condition can be lifted (see Theorem 3.4, Chapter II) and the resolution made by using Exercise  3.4 and Theorem 3.2 of Chapter III. Then the solution chosen is the one which satisfies the estimates
\begin{eqnarray}
\left\| \frac{\p \tilde{X}_2}{\p t} \right\|_{\L^2(\mathbf{H}^3(\mathcal{F}))} & \leq &
C \left( \| f(\tilde{X}_1) \|_{\L^2(\mathbf{H}^{2}(\mathcal{F}))}
 + \left\| \frac{\p X^{\ast}}{\p t} \right\|_{\L^2(\mathbf{H}^{5/2}(\p \mathcal{S}))} \right. \nonumber \\
& &  \left. + \| \tilde{h}' \|_{\L^2(0,T;\R^3)}
+ \| \tilde{\omega} \|_{\L^2(0,T;\R^3)}
\text{\textcolor{black}{$(1+\| h \|_{\L^{\infty}(0,T;\R^3)})$}} \right) \nonumber \\
&  \leq & \tilde{C} \left( \| f(\tilde{X}_1) \|_{\L^2(\mathbf{H}^{2}(\mathcal{F}))}
 + \left\| \frac{\p X^{\ast}}{\p t} \right\|_{\L^2(\mathbf{H}^{3}(\mathcal{S}))} \right. \nonumber \\
& &  \left. + \| \tilde{h}' \|_{\L^2(0,T;\R^3)}  + \| \tilde{\omega} \|_{\L^2(0,T;\R^3)}
\text{\textcolor{black}{$(1+\sqrt{T}\| \tilde{h}' \|_{\L^{2}(0,T;\R^3)})$}}  \right) \nonumber \\ \label{ch3_estdiver}
\end{eqnarray}
and
\begin{eqnarray}
\text{\textcolor{black}{$\displaystyle \frac{\p^2 \tilde{X}_2}{\p t^2}_{|\p \mathcal{O}}$  }}& = & \text{\textcolor{black}{$\displaystyle -\tilde{h}'' - \tilde{\omega}' \wedge \mathbf{R}^T(\Id - h) - \tilde{\omega} \wedge 
\frac{\p \tilde{X}_2}{\p t}$,}} \nonumber \\
\left\| \frac{\p^2 \tilde{X}_2}{\p t^2} \right\|_{\L^2(\mathbf{H}^1(\mathcal{F}))} & \leq &
C \left( \| f(\tilde{X}_1) \|_{\H^1(\mathbf{L}^{2}(\mathcal{F}))}
 + \left\| \frac{\p X^{\ast}}{\p t} \right\|_{\H^1(\mathbf{H}^{1/2}(\p \mathcal{S}))} \right. \nonumber \\
& &   + \| \tilde{h}' \|_{\H^1(0,T;\R^3)}
+ \| \tilde{\omega} \|_{\H^1(0,T;\R^3)}
\text{\textcolor{black}{$(1+\| h \|_{\L^{\infty}(0,T;\R^3)})$}} \nonumber \\
& &  \text{\textcolor{black}{$ \displaystyle \left.+ \| \tilde{\omega} \|_{\L^{\infty}(0,T;\R^3)}\left\| \frac{\p \tilde{X}_2}{\p t} \right\|_{\L^2(\mathbf{H}^{1/2}(\p \mathcal{O}))}\right)$}} \nonumber \\
&  \leq & \tilde{C} \left( \| f(\tilde{X}_1) \|_{\H^1(\mathbf{L}^{2}(\mathcal{F}))}
 + \left\| \frac{\p X^{\ast}}{\p t} \right\|_{\H^1(\mathbf{H}^{1}(\mathcal{S}))}
  \right. \nonumber \\
& &  \left.  + \| \tilde{h}' \|_{\H^1(0,T;\R^3)} + \| \tilde{\omega} \|_{\H^1(0,T;\R^3)} \text{\textcolor{black}{$(1+\sqrt{T}\| \tilde{h}' \|_{\L^{2}(0,T;\R^3)})$}} \right. \nonumber \\
& & \left. \text{\textcolor{black}{$ \displaystyle
+ \left( |\omega_0 |_{\R^3} + \| \tilde{\omega} \|_{\H^1(0,T;\R^3)} \right)
\left\| \frac{\p \tilde{X}_2}{\p t} \right\|_{\L^2(\mathbf{H}^{3}(\mathcal{F}))} $}}
\right).  \label{ch3_estdiver2}
\end{eqnarray}
Let us verify that for $\tilde{X} \in \mathcal{W}(Q_T^0)$, satisfying $\tilde{X}(\cdot,0) = \Id_{\mathcal{F}}$, we have $f(\tilde{X}) \in \L^2(0,T;\mathbf{H}^{2}(\mathcal{F})) \cap \H^{1}(0,T;\mathbf{L}^2(\mathcal{F}))$. For that, we use the previous lemma by reminding that $\com \nabla \tilde{X} \in \H^1(0,T;\mathbf{H}^{2}(\mathcal{F})) \cap \H^{2}(0,T;\mathbf{L}^2(\mathcal{F}))$, and we first use the result of Lemma \ref{ch3_lemmaGrubb} with $s = 2$ and $\mu = \kappa = 0$ to get
\begin{eqnarray*}
\| f(\tilde{X}) \|_{\L^2(\mathbf{H}^{2}(\mathcal{F}))} & \leq & C \| \I_{\R^3} - \com \nabla \tilde{X}  \|_{\L^{\infty}(\mathbf{H}^{2}(\mathcal{F}))}  \left\|\frac{\p \nabla \tilde{X} }{\p t}\right\|_{\L^2(\mathbf{H}^{2}(\mathcal{F}))}  \nonumber \\
& \leq & C \sqrt{T} \| \I_{\R^3} - \com \nabla \tilde{X} \|_{\H^1(\mathbf{H}^{2}(\mathcal{F}))}\left\|\frac{\p \tilde{X} }{\p t}\right\|_{\L^2(\mathbf{H}^{3}(\mathcal{F}))}. \label{ch3_est001}
\end{eqnarray*}
For the regularity in $\H^{1}(0,T;\mathbf{L}^2(\mathcal{F}))$, let us first notice that we have by interpolation
\begin{eqnarray*}
\L^2(0,T;\mathbf{H}^{2}(\mathcal{F})) \cap \H^1(0,T;\mathbf{L}^{2}(\mathcal{F})) \hookrightarrow \L^{\infty}(0,T;\mathbf{H}^{m/2-1/2}(\mathcal{F})).
\end{eqnarray*}
Then we use Lemma \ref{ch3_lemmaGrubb} with $s=0$ and $\mu = \kappa = 1$, and the continuous embedding $\mathbf{H}^{2}(\mathcal{F}) \hookrightarrow \mathbf{L}^{\infty}(\mathcal{F})$ in order to get
\begin{eqnarray*}
& & \left\| \frac{\p (f(\tilde{X}))}{\p t} \right\|_{\L^2(\mathbf{L}^2)} \\
& \leq & C \left( \left\| \frac{\p \com \nabla \tilde{X}}{\p t} \right\|_{\L^{\infty}(\mathbf{H}^{1})}
\left\| \frac{\p \nabla \tilde{X}}{\p t} \right\|_{\L^{2}(\mathbf{H}^{1})}  +   \| \I_{\R^3} - \com \nabla \tilde{X} \|_{\L^{\infty}(\mathbf{H}^{2})}
\left\| \frac{\p^2 \nabla \tilde{X}}{\p t^2} \right\|_{\L^{2}(\mathbf{L}^{2})}\right) \\
& \leq & C\sqrt{T} \left( \left\| \frac{\p \com \nabla \tilde{X}}{\p t} \right\|_{\L^{\infty}(\mathbf{H}^{1})}
\left\| \frac{\p \nabla \tilde{X}}{\p t} \right\|_{\L^{\infty}(\mathbf{H}^{1})}  +   \| \I_{\R^3} - \com \nabla \tilde{X} \|_{\H^{1}(\mathbf{H}^{2})}
\left\| \frac{\p^2 \nabla \tilde{X}}{\p t^2} \right\|_{\L^{2}(\mathbf{L}^{2})}\right).
\end{eqnarray*}
Thus, we finally have
\begin{eqnarray}
\|f(\tilde{X})\|_{\L^2(\mathbf{H}^{2}) \cap \H^{1}(\mathbf{L}^2)} & \leq &
C\sqrt{T}\|\I_{\R^3} - \com \nabla \tilde{X} \|_{\H^{1}(\mathbf{H}^{2})\cap \H^2(\mathbf{L}^2(\mathcal{F}))}
\left\| \frac{\p \tilde{X}}{\p t} \right\|_{\mathcal{H}(Q_T^0)}.
\nonumber \\ \label{ch3_est002}
\end{eqnarray}
The estimates \eqref{ch3_estdiver} and \eqref{ch3_estdiver2} combined with \eqref{ch3_est002} and \eqref{ch3_estcof} show that the mapping $\mathfrak{T}$ is well-defined.\\
Moreover, the set defined for some $R>0$ by
\begin{eqnarray*}
\mathfrak{B}_{R,T} & = & \left\{\tilde{X} \in \mathfrak{W}_T \mid 
 \ \left\| \frac{\p \tilde{X}}{\p t} \right\|_{\mathcal{H}(Q_T^0)} \leq R \right\},
\end{eqnarray*}
is stable by $\mathfrak{T}$, for $T$ small enough and $R$ large enough. Notice that $\mathfrak{B}_{R, T}$ is a closed subset of $\mathcal{W}(Q_T^0)$. Let us verify that $\mathfrak{T}$ is a contraction in $\mathfrak{B}_{R, T}$.\\
For $\tilde{X}_1$ and $\tilde{X}_2$ in $\mathfrak{B}_{R, T}$, the difference $\tilde{Z} = \mathfrak{T}(\tilde{X}_2) - \mathfrak{T}(\tilde{X}_1)$ satisfies the divergence system
\begin{eqnarray*}
\left\{ \begin{array} {lcl}
\displaystyle \div \ \frac{\p \tilde{Z}}{\p t} = f(\tilde{X}_2) - f(\tilde{X}_1)  & \quad & \text{in } \mathcal{F} \times (0,T), \\
\displaystyle \frac{\p \tilde{Z}}{\p t} = 0 & \quad & \text{on } \p \mathcal{S} \times (0,T), \\
\displaystyle \frac{\p \tilde{Z}}{\p t} = 0 & \quad & \text{on } \p \mathcal{O} \times (0,T),
\end{array} \right.
\end{eqnarray*}
and thus the estimate
\begin{eqnarray*}
\left\| \frac{\p \tilde{Z}}{\p t} \right\|_{\L^2(\mathbf{H}^{m}(\mathcal{F})) \cap \H^{1}(\mathbf{H}^1(\mathcal{F}))} & \leq & C \| f(\tilde{X}_2) - f(\tilde{X}_1) \|_{\L^2(\mathbf{H}^{2}(\mathcal{F})) \cap \H^{1}(\mathbf{L}^2(\mathcal{F}))}.
\end{eqnarray*}
We write
\begin{eqnarray*}
f(\tilde{X}_2) - f(\tilde{X}_1) & = & \left(\com \nabla \tilde{X}_2 - \com \nabla \tilde{X}_1 \right) : \frac{\p \nabla \tilde{X}_2}{\p t}  + \left(\I_{\R^3} - \com \nabla \tilde{X}_1 \right) : \frac{\p \nabla (\tilde{X}_2 - \tilde{X}_1)}{\p t},
\end{eqnarray*}
By reconsidering the steps of the proofs of the estimate \eqref{ch3_est002}, and by using \eqref{ch3_estcof12}, we can verify that for $T$ small enough the mapping $\mathfrak{T}$ is a contraction in $\mathfrak{B}_{R, T}$. Thus $\mathfrak{T}$ admits a unique fixed point in $\mathfrak{B}_{R, T}$.\\
For the estimate \eqref{ch3_contSF21}, let us just notice that the difference $\tilde{X}$ of two mappings $\tilde{X}_1$ and $\tilde{X}_2$ of $\mathfrak{B}_R$ - corresponding to the data $(X^{\ast},h_1,\mathbf{R}_1)$ and $(X^{\ast},h_2,\mathbf{R}_2)$ respectively - satisfies the system
\begin{eqnarray*}
\left\{ \begin{array} {lcl}
\displaystyle \div \ \frac{\p \tilde{X}}{\p t} = f(\tilde{X}_2) - f(\tilde{X}_1)  & \quad & \text{in } \mathcal{F} \times (0,T), \\
\displaystyle \frac{\p \tilde{X}}{\p t} = 0 & \quad & \text{on } \p \mathcal{S} \times (0,T), \\
\displaystyle \frac{\p \tilde{X}}{\p t} = -(\tilde{h}'_2 - \tilde{h}'_1) - (\tilde{\omega}_2 - \tilde{\omega}_1)\wedge \tilde{X}_1
-\tilde{\omega}_2 \wedge \tilde{X}  & \quad & \text{on } \p \mathcal{O} \times (0,T).
\end{array} \right.
\end{eqnarray*}
Then we proceed as previously, and the end of the proof for the announced estimate is left to the reader.
\end{proof}

\subsection{Lipschitz estimates}

\begin{lemma} \label{ch3_lemmeinverse}
Let be $T_0$, and $T>0$ small enough to define $\tilde{X} \in \mathcal{W}_m(Q_T^0)$ solution of problem \eqref{superpb}, for $X^{\ast} \in \mathcal{W}_m(S_{\infty}^0)$, $h\in \H^2(0,T_0;\R^3)$ and $\mathbf{R} \in \H^2(0,T_0;\R^9)$. Let us denote by $\tilde{Y}(\cdot,t)$ the inverse of the mapping $\tilde{X}(\cdot,t)$ - for all $t \in [0,T)$. Then we have
\begin{eqnarray}
& & \| \nabla \tilde{Y}(\tilde{X}) - \I_{\R^3} \|_{\H^{1}(0,T;\mathbf{H}^{m-1}(\mathcal{F})) \cap \H^{2}(0,T;\mathbf{L}^2(\mathcal{F}))} \nonumber \\
& & \leq  C \left(1+ \| X^{\ast} - \Id_{\mathcal{S}} \|_{\mathcal{W}_m(S_{T_0}^0)} + \| \tilde{h}' \|_{\H^1(0,T_0;\R^3)}
+ \| \tilde{\omega} \|_{\H^1(0,T_0;\R^9)} \right). \label{ch3_eqtrick}
\end{eqnarray}
Let $\tilde{X}_1, \ \tilde{X}_2 \in \mathcal{W}_m(Q_T^0)$ be the solutions of problem \eqref{superpb} - for $T$ small enough - with $(X^{\ast},h_1,\mathbf{R}_1)$ and $(X^{\ast},h_2,\mathbf{R}_2)$ as data respectively. Then, if we denote by $\tilde{Y}_1(\cdot,t)$ and $\tilde{Y}_2(\cdot,t)$ the inverses of $\tilde{X}_1(\cdot,t)$ and $\tilde{X}_2(\cdot,t)$ respectively, we have
\begin{eqnarray}
& & \| \nabla \tilde{Y}_2(\tilde{X}_2) - \nabla \tilde{Y}_1(\tilde{X}_1) \|_{\H^{1}(0,T;\mathbf{H}^{m-1}(\mathcal{F})) \cap \H^{2}(0,T;\mathbf{L}^2(\mathcal{F}))} \nonumber \\
& & \leq C
\left(\| \tilde{h}'_2-\tilde{h}'_1 \|_{\H^1(0,T_0;\R^3)} + \| \tilde{\omega}_2-\tilde{\omega}_1 \|_{\H^1(0,T_0;\R^9)} \right). \label{ch3_eqtrickbis}
\end{eqnarray}
\end{lemma}

\begin{proof}
By considering the first equality of Problem \eqref{superpb}, we have the equality
\begin{eqnarray*}
\nabla \tilde{Y}(\tilde{X}(\cdot,t),t) = \frac{\com \nabla \tilde{X}(\cdot,t)^T}{\det \nabla \tilde{X}(\cdot,t)} = \com \nabla \tilde{X}(\cdot,t)^T,
\end{eqnarray*}
and so Lemma \ref{ch3_lemmecomatrice} combined with the estimates of Lemma \ref{lemmasuper} can be applied.
\end{proof}

\section{Appendix B : Proof of Lemma \ref{Xwstar}}
Let us use a result given in the Appendix of \cite{Bourguignon} (Lemma A.4), which treats of regularity in Sobolev spaces for composition of functions: There exists a positive constant $C$ such that for all $t\in (0,T)$ we have
\begin{eqnarray*}
\| w^{\ast}(\cdot,t) \|_{\mathbf{H}^3(\mathcal{S}^{\ast}(t))} & \leq &  C
\left\| \frac{\p X^{\ast}}{\p t}(\cdot,t) \right\|_{\mathbf{H}^3(\mathcal{S})}
\frac{\|Y^{\ast}(\cdot,t) \|_{\mathbf{H}^{3}(\mathcal{S}^{\ast}(t))}^3 + 1 }{\displaystyle \inf_{x^{\ast} \in \mathcal{S}^{\ast}(t)}
| \det \nabla Y^{\ast}(x^{\ast},t) |^{1/2}} \\
&  \leq &  C
\left\| \frac{\p X^{\ast}}{\p t}(\cdot,t) \right\|_{\mathbf{H}^3(\mathcal{S})}
\left(\|Y^{\ast}(\cdot,t) \|_{\mathbf{H}^{3}(\mathcal{S}^{\ast}(t))}^3 + 1 \right)
\| \det \nabla X^{\ast}(\cdot,t) \|^{1/2}_{\mathbf{L}^{\infty}(\mathcal{S})}.
\end{eqnarray*}
Let us notice that by using the change of variables induced by $X^{\ast}(\cdot,t)$, we have
\begin{eqnarray*}
\| Y^{\ast}(\cdot,t) \|^2_{\mathbf{L}^2(\mathcal{S}^{\ast}(t))}
& = & \int_{\mathcal{S}^{\ast}(t)}  \left|Y^{\ast}(x^{\ast},t) \right|_{\R^3}^2\d x^{\ast} \\
& = & \int_{\mathcal{S}} |y|^2\det \nabla X^{\ast}(y,t) \d y , \\
\| \nabla Y^{\ast}(\cdot,t) \|_{\mathbf{L}^2(\mathcal{S}^{\ast}(t))} & \leq &
\| \nabla Y^{\ast}(X^{\ast}(y,t),t) \|_{\mathbf{L}^2(\mathcal{S})} \| \det \nabla X^{\ast}(\cdot,t) \|^{1/2}_{\L^{\infty}(\mathcal{S})}.
\end{eqnarray*}
The following equality
\begin{eqnarray}
\nabla^2 Y^{\ast}(X^{\ast}(\cdot,t),t)
& = & \left( \nabla \left( \nabla Y^{\ast}(X^{\ast}(\cdot,t),t) \right)\right) \nabla Y^{\ast}(X^{\ast}(\cdot,t),t) \label{ch3_eq123}
\end{eqnarray}
yields
\begin{eqnarray*}
\| \nabla^2 Y^{\ast}(\cdot,t) \|_{\mathbf{L}^2(\mathcal{S}^{\ast}(t))} & \leq & C\| \nabla Y^{\ast}(X^{\ast}(\cdot,t),t) \|_{\mathbf{H}^1(\mathcal{S})}
\| \nabla Y^{\ast}(X^{\ast}(\cdot,t),t) \|_{\mathbf{L}^{\infty}(\mathcal{S})} \| \det \nabla X^{\ast}(\cdot,t) \|^{1/2}_{\L^{\infty}(\mathcal{S})}.
\end{eqnarray*}
Moreover, by applying Lemma \ref{ch3_lemmaGrubb} with $s=1$, $\mu = 0$ and $\kappa  =1$, the equality \eqref{ch3_eq123} implies
\begin{eqnarray*}
\| \nabla^2 Y^{\ast}(X^{\ast}(\cdot,t),t)  \|_{\mathbf{H}^1(\mathcal{S})} & \leq &
C\| \nabla Y^{\ast}(X^{\ast}(\cdot,t),t) \|^2_{\mathbf{H}^2(\mathcal{S})}.
\end{eqnarray*}
The following equality
\begin{eqnarray*}
\nabla^3 Y^{\ast}(X^{\ast}(\cdot,t),t)
& = & \left( \nabla^2 \left( \nabla Y^{\ast}(X^{\ast}(\cdot,t),t) \right)\right) \left(\nabla Y^{\ast}(X^{\ast}(\cdot,t),t)\right)^2 \\
& & + \left( \nabla \left( \nabla Y^{\ast}(X^{\ast}(\cdot,t),t) \right)\right) \left(\nabla^2 Y^{\ast}(X^{\ast}(\cdot,t),t)\right)
\end{eqnarray*}
combined with the previous estimate enables us to obtain
\begin{eqnarray*}
\| \nabla^3 Y^{\ast}(\cdot,t) \|_{\mathbf{L}^2(\mathcal{S}^{\ast}(t))} & \leq &
C\| \nabla Y^{\ast}(X^{\ast}(\cdot,t),t) \|^3_{\mathbf{H}^2(\mathcal{S})}
 \| \det \nabla X^{\ast}(\cdot,t) \|^{1/2}_{\L^{\infty}(\mathcal{S})}.
\end{eqnarray*}
Finally we get
\begin{eqnarray*}
\| w^{\ast}\|_{\L^2(\mathbf{H}^3(\mathcal{S}^{\ast}(t)))} & \leq &
\tilde{C} \left\| \frac{\p X^{\ast}} {\p t} \right\|_{\L^2(\mathbf{H}^3(\mathcal{F}))}
\| \det \nabla X^{\ast}(\cdot,t) \|^{1/2}_{\L^{\infty}(\mathbf{L}^{\infty}(\mathcal{S}))} \times \\
& &
\left(1+ \left(\| \det \nabla X^{\ast}(\cdot,t) \|_{\L^{\infty}(\mathbf{L}^{\infty}(\mathcal{S}))}
\sum_{k=1}^3 \| \nabla Y^{\ast}(X^{\ast})\|^{2k}_{\L^{\infty}(\mathbf{H}^2(\mathcal{S}))} \right)^{3/2}  \right).
\end{eqnarray*}
For the regularity of $w^{\ast}$ in $\H^1(0,T;\mathbf{H}^1(\mathcal{S}^{\ast}(t)))$, we estimate
\begin{eqnarray*}
\left\| \frac{\p w^{\ast}}{\p t}(\cdot,t) \right\|_{\mathbf{L}^2(\mathcal{S}^{\ast}(t))}
& \leq & C\| \det \nabla X^{\ast}(\cdot,t) \|^{1/2}_{\mathbf{L}^{\infty}(\mathcal{S})}
\left\| \frac{\p w^{\ast}}{\p t}(X^{\ast}(\cdot,t),t) \right\|_{\mathbf{L}^2(\mathcal{S})},
\end{eqnarray*}
and we calculate
\begin{eqnarray*}
\frac{\p w^{\ast}}{\p t}(x^{\ast},t) & = & \frac{\p^2 X^{\ast}}{\p t^2}(Y^{\ast}(x^{\ast},t),t) + \frac{\p \nabla X^{\ast}}{\p t}(Y^{\ast}(x^{\ast},t),t) \frac{\p Y^{\ast}}{\p t}(x^{\ast},t), \quad x^{\ast} \in \mathcal{S}^{\ast}(t), \\
\frac{\p w^{\ast}}{\p t}(X^{\ast}(y,t),t) & = & \frac{\p^2 X^{\ast}}{\p t^2}(y,t)
- \frac{\p \nabla X^{\ast}}{\p t}(y,t) \nabla Y^{\ast}(X^{\ast}(y,t),t) \frac{\p X^{\ast}}{\p t}(y,t) ,
\quad y \in \mathcal{S}.
\end{eqnarray*}
Thus we have
\begin{eqnarray*}
\left\| \frac{\p w^{\ast}}{\p t} \right\|_{\L^2(\mathbf{L}^2(\mathcal{S}^{\ast}(t)))} & \leq &
C \| \det \nabla X^{\ast}(\cdot,t) \|^{1/2}_{\L^{\infty}(\mathbf{L}^{\infty}(\mathcal{S}))} \times \left(
\left\| \frac{\p^2 X^{\ast}}{\p t^2} \right\|_{\L^2(\mathbf{L}^2(\mathcal{S}))} \right. \\
& & \left. +
\left\| \frac{\p \nabla X^{\ast}}{\p t} \right\|_{\L^2(\mathbf{L}^2(\mathcal{S}))} \| \nabla Y^{\ast}(X^{\ast})\|_{\L^{\infty}(\mathbf{H}^2(\mathcal{S}))}
\left\| \frac{\p X^{\ast}}{\p t} \right\|_{\L^{\infty}(\mathbf{H}^{2}(\mathcal{S}))}
\right).
\end{eqnarray*}
Finally we control $\frac{\p w^{\ast}}{\p t}$ in $\L^2(0,T;\mathbf{H}^1(\mathcal{S}^{\ast}(t)))$ by writing
\begin{eqnarray*}
\frac{\p \nabla w^{\ast}}{\p t}(x^{\ast},t) & = & \frac{\p^2 \nabla X^{\ast}}{\p t^2}(Y^{\ast}(x^{\ast},t),t) \nabla Y^{\ast}(x^{\ast},t)
+ \frac{\p \nabla^2 X^{\ast}}{\p t}(Y^{\ast}(x^{\ast},t),t)\nabla Y^{\ast}(x^{\ast},t)\frac{\p Y^{\ast}}{\p t}(x^{\ast},t) \\
& & + \frac{\p \nabla X^{\ast}}{\p t}(Y^{\ast}(x^{\ast},t),t) \frac{\p \nabla Y^{\ast}}{\p t}(x^{\ast},t), \quad x^{\ast} \in \mathcal{S}^{\ast}(t), \\
\frac{\p \nabla w^{\ast}}{\p t}(X^{\ast}(y,t),t) & = & \frac{\p^2 \nabla X^{\ast}}{\p t^2}(y,t),t) \nabla Y^{\ast}(X^{\ast}(y,t),t)
- \frac{\p \nabla^2 X^{\ast}}{\p t}(y,t)(\nabla Y^{\ast}(X^{\ast}(y,t),t))^2\frac{\p X^{\ast}}{\p t}(y,t) \\
& & + \frac{\p \nabla X^{\ast}}{\p t}(y,t)\left( \frac{\p }{\p t}\left(\nabla Y^{\ast}(X^{\ast}(y,t),t) \right)
- \nabla (\nabla Y^{\ast}(X^{\ast}(y,t),t) ) \right), \quad y \in \mathcal{S},
\end{eqnarray*}
and by estimating
\begin{eqnarray*}
\left\| \frac{\p \nabla w^{\ast}}{\p t}\right\|_{\L^2(\mathbf{L}^2(\mathcal{S}^{\ast}(t)))} & \leq &
C \| \det \nabla X^{\ast}(\cdot,t) \|^{1/2}_{\L^{\infty}(\mathbf{L}^{\infty}(\mathcal{S}))} \times
\left( \left\| \frac{\p^2 \nabla X^{\ast}}{\p t^2} \right\|_{\L^2(\mathbf{L}^2(\mathcal{S}))}
\| \nabla Y^{\ast}(X^{\ast}) \|_{\L^{\infty}(\mathbf{H}^2(\mathcal{S}))} + \right. \\
& & \left. + \left\| \frac{\p \nabla^2 X^{\ast}}{\p t} \right\|_{\L^2(\mathbf{L}^2(\mathcal{S}))}
\| \nabla Y^{\ast}(X^{\ast})\|^{2}_{\L^{\infty}(\mathbf{H}^2(\mathcal{S}))}
\left\| \frac{\p X^{\ast}}{\p t} \right\|_{\L^{\infty}(\mathbf{H}^2(\mathcal{S}))} \right. \\
& &  \left. \left\| \frac{\p \nabla X^{\ast}}{\p t} \right\|_{\L^{\infty}(\mathbf{L}^2(\mathcal{S}))}
\| \nabla Y^{\ast}(X^{\ast})\|_{\H^{1}(\mathbf{H}^2(\mathcal{S}))}
\right).
\end{eqnarray*}


\begin{thebibliography}{10}

  \bibitem{Boulakia}
  {\sc M.~Boulakia}, {\em Existence of weak solutions for the motion of an elastic structure in an incompressible viscous fluid}, C. R. Math. Acad. Sci. Paris, 336 (2003), pp.~985-–990.

  \bibitem{BoulST}
  {\sc M.~Boulakia, E.~L. Schwindt, T.~Takahashi}, {\em Existence of Strong Solutions for the Motion of an Elastic Structure in an Incompressible Viscous Fluid}, Interfaces and Free Boundaries, 14 (2012), pp.~273--306.

  \bibitem{Bourguignon}
  {\sc J.~P. Bourguignon, H.~Brezis}, {\em Remarks on the Euler Equation}, J. Funct. Analysis, 15 (1974), pp.~341--363.

  \bibitem{Chambolle}
  {\sc A.~Chambolle, B.~Desjardins, M.~J.~Esteban and C.~Grandmont}, {\em Existence of weak solutions for the unsteady interaction of a viscous fluid with an elastic plate}, J. Math. Fluid Mech., 7 (2005), pp.~368--404.
  
  \bibitem{Ciarlet}
  {\sc P.~G. Ciarlet}, {\em Mathematical elasticity. {V}ol. {I}: Three-dimensional elasticity}, North-Holland, Amsterdam, 1988.

  \bibitem{Coutand1}
  {\sc D.~Coutand and S.~Shkoller}, {\em Motion of an elastic solid inside an incompressible viscous fluid}, Arch. Ration. Mech. Anal., 176 (2005), pp.~25-–102.

  \bibitem{Coutand2}
  {\sc D.~Coutand and S.~Shkoller}, {\em The interaction between quasilinear elastodynamics and the Navier–Stokes equations}, Arch. Ration. Mech. Anal., 179 (2006), pp.~303-–352.

  \bibitem{Cumsille}
  {\sc P.~Cumsille and T.~Takahashi}, {\em Wellposedness for the system modeling the motion of a rigid body of arbitrary form in an incompressible viscous fluid}, Czechoslovak Math. J., 58 (133) (2008), pp.~961--992.

  \bibitem{Desjardins1}
  {\sc B.~Desjardins, M.~J. Esteban, C.~Grandmont and P.~Le Tallec}, {\em Weak solutions for a fluid-elastic structure interaction model}, Rev. Math. Comput., 14 (2001), pp.~523–-538.

  \bibitem{Galdi1}
  {\sc G.~P. Galdi}, {\em An Introduction to the Mathematical Theory of the Navier-Stokes Equations}, Volume 1, Springer-Verlag, 1994.


  \bibitem{Grubb}
  {\sc G.~Grubb, V.~A. Solonnikov}, {\em Boundary value problems for nonstationary Navier-Stokes equations treated by pseudo-differential methods}, Math. Scand., 69 (1991), pp.~217--290.

  \bibitem{Gurtin}
  {\sc M.~E. Gurtin}, {\em An Introduction to Continuum Mechanics}, Academic Press Inc., Harcourt Brace Jovanovich Publishers, New York, 1981.

  \bibitem{IW}
  {\sc A.~Inoue and M.~Wakimoto}, {\em On existence of the Navier-Stokes equation in a time dependent domain,} J. Fac. Sci. Univ. Tokyo Sect. IA Math., 24 (1977), pp.~303--319.

  \bibitem{Lions}
  {\sc J.-L.~Lions and E.~Magenes}, {\em Non-homogeneous boundary value problems and applications}, Vol. I, Springer , Berlin-Heidelberg-New York, 1972.

  \bibitem{Necasova}
  {\sc \v{S}.~Ne\v{c}asov\'a, T.~Takahashi, M.~Tucsnak}, {\em Weak Solutions for the Motion of a Self-propelled Deformable Structure in a Viscous Incompressible Fluid}, Acta Appl. Math., 116 (2011), no.~3, pp.~329--352.

  \bibitem{JPR1}
  {\sc J.-P.~Raymond}, {\em Stokes and Navier-Stokes equations with nonhomogeneous boundary conditions}, Ann. I. H. Poincar\'e - AN 24 (2007), pp.~921--951.


  \bibitem{SMSTT}
  {\sc J.~San Mart\'{\i}n, J.-F.~Scheid, T.~Takahashi and M.~Tucsnak}, {\em An initial and boundary value problem modeling of fish-like swimming,} Arch. Rational Mech. Anal., 188 (2008), pp.~429--455.

  \bibitem{TT}
  {\sc T.~Takahashi}, {\em Analysis of strong solutions for the equations modeling the motion of a rigid-fluid system in a bounded domain,} Adv. Differ. Equ., 8 (2003), pp.~1499--1532.

  \bibitem{Tucsnak}
  {\sc T.~Takahashi, M.~Tucsnak}, {\em Global strong solutions for the two-dimensional motion of an infinite cylinder in a viscous fluid}, J. Math. Fluid Mech., 6 (2004), pp.~53--77.

  \bibitem{Temam}
  {\sc R.~Temam}, {\em Probl\`emes math\'ematiques en plasticit\'e}, Gauthier-Villars, 1983.

\end{thebibliography}
\end{document}